\documentclass[12pt]{article}

\usepackage{verbatim}

\usepackage{amsfonts}

\usepackage{amsthm}

\usepackage{amssymb}

\usepackage{amsmath}

\usepackage{mathabx}

\usepackage{enumerate}

\usepackage[all]{xy}

\usepackage{graphicx}

\usepackage{centernot}

\usepackage{ stmaryrd }

\usepackage[pagebackref,colorlinks,linkcolor=red,citecolor=blue,urlcolor=blue,hypertexnames=true]{hyperref}

\usepackage[pagebackref,hypertexnames=true]{hyperref}

\newcommand{\N}{\mathbb{N}}

\newcommand{\Z}{\mathbb{Z}}

\newcommand{\R}{\mathbb{R}}

\newcommand{\C}{\mathbb{C}}

\newcommand{\LL}{\mathcal{L}}

\newcommand{\K}{\mathcal{K}}

\newcommand{\LB}{\mathcal{L}_B}

\newcommand{\LSB}{\mathcal{L}_{SB}}

\newcommand{\KB}{\mathcal{K}_B}

\newcommand{\Hawaii}{Hawai\kern.05em`\kern.05em\relax i}

\theoremstyle{plain}
\newtheorem{theorem}{Theorem}[section]
\newtheorem{lemma}[theorem]{Lemma}
\newtheorem{corollary}[theorem]{Corollary}
\newtheorem{proposition}[theorem]{Proposition}
\newtheorem{conjecture}[theorem]{Conjecture}
\newtheorem{question}[theorem]{Question}
\newtheorem{definition-theorem}[theorem]{Definition / Theorem}

\newtheorem*{conjecture*}{Conjecture}
\newtheorem*{theorem*}{Theorem}

\theoremstyle{definition}
\newtheorem{definition}[theorem]{Definition}
\newtheorem{example}[theorem]{Example}
\newtheorem{examples}[theorem]{Examples}

\theoremstyle{remark}
\newtheorem{remark}[theorem]{Remark}

\newtheorem*{example*}{Example}  
\newtheorem*{remark*}{Remark}

\begin{document}

\title{The UCT for $C^*$-algebras with finite complexity}

\author{Rufus Willett and Guoliang Yu}


\maketitle

\begin{abstract}
A $C^*$-algebra satisfies the Universal Coefficient Theorem (UCT) of Rosenberg and Schochet if it is equivalent in Kasparov's $KK$-theory to a commutative $C^*$-algebra.  This paper is motivated by the problem of establishing the range of validity of the UCT, and in particular, whether the UCT holds for all nuclear $C^*$-algebras.

We introduce the idea of a $C^*$-algebra that ``decomposes'' over a class $\mathcal{C}$ of $C^*$-algebras.  Roughly, this means that locally, there are approximately central elements that approximately cut the $C^*$-algebra into two  $C^*$-subalgebras from $\mathcal{C}$ that have well-behaved intersection.   We show that if a $C^*$-algebra decomposes over the class of nuclear, UCT $C^*$-algebras, then it satisfies the UCT.  The argument is based on a Mayer-Vietoris principle in the framework of controlled $KK$-theory; the latter was introduced by the authors in earlier work.   Nuclearity is used via Kasparov's Hilbert module version of Voiculescu's theorem, and Haagerup's theorem that nuclear $C^*$-algebras are amenable.

We say that a $C^*$-algebra has finite complexity if it is in the smallest class of $C^*$-algebras containing the finite-dimensional $C^*$-algebras, and closed under decomposability; our main result implies that all $C^*$-algebras in this class satisfy the UCT.  The class of $C^*$-algebras with finite complexity is large, and comes with an ordinal-number invariant measuring the complexity level.  We conjecture that a $C^*$-algebra of finite nuclear dimension and real rank zero has finite complexity; this (and several other related conjectures) would imply the UCT for all separable nuclear $C^*$-algebras.  We also give new local formulations of the UCT, and some other necessary and sufficient conditions for the UCT to hold for all nuclear $C^*$-algebras.  
\end{abstract}

\tableofcontents

\section{Introduction}

Our aim in this paper is to present some new techniques to establish the Universal Coefficient Theorem in $C^*$-algebra $K$-theory, and some new necessary and sufficient conditions for the Universal Coefficient Theorem to hold for all nuclear $C^*$-algebras.  

Unless otherwise stated, anything in this introduction called $A$ or $B$ is a \emph{separable} $C^*$-algebra.  

\subsection{The Universal Coefficient Theorem}

A $C^*$-algebra $A$ satisfies the \emph{Universal Coefficient Theorem} (UCT) of Rosenberg and Schochet \cite{Rosenberg:1987bh} if for any $C^*$-algebra $B$, there is a canonical short exact sequence 
$$
0\to \text{Ext}(K_*(A),K_*(B))\to KK(A,B)\to \text{Hom}(K_*(A),K_*(B))\to 0.
$$
Equivalently (see \cite[page 456]{Rosenberg:1987bh} or \cite[Proposition 5.2]{Skandalis:1988rr}), $A$ satisfies the UCT if it is $KK$-equivalent to a commutative $C^*$-algebra.  

The UCT is known to hold for a large class of $C^*$-algebras.  The fundamental examples are the $C^*$-algebras in the \emph{bootstrap class} $\mathcal{N}$.  This is the smallest collection of separable, nuclear $C^*$-algebras that contains all type I $C^*$-algebras, and that is closed under the following operations: extensions; stable isomorphisms; inductive limits; and crossed products by $\R$ and $\Z$.  Rosenberg and Schochet \cite{Rosenberg:1987bh} showed that any $C^*$-algebra in $\mathcal{N}$ satisfies the UCT.  Another important class of examples was established by Tu in \cite[Proposition 10.7]{Tu:1999bq}: building on the work of Higson and Kasparov \cite{Higson:2001eb} on the Baum-Connes conjecture for a-T-menable groups, Tu showed that the groupoid\footnote{To be more precise, we need standard assumptions so that the groupoid $C^*$-algebra is defined and separable: here, appropriate assumptions are that the groupoid is locally compact, Hausdorff, and second countable, and that it admits a Haar system.} $C^*$-algebra of any a-T-menable groupoid satisfies the UCT.  In particular, Tu's work applies to the groupoid $C^*$-algebras of amenable groupoids.

There has been other significant work giving sufficient conditions for the UCT to hold, and in some cases also necessary conditions: as well as the work mentioned already, one also has for example \cite[Proposition 5.2]{Skandalis:1988rr}, \cite[Corollary 8.4.6]{Rordam:2002cs}, \cite{Dadarlat:2003tg}, \cite[Remark 2.17]{Kirchberg:2015ue}, \cite[Theorem 4.17]{Barlak:2017ts}, \cite{Barlak:2017aa}, and \cite{Barlak:2020tz}.  Nonetheless, the bootstrap class and the class of $C^*$-algebras of a-T-menable groupoids, which are defined in terms of \emph{global} properties of the $C^*$-algebras involved, remain the most important classes of $C^*$-algebras known to satisfy the UCT.

On the other hand, Skandalis \cite[page 571]{Skandalis:1988rr} has shown\footnote{See also the exposition in \cite[Sections 6.1 and 6.2]{Higson:2004la}.}  that there are $C^*$-algebras that do not satisfy the UCT.  Skandalis's examples are quite concrete: they are reduced group $C^*$-algebras of countably infinite hyperbolic groups with property (T), and in particular are exact \cite[Section 6.E]{Kirchberg:1999ss}.  Looking to more exotic examples, failures of exactness can also be used to produce non-UCT $C^*$-algebras: see for example \cite[Remark 4.3]{Chabert:2004fj}.  

Despite these counterexamples, there are no known \emph{nuclear} $C^*$-algebras that do not satisfy the UCT.  Whether or not the UCT holds for all nuclear $C^*$-algebras is a particularly important open problem.  One reason for this is the spectacular recent progress (see for example \cite{Kirchberg-ICM,Phillips-documenta,GongLinNiu-1,GongLinNiu-2,ElliottGongLinNiu,Elliott:2015fb,Tikuisis:2015kx,Carrion:2020aa}) in the Elliott program \cite{Elliott:1995dq} to classify simple, separable, nuclear $C^*$-algebras by $K$-theoretic invariants.  Establishing the range of validity of the UCT is now the only barrier to getting the `best possible' classification result in this setting.  

On the other hand, work inspired by the Elliott program has led to recent, and again spectacular, success in the general structure theory of nuclear $C^*$-algebras, including the recent solution of a large part of the Toms-Winter conjecture \cite{Castillejos:2019ab,Castillejos:2019aa}.  Our motivation in the current paper is to try to bridge the gap between properties that are relevant in this structure theory -- in particular the theory of nuclear dimension \cite{Winter:2010eb} introduced by Winter and Zacharias -- and properties that imply the UCT.  In particular, our aim is to give \emph{local} conditions that imply the UCT, in contrast to the global conditions from the work of Rosenberg and Schochet \cite{Rosenberg:1987bh} and Tu \cite{Tu:1999bq} mentioned above.


\subsection{Decompositions and the main theorem}

We now introduce our sufficient condition for the UCT.   For the statement below, if $X$ is a metric space, $S$ is a subset of $X$, $x\in X$, and $\epsilon>0$ we write ``$x\in_\epsilon S$'' if there exists $s\in S$ with $d(x,s)<\epsilon$.  

\begin{definition}\label{ais}
Let $\mathcal{C}$ be a class of unital $C^*$-algebras.  A unital $C^*$-algebra\footnote{Not necessarily separable.  For applications to the UCT, only the separable case is relevant, but the definition admits interesting examples in the non-separable case, and it seems plausible there will be other applications.} $A$ \emph{decomposes over $\mathcal{C}$} if for every finite subset $X$ of the unit ball of $A$ and every $\epsilon>0$ there exist $C^*$-subalgebras $C$, $D$, and $E$ of $A$ that are in the class $\mathcal{C}$ and contain $1_A$, and a positive contraction $h\in E$ such that:
\begin{enumerate}[(i)]
\item $\|[h,x]\|< \epsilon$ for all $x\in X$;
\item $hx\in_\epsilon C$, $(1-h)x\in_\epsilon D$, and $h(1-h)x\in_\epsilon E$ for all $x\in X$;
\item for all $e$ in the unit ball of $E$, $e\in_\epsilon C$ and $e\in_\epsilon D$.
\end{enumerate}
\end{definition}

One should think of $C$ and $D$ as being approximately (unitizations of) ideals in $A$ such that $C+D=A$, and $E$ being approximately equal to (the unitization of) $C\cap D$.  We will discuss examples later.



Here is our main theorem, which was inspired partly by our earlier work on the K\"{u}nneth formula (partly in collaboration with Oyono-Oyono) \cite{Oyono-Oyono:2016qd,Willett:2019aa}, and partly by our earlier work on finite dynamical complexity (in collaboration with Guentner) \cite{Guentner:2014bh}\footnote{This was in turn inspired by the work of Guentner, Tessera, and the second author on the stable Borel conjecture for groups with finite decomposition complexity  \cite{Guentner:2009tg}.}.   See Corollary \ref{nu cor} below for the proof.

\begin{theorem}\label{main}
If $A$ is a separable, unital $C^*$-algebra that decomposes over the class of separable, nuclear $C^*$-algebras that satisfy the UCT, then $A$ is nuclear and satisfies the UCT.  
\end{theorem}

One can thus think of decomposability as an addition to the closure operations that are used in the definition of the bootstrap class $\mathcal{N}$.

\subsection{$C^*$-algebras with finite complexity}
Following the precedent established by \cite{Guentner:2013aa} in coarse geometry, the notion of decomposability suggests a complexity hierarchy on $C^*$-algebras.

\begin{definition}\label{f c}
Let $\mathcal{D}$ denote a class of unital $C^*$-algebras.  For an ordinal number $\alpha$:
\begin{enumerate}[(i)]
\item if $\alpha=0$, let $\mathcal{D}_0$ be the class of $C^*$-algebras $D$ that are locally\footnote{A $C^*$-algebra is \emph{locally} in a class $\mathcal{D}$ if for any finite subset $X$ of $D$ and any $\epsilon>0$ there is a $C^*$-subalgebra $C$ of $D$ that is in $\mathcal{D}$, and such that $x\in_\epsilon C$ for all $x\in X$.} in $\mathcal{D}$;
\item if $\alpha>0$, let $\mathcal{D}_\alpha$ be the class of $C^*$-algebras that decompose over $C^*$-algebras in $\bigcup_{\beta<\alpha}\mathcal{D}_\beta$.
\end{enumerate}
A unital $C^*$-algebra $D$ has \emph{finite complexity relative to $\mathcal{D}$} if it is in $\mathcal{D}_\alpha$ for some $\alpha$.  If $\mathcal{D}$ is the class of finite-dimensional $C^*$-algebras, we just say that $D$ has \emph{finite complexity}.  

If a unital $C^*$-algebra $D$ has finite complexity relative to $\mathcal{D}$, the \emph{complexity rank of $D$ relative to $\mathcal{D}$} is the smallest $\alpha$ such that $D$ is in $\mathcal{D}_\alpha$.  If $\mathcal{D}$ is the class of finite-dimensional $C^*$-algebras, we just say the \emph{complexity rank of $D$} with no additional qualifiers.
\end{definition}

The following result is equivalent to Theorem \ref{main} above.  However, we think the reframing in terms of complexity is quite suggestive.

\begin{theorem}\label{complex cor}
Let $\mathcal{C}$ be a class of separable, unital, nuclear $C^*$-algebras that satisfy the UCT.  Then the class of separable, unital $C^*$-algebras that have finite complexity relative to $\mathcal{C}$ consists of nuclear $C^*$-algebras that satisfy the UCT.  

In particular, every separable $C^*$-algebra of finite complexity is nuclear and satisfies the UCT.
\end{theorem}

We can now give some non-trivial examples of $C^*$-algebras that decompose over natural, simpler, classes.

\begin{examples}\label{com ex}
\begin{enumerate}[(i)]
\item \label{cuntz ex} In Proposition \ref{ca d}, we show that for $2\leq n <\infty$, the Cuntz algebra $\mathcal{O}_n$ has complexity rank one.
\item \label{gpd ex} In \cite{Guentner:2014bh}, Guentner and the authors introduced ``finite dynamical complexity'' for groupoids, which also comes with a notion of complexity rank.  In Proposition \ref{fdc prop} we show that if $G$ is a locally compact, Hausdorff, \'{e}tale, principal, ample groupoid with compact base space, then the complexity rank of $C^*_r(G)$ is bounded above by that of $G$.  The class of groupoids with finite dynamical complexity is quite large: see Examples \ref{fdc exes} and \ref{fdc exes 2} below.
\end{enumerate}
\end{examples}

Combining part \eqref{gpd ex} above with Theorem \ref{complex cor} gives a new proof of the UCT for the groupoid $C^*$-algebras of a large class of groupoids.  However, we cannot claim any genuinely new examples: this is because the groupoids involved are all amenable, so the UCT for their $C^*$-algebras also follows from Tu's theorem \cite{Tu:1999bq} (see Remark \ref{no new uct} below for more details).



\subsection{Kirchberg algebras}

Generalizing the Cuntz algebras from \eqref{cuntz ex} above, recall that a \emph{Kirchberg algebra} is a separable, nuclear $C^*$-algebra $A$ such that for any non-zero $a\in A$, there are $b,c\in A$ such that $bac=1_A$.  Kirchberg algebras are closely connected to the UCT problem for nuclear $C^*$-algebras thanks to the following theorem of Kirchberg: see \cite[Corollary 8.4.6]{Rordam:2002cs} or \cite[Remark 2.17]{Kirchberg:2015ue}.

\begin{theorem}[Kirchberg]\label{kirchberg reduce}
To establish the UCT for all separable, nuclear $C^*$-algebras, it suffices to establish the UCT for any Kirchberg algebra with zero $K$-theory. \qed
\end{theorem}

Theorems \ref{complex cor} and \ref{kirchberg reduce} imply that if any Kirchberg algebra with zero $K$-theory has finite complexity, then the UCT holds for all separable, nuclear $C^*$-algebras.  Conversely, if the UCT holds for all separable, nuclear $C^*$-algebras, then from the Kirchberg-Phillips classification theorem \cite{Kirchberg-ICM,Phillips-documenta} (see also \cite[Corollary 8.4.2]{Rordam:2002cs} for the precise statement we want here), any unital Kirchberg algebra with zero $K$-theory will be isomorphic to the Cuntz algebra $\mathcal{O}_2$, and so will have complexity rank one by Examples \ref{com ex}, part \eqref{cuntz ex}.  We summarize this discussion in the theorem below.

\begin{theorem}\label{kirch the}
The following are equivalent:
\begin{enumerate}[(i)] 
\item Any Kirchberg algebra with zero $K$-theory has complexity rank one.
\item All separable nuclear $C^*$-algebras satisfy the UCT.  \qed
\end{enumerate}
\end{theorem}

Generalizing Example \ref{com ex}, part \eqref{cuntz ex} above Jaime and the first author show in  \cite{Jaime:2021vh} that a Kirchberg algebra \emph{that satisfies the UCT} has complexity rank one if and only if its $K_1$ group is torsion free, and that moreover any UCT Kirchberg algebra has complexity rank at most two.  From Theorem \ref{kirch the}, if one could prove this without the UCT assumption, then the UCT for all separable nuclear $C^*$-algebras would follow.

The paper \cite{Jaime:2021vh} also discusses several other connections between complexity rank, real rank zero, and nuclear dimension.  We will not go into this any more deeply here; suffice to say that these other connections inspired us to make the following conjectures.  

\begin{conjecture}\label{fcc1}
Any separable unital $C^*$-algebra with real rank zero and finite nuclear dimension has finite complexity.
\end{conjecture}

\begin{conjecture}\label{fcc2}
Any separable unital $C^*$-algebra with finite nuclear dimension has finite complexity relative to the class of subhomogeneous\footnote{Recall that a $C^*$-algebra $C$ is \emph{subhomogeneous} if there is $N\in \N$ and a compact Hausdorff space $X$ such that $C$ is a $C^*$-subalgebra of $M_N(C(X))$: see for example \cite[IV.1.4]{Blackadar:2006eq} for background.} $C^*$-algebras.
\end{conjecture}

Thanks to Theorem \ref{kirch the} and the fact that all Kirchberg algebras have nuclear dimension one (see \cite[Theorem G]{Bosa:2014zr}) and real rank zero (see \cite{Zhang:1990aa}), either of these conjectures implies the UCT for all separable, nuclear $C^*$-algebras.  There are many other related conjectures one could reasonably make that imply the UCT for all nuclear $C^*$-algebras.  About the strongest such conjecture would be that any separable, nuclear $C^*$-algebra with real rank zero has finite complexity\footnote{It would also be natural to drop the real rank zero assumption, and then only ask for finite complexity relative to the subhomogeneous $C^*$-algebras, or even just relative to the type I $C^*$-algebras.}.  One of the weakest is that any Kirchberg algebra with zero $K$-theory has finite complexity.

\subsection{A local reformulation of the UCT}

We now discuss the methods that go into the proof of Theorem \ref{main}.

In our earlier work \cite{Willett:2020aa}, we introduced \emph{controlled $KK$-theory} groups $KK_\epsilon(X,B)$ associated to a $C^*$-algebra $B$, a finite subset $X$ of a $C^*$-algebra $A$ and a constant $\epsilon>0$.  Very roughly (we give more details below), one defines these by representing $A$ in ``general position'' inside the stable multiplier algebra $M(B\otimes \mathcal{K})$ of $B$.  The group $KK_\epsilon(X,B)$ then consists of the ``part of the $K$-theory of $B$ that commutes with $X$, up to $\epsilon$''.

To  be more precise about this, assume that $A$ and $B$ are $C^*$-algebras, and assume for simplicity\footnote{The theory also works for $C^*$-algebras that are not unital, but the definitions are a little more complicated.} that $A$ is unital.  Let $\pi:A\to M(B\otimes \K)$ be a faithful, unital, and strongly unitally absorbing\footnote{Roughly, a strongly unitally absorbing representation is one that satisfies the conclusion of Voiculescu's theorem for all representations of $A$ on Hilbert $B$-modules; for the current discussion, it is just important that such a representation always exists.  See Definition \ref{usa def} below for details.} representation.  Fixing such a representation, identify $A$ with a diagonal subalgebra of $M_2(M(B\otimes \K))$ via the representation $\pi\oplus \pi$.  For a finite subset $X$ of the unit ball of $A$ and $\epsilon>0$, define $\mathcal{P}_\epsilon(X,B)$ to be the set of projections in $M_2(M(B\otimes \mathcal{K}))$ such that $p-$${\tiny\begin{pmatrix} 1 & 0 \\ 0 & 0 \end{pmatrix}}$ is in $M_2(B\otimes \mathcal{K})$, and such that $\|[p,x]\|< \epsilon$ for all $x\in X$.  The associated \emph{controlled $KK$-theory group}\footnote{It is canonically a group, with the operation given by Cuntz sum in an appropriate sense.} is then defined to be the set
$$
KK^0_{\epsilon}(X,B):=\pi_0(\mathcal{P}_\epsilon(X,B))
$$
of path components in $\mathcal{P}_\epsilon(X,B)$.  One can show that this group is determined up to canonical isomorphism by the subset inclusion $X\subseteq A$, by $B$, and by $\epsilon$: it does not depend on the choice of representation.  

Note that if $X=\varnothing$, then $KK^0_{\epsilon}(\varnothing,B)$ is canonically isomorphic to the usual $K$-theory group $K_0(B)$ (for any $\epsilon)$: this is what we mean when we say $KK_\epsilon(X,B)$ consists of the ``part of the $K$-theory of $B$ that commutes with $X$, up to $\epsilon$''.

Now, if $0<\delta\leq\epsilon$ and if $Y\supseteq X$ are finite subsets of $A_1$, then there is an inclusion $\mathcal{P}_\delta(Y,B)\subseteq \mathcal{P}_\epsilon(X,B)$ that induces a ``forget control map''
$$
KK_\delta(Y,B)\to KK_\epsilon(X,B)
$$
In \cite[Theorem 1.1]{Willett:2020aa}, we showed that there is a short exact `Milnor sequence' relating the inverse system built from these forget control maps to the usual $KK$-group $KK(A,B)$: see Theorem \ref{milnor ses} below for details.  This sequence is an analogue of the Milnor sequence appearing in Schochet's work \cite{Schochet:1996aa,Schochet:2002aa}; however, unlike Schochet's version, it is local in nature, and does not require the UCT.

Our first goal in this paper is to use the Milnor sequence to establish the following `local reformulation' of the UCT.

\begin{theorem}\label{reform intro}
Let $A$ be a unital $C^*$-algebra.  Then the following are equivalent:
\begin{enumerate}[(i)]
\item $A$ satisfies the UCT.
\item \label{intro reform 2} Let $B$ be a separable $C^*$-algebra such that $K_*(B)=0$, and let $\pi:A\to M(SB\otimes \K)$ be a strongly unitally absorbing representation into the stable multiplier algebra of the suspension of $B$.  Then for any finite subset $X$ of $A$ and any $\epsilon>0$ there exists a finite subset $Y$ of $A$ containing $X$ and $\delta\leq \epsilon$ such that the canonical forget control map
$$
KK_\delta(Y,SB)\to KK_\epsilon(X,SB)
$$
for the suspension of $B$ is zero. 
\end{enumerate}
\end{theorem}
This is a key ingredient in our main results, but we hope it will prove to be useful in its own right.  Note in particular that there are no assumptions on $A$ other than that it is separable and unital\footnote{Unitality is not really necessary - we do not do it in this paper, but similar techniques establish the result above for non-unital separable $C^*$-algebras, with appropriately reformulated controlled $KK$-groups.}.

There is a technical variation of Theorem \ref{reform intro} that applies to nuclear $C^*$-algebras, and that plays an important role in our arguments.  The key point is one of order of quantifiers: condition \eqref{intro reform 2} from Theorem \ref{reform intro} starts with quantifiers of the form 
$$
\text{``}\forall B ~\forall \pi ~\forall X~\forall \epsilon~\exists Y ~\exists \delta .... \text{''}
$$
If $A$ is nuclear, the same statement is true with the order of quantifiers replaced with 
$$
\text{``}\forall \epsilon~ \exists \delta ~ \forall B~ \forall \pi ~\forall X~\exists Y ... \text{''}
$$
i.e.\ $\delta$ depends \emph{only} on $\epsilon$ and not on any of the other choices involved.  To establish this, we adapt an averaging argument due to Christensen, Sinclair, Smith, White, and Winter \cite[Section 3]{Christensen:2012tt}, which is in turn based on Haagerup's theorem that nuclear $C^*$-algebras are amenable \cite{Haagerup:1983wg}.

\subsection{Strategy for the proof of the main theorem}\label{strat subsec}

Assume that $A$ is a nuclear, unital $C^*$-algebra that decomposes with respect to the class of nuclear UCT $C^*$-algebras as in the statement of Theorem \ref{main}.  Assume moreover that $K_*(B)=0$.  Thanks to Theorem \ref{reform intro} above, to establish the UCT for $A$ it suffices to show that for any finite subset $X$ of the unit ball $A_1$ of $A$, and any $\epsilon>0$ there exist $Y\supseteq X$ and $\delta\leq \epsilon$ such that the canonical forget control map 
$$
KK_\delta^0(Y,SB)\to KK_\epsilon^0(X,SB)
$$
is zero.  

Our approach to this is inspired directly by our earlier work with several collaborators: this includes the work on the K\"{u}nneth formula of Oyono-Oyono and the second author \cite{Oyono-Oyono:2016qd}, and separately by the first author \cite{Willett:2019aa}; the work of Guentner and the authors on the Baum-Connes conjecture for transformation groupoids with finite dynamical complexity \cite{Guentner:2014bh}; and the work of Guentner, Tessera, and the second author on the stable Borel conjecture for groups of finite decomposition complexity \cite{Guentner:2009tg}.  These other papers all use controlled $K$-theory as opposed to $KK$-theory; the seminal result along these lines is the second author's work on the Novikov conjecture for groups with finite asymptotic dimension \cite{Yu:1998wj}.

In the current context, we use decomposability and a Mayer-Vietoris argument.  Let $\gamma>0$ be a very small constant, which is in particular smaller than $\epsilon$.  Then any suitably small\footnote{The size of $\gamma$ depends linearly on $\epsilon$ and the size of $\delta$ depends linearly on $\gamma$; the constants involved are very large.} $\delta>0$ will have the following property.  Let $h$ and $C$, $D$, and $E$ be nuclear UCT algebras as in the definition of decomposability for the given set $X$ and parameter $\delta$.   Let $Y_C$, $Y_D$ and $Y_E$ be finite subsets of the unit balls $C_1$, $D_1$, and $E_1$ respectively that contain $hX\cup\{h\}$, $(1-h)X\cup\{h\}$ and $h(1-h)X\cup \{h\}$ respectively up to $\delta$-error, and so that $Y_C$ and $Y_D$ both contain $Y_E$ up to $\delta$-error.  Let $Y=Y_C\cup Y_D\cup Y_E\cup X$.  Then one can construct a diagram\footnote{The form of this diagram is not new: the basic idea is modeled on \cite[Diagram (5.8)]{Guentner:2009tg} from the work of the Guentner, Tessera, and the second author on the stable Borel conjecture for groups with finite decomposition complexity.  See also \cite[Proposition 7.6]{Guentner:2014bh} from work of the Guentner and the authors in a more closely related context.} of the form:
\begin{equation}\label{main diag}
\xymatrix{ & KK^0_\delta(Y,SB) \ar[d] \ar[r]^-{\kappa_C\oplus \kappa_D} & KK^0_{2\delta}(Y_C,SB)\oplus KK^0_{2\delta} (Y_D,SB) \\
KK_{\gamma}(Y_E,S^2B) \ar[r]^-\partial & KK^0_\epsilon(X,SB) & },
\end{equation}
where the vertical arrow is the canonical forget control map.   This diagram has the ``exactness'' property that if $[p]$ goes to zero under the map 
\begin{equation}\label{hoz map 1}
\kappa_C\oplus \kappa_D : KK^0_\delta(Y,B) \to KK^0_{2\delta}(Y_C,SB)\oplus KK^0_{2\delta} (Y_D,SB)
\end{equation}
then the image of $[p]$ under the forget control map $KK^0_\delta(Y,SB) \to KK^0_\epsilon(X,SB)$ is in the image of the map
\begin{equation}\label{hoz map 2}
\partial: KK_{\gamma}(Y_E,S^2B) \to KK^0_\epsilon(X,SB).
\end{equation}
However, as $K_*(B)=0$, if $\gamma$ and $\delta$ are small enough, one can use Theorem \ref{uct reform} (in the stronger form for nuclear $C^*$-algebras) to choose $Y_C$, $Y_D$, and $Y_E$ large enough so that the maps in lines \eqref{hoz map 1} and \eqref{hoz map 2} are zero.  This completes the proof.  

In the detailed exposition below we structure the proof to give it as `local' a flavor as possible, partly as we suspect that the ideas might be useful in other contexts.  The two main `local'(ish) technical results are recorded as Propositions \ref{uct incl} and \ref{mv vanish technical} below.

The argument above is directly inspired by the classical Mayer-Vietoris principle.  Indeed, assume that $C$ and $D$ are nuclear \emph{ideals} in $A$ with intersection $E$, and such that $A=C+ D$.  Then there is\footnote{It is not in the literature as far as we can tell.  For nuclear $C^*$-algebras, it can be derived from the usual long exact sequence in $KK$-theory using, for example, the argument of \cite[Proposition 2.7.15]{Willett:2010ay}.} an exact Mayer-Vietoris sequence 
$$
\cdots \to KK^0(E,SB) \to KK^0(A,B) \to KK^0(C,B)\oplus KK^0(D,B)\to \cdots .
$$
In particular if the groups at the left and right are zero, then the group in the middle is also zero.  Our analysis of the diagram in line \eqref{main diag} is based on a concrete construction of this classical Mayer-Vietoris sequence that can be adapted to our controlled setting.  The idea has its roots in algebraic $K$-theory, going back at least as far  as \cite[Chapter 2]{Milnor:1971kl}.   Having said this, there is significant work to be done adapting these classical ideas to the analytic superstructure that we built in \cite{Willett:2020aa}, and the resulting formulas and arguments end up being quite different.

\begin{remark}\label{intro nuc rem}
It would be very interesting to remove the nuclearity hypothesis from Theorem \ref{main}, or at least to replace it with something weaker such as exactness.  Let us explain how nuclearity is used in the proof of Theorem \ref{main}, in the hope that some reader will see a way around it.  

The first use of nuclearity is to show that any nuclear, unital $C^*$-algebra admits strongly unitally absorbing representations whose restriction to any nuclear, unital $C^*$-subalgebra is also strongly unitally absorbing: see Corollary \ref{sub sua} below.  The proof of this is based on Kasparov's version of Voiculescu's theorem for Hilbert modules \cite[Section 7]{Kasparov:1980sp}.  It seems plausible from the discussion in Remark \ref{nuc rem} below that some form of nuclearity is necessary for this to hold, but we do not know this.  

The second place nuclearity is used is via an averaging argument due to Christensen, Sinclair, Smith, White, and Winter \cite[Section 3]{Christensen:2012tt}; this is applicable to nuclear $C^*$-algebras thanks to Haagerup's theorem that nuclear $C^*$-algebras are always amenable \cite{Haagerup:1983wg}.  This lets us prove a stronger version of Theorem \ref{reform intro}: see Corollary \ref{uct reform 2} below.  We do not know if this result holds without nuclearity: see Remark \ref{con rem} for a more detailed discussion.
\end{remark}

\subsection{Notation and conventions}

For a subset $S$ of a metric space $X$, $x\in X$ and $\epsilon>0$, we write ``$x\in_\epsilon S$'' if there is $s\in S$ with $d(x,s)<\epsilon$.  For elements $x,y$ of a metric space $X$, we write ``$x\approx_\epsilon y$'' if $d(x,y)<\epsilon$.  

We write $\ell^2$ for $\ell^2(\N)$.  Throughout, the letters $A$ and $B$ are reserved for \emph{separable} $C^*$-algebras.  The letter $C$ will refer to a possibly non-separable $C^*$-algebra. The unit ball of $C$ (or a more general normed space) is denoted by $C_1$, its unitization is $C^+$, its multiplier algebra is $M(C)$, its suspension is $SC$, and its $n$-fold suspension is $S^nC$.  We write $M_n$ or $M_n(\C)$ for the $n\times n$ matrices, and $M_n(C)$ for the $n\times n$ matrices over a $C^*$-algebra $C$.

Our conventions on Hilbert modules follow those of Lance \cite{Lance:1995ys}.   We will write $H_B:=\ell^2\otimes B$ for the standard Hilbert $B$-module, and $\LB$, respectively $\KB$, as shorthand for the $C^*$-algebra $\LL(H_B)$ of adjointable operators on $H_B$, respectively the $C^*$-algebra $\K(H_B)$ of compact operators on $H_B$.  We will typically identify $\LB$ with the ``diagonal subalgebra'' $1_{M_n}\otimes \LB$ of $M_n\otimes \LB= M_n(\LB)$.  Thus we might write ``$[x,y]$'' for the commutator of $x\in\LB$ and $y\in M_n(\LB)$, when it would be more strictly correct to write something like ``$[1_{M_n}\otimes x,y]$''.

The symbol ``$\otimes$'' always denotes a completed tensor product: either the external tensor product of Hilbert modules (see \cite[Chapter 4]{Lance:1995ys} for background on this), or the minimal tensor product of $C^*$-algebras (see for example \cite[Chapter 3]{Brown:2008qy}).  

We will sometimes write $0_n$ and $1_n$ for the zero matrix and identity matrix of size $n$ when this seems helpful to avoid confusion, although we will generally omit the subscripts to avoid clutter.  If $n\leq m$, we will also use $1_n\in M_m(\C)$ for the rank $n$ projection with $n$ ones in the top-left part of the diagonal and zeros elsewhere.  Given an $n\times n$ matrix $a$ and an $m\times m$ matrix $b$, $a\oplus b$ denotes the ``block sum'' $(n+m)\times (n+m)$ matrix defined by 
$$ 
a\oplus b:=\begin{pmatrix} a & 0 \\ 0 & b \end{pmatrix}.
$$

Finally, $K_*(A):=K_0(A)\oplus K_1(A)$ denotes the graded $K$-theory group of a $C^*$-algebra, and $KK^*(A,B):=KK^0(A,B)\oplus KK^1(A,B)$ the graded $KK$-theory group.  We will typically just write $KK(A,B)$ instead of $KK^0(A,B)$.

\subsection{Outline of the paper}

Section \ref{reform sec} gives our reformulation of the UCT in terms of a concrete vanishing condition for controlled $KK$-theory.  The key ingredients for this are the Milnor sequence from \cite[Theorem 1.1]{Willett:2020aa}, and some ideas around the Mittag-Leffler condition from the theory of inverse limits (see for example \cite[Section 3.5]{Weibel:1995ty}).  We also show that a stronger vanishing result holds for nuclear, UCT $C^*$-algebras using an averaging argument of Christensen, Sinclair, Smith, White, and Winter \cite[Section 3]{Christensen:2012tt}; the averaging argument is in turn based on Haagerup's theorem \cite{Haagerup:1983wg} that nuclearity implies amenability.

Section \ref{flex sec} discusses our controlled $KK^0$-groups.  We introduced these in \cite{Willett:2020aa}, but we need a technical variation here.  This is essentially because in \cite{Willett:2020aa} we were setting up general theory, and for this it is easier to work with projections in a fixed $C^*$-algebra.  In this paper we are doing computations with concrete algebraic formulas, where it is more convenient to work with general idempotents, and to allow taking matrix algebras.  We will, however, use both versions in this paper, as we need to relate our work back here to the general theory of \cite{Willett:2020aa}.  We also introduce controlled $KK^1$-groups in a concrete formulation using invertible operators: in our earlier work \cite{Willett:2020aa} we (implicitly) defined controlled $KK^1$-groups using suspensions, but here we also need the more concrete version.

Section \ref{tech sec} collects together some technical facts.  These are all analogues for controlled $KK$-theory of well-known results from $K$-theory: for example, we prove ``controlled versions'' of the statements that homotopic idempotents are similar, and that similar idempotents are homotopic (up to increasing matrix sizes).   Some arguments in this section are adapted from the work of Oyono-Oyono and the second author \cite{Oyono-Oyono:2011fk} on controlled $K$-theory.

Section \ref{fd sec} revisits the vanishing conditions of Section \ref{reform sec}.  Using the techniques of Section \ref{tech sec}, we reformulate these results in the more flexible setting allowed by Section \ref{flex sec}.  This gives us the vanishing conditions that are the first main technical ingredient needed for Theorem \ref{main}.

Section \ref{bound sec} establishes the second main technical ingredient needed for Theorem \ref{main}.  Here we construct a ``Mayer-Vietoris boundary map'' for controlled $KK$-theory, and prove that it has an exactness property.  The construction is an analogue of the usual index map of operator $K$-theory (see for example \cite[Chapter 9]{Rordam:2000mz}), although concrete formulas for the Mayer-Vietoris boundary map unfortunately seem to be missing from the $C^*$-algebra literature.  The  formulas we use are instead inspired by classical formulas from algebraic $K$-theory \cite[Chapter 2]{Milnor:1971kl}, adapted to reflect our analytic setting.

Finally in the main body of the paper, Section \ref{dec sec} puts everything together and gives the proofs of Theorem \ref{main} and Theorem \ref{complex cor}.  We also include technical `local' vanishing results that we hope elucidate the structure of the proof, and might be useful in other contexts.
 
The paper concludes with Appendix \ref{examples app}, which gives examples of $C^*$-algebras with finite complexity.  We first use a technique of Winter and Zacharias \cite[Section 7]{Winter:2010eb} to show that the Cuntz alegbras $\mathcal{O}_n$ with $2\leq n<\infty$ have complexity rank one.  We then use our joint work with Guentner on dynamic complexity \cite{Guentner:2014bh} to show that ample, principal, \'{e}tale groupoids with finite dynamical complexity and compact base space have $C^*$-algebras of finite complexity; we also get a similar result without the ampleness assumption if we allow $C^*$-algebras with finite complexity relative to subhomogeneous $C^*$-algebras.

\subsection{Acknowledgements}

The authors were supported by the US NSF (DMS 1564281, DMS 1700021, DMS 1901522, DMS 2000082, DMS 2247313, DMS 2247968), and the Simons Fellow Program throughout the writing of this paper, and gratefully acknowledge this support.  We thank Wilhelm Winter for some helpful comments on a preliminary draft.  We thank Claude Schochet for pointing out the work of Gray \cite{Gray:1966tq} on inverse limits: this reference allowed us to significantly improve the results from an earlier circulated draft of this paper.  Finally, we thank the anonymous referee for useful comments.

\section{Reformulating the UCT}\label{reform sec}

In this section (as throughout), if $B$ is a separable $C^*$-algebra, then $\LB$ and $\K_B$ are respectively the adjointable and compact operators on the standard Hilbert $B$-module $\ell^2\otimes B$.  

Our goal in this section is to recall the definition of the controlled $KK$-theory groups, and then to reformulate the universal coefficient theorem in these terms.  

We first recall the definition of the controlled $KK$-theory groups from \cite{Willett:2020aa}: to be precise, we need the version from \cite[Sections A.1 and A.2]{Willett:2020aa} that is specific to unital $C^*$-algebras.   We need a definition.

\begin{definition}\label{np even}
Let $B$ be a separable $C^*$-algebra.  Choose a unitary isomorphism $\ell^2\cong \C^2\otimes \ell^2\otimes \ell^2$, which induces a unitary isomorphism $\ell^2\otimes B\cong (\C^2\otimes \ell^2\otimes \ell^2)\otimes B$ of Hilbert $B$-modules.  With respect to this isomorphism, let $e\in \LB$ be the projection corresponding to $\begin{pmatrix} 1 & 0 \\ 0 & 0 \end{pmatrix} \otimes 1_{\ell^2\otimes \ell^2\otimes B}$.  We call $e$ the \emph{neutral projection}.  A subset $X$ of $\LB$ is called \emph{large} if every $x\in X$ is of the form $1_{\C^2\otimes \ell^2}\otimes y$ for some $y\in \mathcal{L}(\ell^2\otimes B)$ with respect to this decomposition.
\end{definition}

\begin{definition}\label{basic con kk}
Let $B$ be a separable $C^*$-algebra.  Let $\epsilon>0$, let $X$ be a finite, large, subset of the unit ball of $\LB$ and let $e\in \LB$ be the neutral projection as in Definition \ref{np even}.  Let $\mathcal{P}_\epsilon(X,B)$ consist of those projections $p$ in $\LB$ such that:
\begin{enumerate}[(i)]
\item $p-e\in \K_B$; and
\item $\|[p,x]\|<\epsilon$ for all $x\in X$.
\end{enumerate}
Define $KK_{\epsilon}(X,B)$ to be the set $\pi_0(\mathcal{P}_\epsilon(X,B))$ of path components of $\mathcal{P}_\epsilon(X,B)$.  We write $[p]\in KK_\epsilon(X,B)$ for the class of $p\in \mathcal{P}_\epsilon(X,B)$.  

Choose now isometries $t_1,t_2\in \mathcal{B}(\ell^2)$ satisfying the \emph{Cuntz relation} $t_1t_1^*+t_2t_2^*=1$, and define $s_i:=1_{\C^2}\otimes t_i \otimes 1_{\ell^2\otimes B}\in \LB$.  Define an operation on $KK_\epsilon(X,B)$ by the \emph{Cuntz sum}
$$
[p]+[q]:=[s_1ps_1^*+s_2qs_2^*].
$$
\end{definition}  

The same proof as \cite[Lemma A.4]{Willett:2020aa} shows that $KK_\epsilon(X,B)$ is an abelian group, with identity element given by the class $[e]$ of the neutral projection.

We finish this subsection with two ancillary lemmas.  The first is extremely well-known; we include an argument for completeness as we do not know a convenient reference.

\begin{lemma}\label{spec lem}
Let $a$ and $b$ be elements of a unital $C^*$-algebra with $b$ normal.  Then any $z$ in the spectrum of $a$ is contained within distance $\|a-b\|$ of the spectrum of $b$.
\end{lemma}

\begin{proof}
We need to show that if $z$ is further than $\|a-b\|$ from the spectrum of $b$, then $a-z$ is invertible.  Indeed, in this case the continuous functional calculus implies that $\|(b-z)^{-1}\|< \|a-b\|^{-1}$.  Hence 
$$
\|(a-z)(b-z)^{-1}-1\|\leq \|(a-z)-(b-z)\|\|(b-z)^{-1}\|<1,
$$
whence $(a-z)(b-z)^{-1}$ is invertible, and so $a-z$ is invertible too.
\end{proof}

\begin{lemma}\label{con kk countable}
Let $B$ be a separable $C^*$-algebra, let $\epsilon>0$, and let $X$ be a finite, large, subset of the unit ball of $\LB$.  With notation as in Definition \ref{basic con kk}, the group $KK_\epsilon(X,B)$ is countable.
\end{lemma}

\begin{proof}
As $B$ is separable $\K_B$ is separable, and so the set $\mathcal{P}_\epsilon(X,B)$ is also separable.  Let $S$ be a countable dense subset of $\mathcal{P}_\epsilon(X,B)$.  It suffices to show that the map $S\to KK_\epsilon(X,B)$ defined by $p\mapsto [p]$ is surjective.  

Let $p\in \mathcal{P}_\epsilon(X,B)$ be arbitrary, and define 
$$
\delta:=\min\Big\{\frac{1}{4}(\epsilon-\max_{x\in X} \|[p,x]\|),1/2\Big\}.
$$
Let $q\in S$ be such that $\|p-q\|<\delta$, and let $p_t:=(1-t)p+tq$ for $t\in [0,1]$.  Then for each $t\in [0,1]$, $\|p_t-p\|<\delta$, so Lemma \ref{spec lem} and that $p_t$ is a positive contraction implies that the spectrum $p_t$ is contained in $[0,\delta)\cup (1-\delta,1]$.  Let $\chi$ be the characteristic function of $(1/2,\infty)$.  Then $\|\chi(p_t)-p_t\|<\delta$ for all $t$, whence $\|\chi(p_t)-p\|<2\delta$ for all $t$, from which it follows that $\|[\chi(p_t),x]\|<\epsilon$ for all $t$ and all $x\in X$.  As $p_t-e\in \K_B$ for all $t$, it follows from the fact that $\K_B$ is an ideal in $\LB$ that $\chi(p_t)-e\in \K_B$ too.  Hence $(\chi(p_t))_{t\in [0,1]}$ is a path connecting $p$ and $q$ within $\mathcal{P}_\epsilon(X,B)$ so $[p]=[q]$, and we are done.
\end{proof}

\subsection{The general case}\label{gen reform sec}

We need a special class of representations on Hilbert $B$-modules, essentially taken from work of Thomsen \cite[Definition 2.2]{Thomsen:2000aa} (see also \cite[Definition A.11]{Willett:2020aa}).  We do not need the details of the definition below, and only include it for completeness: all we really need are the facts about existence of such representations in Lemma \ref{scalar sa exist} below.

\begin{definition}\label{usa def}
Let $A$ be a separable, unital $C^*$-algebra, and let $B$ be a separable $C^*$-algebra.  A representation $\sigma:A\to \LB$ is \emph{unitally absorbing} if for any unital completely positive map $\phi:A\to \LB$ there exists a sequence of isometries $(v_n)$ in $\LB$ such that $\|v_n^*\sigma(a)v_n-\phi(a)\|\to 0$ as $n\to\infty$, and such that $v_n^*\sigma(a)v_n-\phi(a)\in \K_B$ for all $n\in \N$.  

For a representation $\sigma:A\to \LB=\mathcal{L}(H_B)$, let $\sigma^\infty:A\to \mathcal{L}(H_B^{\oplus\infty})$ be its infinite amplification, which we identify with a representation $\sigma^\infty:A\to \LB$ via a choice of unitary isomorphism $(\ell^2)^{\oplus \infty}\cong \ell^2$ as in the string of identifications below 
$$
\mathcal{L}(H_B^{\oplus \infty})= \mathcal{L}((\ell^2\otimes B)^{\oplus \infty})= \mathcal{L}((\ell^2)^{\oplus \infty}\otimes B)\cong \mathcal{L}(\ell^2\otimes B)=\LB
$$
(all of the identifications labeled ``='' are canonical).  A unital representation $\pi:A\to \LB$ is \emph{strongly unitally absorbing} if there is a unitally absorbing representation $\sigma:A\to \LB$ such that $\pi=\sigma^{\oplus\infty}$.
\end{definition}

Note that a (strongly) unitally absorbing representation is faithful.  The following result is essentially due to Thomsen and Kasparov.  Our main use of part \eqref{sa nuc} occurs much later in the paper.

\begin{lemma}\label{scalar sa exist}
Let $A$ be a separable, unital $C^*$-algebra, and let $B$ be a separable $C^*$-algebra.  Then:
\begin{enumerate}[(i)]
\item \label{sa gen} There exists a strongly unitally absorbing representation $\pi:A\to \LB$.
\item \label{sa nuc} Assume in addition that $A$ or $B$ is nuclear.  Let $\sigma:A\to \mathcal{B}(\ell^2)$ be any faithful unital representation, let $\iota:\mathcal{B}(\ell^2)\to \LB$ be the canonical inclusion arising from the decomposition $H_B=\ell^2\otimes B$, and let $\pi:A\to \LB$ be the infinite amplification of $\iota\circ \sigma$.  Then $\pi$ is strongly unitally absorbing.
\end{enumerate}
\end{lemma}

\begin{proof}
For part \eqref{sa gen}, Thomsen shows in \cite[Theorem 2.4]{Thomsen:2000aa} that a unitally absorbing representation $\sigma:A\to \LB$ exists under the given hypotheses.  Its infinite amplification  $\pi$ is then strongly unitally absorbing.  

For part \eqref{sa nuc}, note first that identifying $(\iota\circ\sigma)^\infty$ with $(\iota\circ (\sigma^{\oplus \infty}))^\infty$ we may assume $\sigma$ is the infinite amplification of some faithful unital representation $A\to \mathcal{B}(\ell^2)$.  Having made this assumption, note that $\sigma(A)\cap\K(\ell^2)=\{0\}$.  In \cite[Theorem 5]{Kasparov:1980sp}, Kasparov shows that if $A$ is a separable, unital $C^*$-algebra and $\sigma:A\to \mathcal{B}(\ell^2)$ is a faithful representation such that $\sigma(A)\cap \K(\ell^2)=\{0\}$, and moreover if either $A$ or $B$ is nuclear, then the composition $\iota\circ \sigma$ satisfies the condition Thomsen gives in \cite[Theorem 2.1, condition (4)]{Thomsen:2000aa}.  Comparing \cite[Theorem 2.1]{Thomsen:2000aa} and Definition \ref{usa def}, we see that $\iota\circ \sigma$ is unitally absorbing.  Hence $\pi=(\iota\circ \sigma)^{\oplus \infty}$ is strongly unitally absorbing.
\end{proof}

The following corollary is immediate from part \eqref{sa nuc} of Lemma \ref{scalar sa exist}.

\begin{corollary}\label{sub sua}
Let $A$ be a separable, unital, nuclear $C^*$-algebra, and let $B$ be a separable $C^*$-algebra.  Then there exists a strongly unitally absorbing representation $\pi:A\to \LB$ such that the restriction of $\pi$ to any unital, nuclear $C^*$-subalgebra of $A$ is also strongly unitally absorbing. \qed 
\end{corollary}

\begin{remark}\label{nuc rem}
Corollary \ref{sub sua} is one of the two places nuclearity is used in the proof of Theorem \ref{main}, so it would be interesting to establish the corollary under some weaker assumption than nuclearity.  The following observation shows that the method we used to establish Corollary \ref{sub sua} cannot extend beyond the nuclear case, however.  

Let $A$ be a separable, unital $C^*$-algebra, and let $A=B$.  Let $\sigma:A\to \mathcal{B}(\ell^2)$ be a unital representation, and let $\pi:=\iota\circ \sigma:A\to \mathcal{L}_A$ be as in Lemma \ref{scalar sa exist} part \eqref{sa nuc}.  We claim that if $\pi$ is unitally absorbing, then $A$ is nuclear\footnote{The following argument is inspired by \cite[Th\'{e}or\`{e}me 1.5, Definition 1.6, and Remarque 1.7]{Skandalis:1988rr}.}.  Let $\phi:A\to \mathcal{L}_A$ be the $*$-homomorphism $a\mapsto 1_{\ell^2}\otimes a$.  If $\pi$ is unitally absorbing then for any $\epsilon$ and finite subset $X$ of $A$ there is an isometry $v\in \mathcal{L}_A$ such that $\|v^*\pi(a)v-\phi(a)\|<\epsilon$ for all $a\in X$.  For each $n$, let $p_n\in \mathcal{B}(\ell^2)$ be the orthogonal projection onto $\ell^2(\{1,...,n\})$, and let $q_n:=p_n\otimes 1_A\in \mathcal{L}_A$.  Note that $q_1\mathcal{L}_Aq_1$ identifies canonically with $A$, and up to this identification $q_1\phi(a)q_1=a$ for all $a\in A$, so in particular $\|q_1v^*\pi(a)vq_1-a\|<\epsilon$ for all $a\in X$.  As $(q_n)$ converges strictly to the identity in $\mathcal{L}_A$, and as $q_1v\in \mathcal{K}_A$, we have moreover that $q_1v^*q_n\pi(a)q_n vq_1$ converges in norm to $q_1v^*\pi(a)v^*q_1$, so there is $n$ such that $\|q_1v^*q_n\pi_n(a)q_nvq_1-a\|<\epsilon$ for all $a\in X$.  We thus have ucp maps
$$
\xymatrix{ A \ar[rr]^-{a\mapsto q_n\pi(a)q_n}  & & q_n(\mathcal{B}(\ell^2)\otimes 1_A)q_n\cong M_n(\C) \ar[rr]^-{b\mapsto q_1v^*bvq_1} & & A}  
$$
whose composition agrees with the identity on $X$ to within $\epsilon$ error.  As $X$ and $\epsilon$ were arbitrary, this implies nuclearity of $A$ (see for example \cite[Chapter 2]{Brown:2008qy}).  
\end{remark}

To state the main result of \cite{Willett:2020aa}, we need some more definitions.

\begin{definition}\label{large rep}
Let $A$ be a separable, unital $C^*$-algebra, and let $B$ be a separable $C^*$-algebra.  A representation $\pi:A\to \LB$ is \emph{large} if there is a unitally absorbing representation $\sigma:A\to \LB$ such that with respect to the choice of isomorphism $\ell^2\otimes B\cong \C^2\otimes \ell^2\otimes \ell^2\otimes B$ of Definition \ref{np even}, we have $\pi(a)=1_{\C^2\otimes\ell^2}\otimes \sigma(a)$ for all $a\in A$.
\end{definition}

Lemma \ref{scalar sa exist} part \eqref{sa gen} implies that large representations exist for any (separable) $A$ and $B$.  Note that if $\pi$ is large in the sense of Definition \ref{large rep} then for any $X\subseteq A$, the subset $\pi(X)\subseteq \LB$ is large in the sense of Definition \ref{np even}.  In particular, if we identify $X$ with $\pi(X)$, the group $KK_\epsilon(X,B)$ of Definition \ref{basic con kk} makes sense.

\begin{definition}\label{dir set 0}
Let $C$ be a $C^*$-algebra, and let $\mathcal{X}_C$ consist of all pairs of the form $(X,\epsilon)$ where $X$ is a finite subset of $C_1$, and $\epsilon>0$.  Put a partial order on $\mathcal{X}_C$ by stipulating that $(X,\epsilon)\leq (Y,\delta)$ if $\delta\leq \epsilon$, and if for all $x\in X$ there exists $y\in Y$ with $\|x-y\|\leq \frac{1}{2}(\epsilon-\delta)$.

A \emph{good approximation} of $C$ is a cofinal sequence\footnote{A sequence $(s_n)_{n=1}^\infty$ in a partially ordered set $S$ is \emph{cofinal} if $s_1\leq s_2\leq s_3\leq \cdots $ and if for all $s\in S$ there is $n$ such that $s\leq s_n$.} $((X_n,\epsilon_n))_{n=1}^\infty$ of elements of $\mathcal{X}_C$.
\end{definition}

Note that if $X\subseteq Y$ and $\delta\leq \epsilon$, then $(X,\epsilon)\leq (Y,\delta)$; in particular, this implies that $\mathcal{X}_C$ is a directed set.  Note also that good approximations exist if and only if $C$ is separable: if $(\epsilon_n)$ is a decreasing sequence that tends to zero, and $(X_n)$ is an increasing sequence with dense union in $C_1$, then $((X_n,\epsilon_n))_{n=1}^\infty$ is a good approximation; and if $((X_n,\epsilon_n)_{n=1}^\infty$ is a good approximation, then $\bigcup_{n=1}^\infty X_n$ is a countable dense subset of $C_1$.

\begin{definition}\label{dir set} 
Let $B$ be a separable $C^*$-algebra, and let $\mathcal{X}_{\LB}$ be the directed set from Definition \ref{dir set 0} above for the $C^*$-algebra $\LB$.  If $(X,\epsilon)\leq (Y,\delta)$ and $X$ and $Y$ are both large in the sense of Definition \ref{np even}, then with notation as in Definition \ref{basic con kk} there is an inclusion
\begin{equation}\label{p inclusion}
\mathcal{P}_\delta(Y,B)\subseteq \mathcal{P}_\epsilon(X,B).
\end{equation}
We call the canonical map 
$$
KK_\delta(Y,B)\to KK_\epsilon(X,B)
$$
induced by the inclusion in line \eqref{p inclusion} above a \emph{forget control map}.  
\end{definition}

We now briefly recall some terminology from homological algebra: see for example \cite[Section 3.5]{Weibel:1995ty} or \cite[Section 3]{Schochet:2003vq} for more background on this material\footnote{Readers interested in a more sophisticated and general treatment can also see \cite{Jensen:1972wu}.}.  An \emph{inverse system} of abelian groups consists of a sequence of abelian groups and homomorphisms 
$$
\xymatrix{ \cdots \ar[r]^-{\phi_n} & A_n \ar[r]^-{\phi_{n-1}} & A_{n-1} \ar[r]^-{\phi_{n-2}} & \cdots \ar[r]^-{\phi_2} & A_2 \ar[r]^-{\phi_1} & A_1}.
$$
Associated to such a system is a homomorphism
$$
\phi:\prod_{n\in \N} A_n\to \prod_{n\in \N} A_n,\quad (a_n)\mapsto (\phi_n(a_{n+1})).
$$
The \emph{inverse limit}, denoted ${\displaystyle \lim_{\leftarrow} A_n}$, is defined to be the kernel of $\text{id}-\phi$, and the \emph{${\displaystyle\lim_\leftarrow\!^1}$-group}, denoted ${\displaystyle \lim_\leftarrow\!^1 A_n}$, is defined to be the cokernel of $\text{id}-\phi$.   Note that if $m\geq n$, there is a canonical homomorphism $A_m\to A_n$ defined as $\phi_n\circ \phi_{n+1}\circ \cdots \circ \phi_{m-1}$.   The inverse system satisfies the \emph{Mittag-Leffler condition} if for any $n$ there is $N\geq n$ such that for all $m\geq N$, the image of the canonical map $A_m\to A_n$ equals the image of the canonical map $A_N\to A_n$.

\begin{proposition}\label{gray the}
Let $(A_n)$ be an inverse system of abelian groups.  If $(A_n)$ satisfies the Mittag-Leffler condition, then ${\displaystyle \lim_\leftarrow\!^1 A_n}=0$.  Conversely, if ${\displaystyle \lim_\leftarrow\!^1}A_n=0$ and each $A_n$ is countable, then the inverse system satisfies the Mittag-Leffler condition. 
\end{proposition}

\begin{proof}
It is well-known that the Mittag-Leffler condition implies vanishing of ${\displaystyle \lim_\leftarrow\!^1 A_n=0}$: see for example \cite[Proposition 3.5.7]{Weibel:1995ty}.  The converse in the case of countable groups follows from \cite[Proposition on page 242]{Gray:1966tq}.
\end{proof}

Now, let $A$ be a separable, unital $C^*$-algebra, let $B$ be a separable $C^*$-algebra, and use a large representation $\pi:A\to \LB$ (see Definition \ref{large rep}) to identify $A$ with a $C^*$-subalgebra of $\LB$.  Let $((X_n,\epsilon_n))_{n=1}^\infty$ be a good approximation of $A$ as in Definition \ref{dir set 0}, so the forget control maps of Definition \ref{dir set} form an inverse system 
$$
\xymatrix{ \cdots \ar[r] & KK_{\epsilon_n}(X_n,B) \ar[r] & KK_{\epsilon_{n-1}}(X_{n-1},B) \ar[r] & \cdots \ar[r] & KK_{\epsilon_1}(X_1,B)}
$$
from which we define ${\displaystyle \lim_{\leftarrow} KK_{\epsilon_n}(X_n,B)}$ and ${\displaystyle \lim_{\leftarrow}\!^1 KK_{\epsilon_n}(X_n,B)}$ as above.

The following is \cite[Proposition A.10]{Willett:2020aa}. 

\begin{theorem}\label{milnor ses}
Let $A$ and $B$ be separable $C^*$-algebras with $A$ unital.  Let $\pi:A\to \LB$ be a large representation, and use this to identify $A$ with a $C^*$-subalgebra of $\LB$.  Let $((X_n,\epsilon_n))_{n=1}^\infty$ be a good approximation for $A$.  Then there is a short exact sequence
$$
0 \to {\displaystyle \lim_{\leftarrow}\!^1 KK_{\epsilon_n}(X_n,SB)} \to KK(A,B) \to {\displaystyle \lim_{\leftarrow} KK_{\epsilon_n}(X_n,B)} \to 0 \eqno\qed.
$$
\end{theorem}

We are now almost ready to state and prove our reformulation of the UCT.  It will be convenient to use the following well-known reformulation of the UCT: see \cite[Page 457]{Rosenberg:1987bh} or \cite[Proposition 5.3]{Skandalis:1988rr} for a proof.

\begin{theorem}\label{uct zero}
A separable $C^*$-algebra $A$ satisfies the UCT if and only if for any separable $C^*$-algebra $B$ such that $K_*(B)=0$ we have that $KK(A,B)=0$.  \qed
\end{theorem}

\begin{theorem}\label{uct reform}
Let $A$ be a separable $C^*$-algebra.  The following are equivalent:
\begin{enumerate}[(i)]
\item \label{a uct} $A$ satisfies the UCT.
\item \label{a vanish} Let $B$ be a separable $C^*$-algebra with $K_*(B)=0$.  Let $\pi:A\to \LSB$ be a large representation, and use this to identify $A$ with a $C^*$-subalgebra of $\LSB$.   Then for any $(X,\gamma)$ in the set $\mathcal{X}_A$ of Definition \ref{dir set 0}, there is $(Z,\epsilon)\in \mathcal{X}_A$ with $(X,\gamma)\leq (Z,\epsilon)$ and so that the forget control map
$$
KK_\epsilon(Z,SB)\to KK_\gamma(X,SB)
$$
of Definition \ref{dir set} is zero.
\end{enumerate}
\end{theorem}

\begin{proof}
Assume first that $A$ satisfies condition \eqref{a uct}, and let $X$, $\epsilon$, $B$ and $\pi$ be as in condition \eqref{a vanish}.  Let $((X_n,\epsilon_n))_{n=1}^\infty$ be a good approximation of $A$ with $X_1=X$ and $\epsilon_1=\gamma$.  As $A$ satisfies the UCT and as $K_*(B)=0$, we have $KK(A,B)=0$.  Hence using Theorem \ref{milnor ses}, ${\displaystyle \lim_{\leftarrow}\!^1 KK_{\epsilon_n}(X_n,SB)=0}$.  Lemma \ref{con kk countable} implies that the groups $KK_{\epsilon_n}(X_n,SB)$ are all countable, whence by Proposition \ref{gray the}, the inverse system $(KK_{\epsilon_n}(X_n,SB))_{n=1}^\infty$ satisfies the Mittag-Leffler condition.  On the other hand, as $A$ satisfies the UCT and $K_*(SB)=0$, we have $KK(A,SB)=0$ by Theorem \ref{uct zero}.  Hence by Theorem \ref{milnor ses} again, ${\displaystyle \lim_{\leftarrow} KK_{\epsilon_n}(X_n,SB)=0}$, whence the definition of the inverse limit implies that for any $n$, 
$$
\bigcap_{m\geq n} \text{Image}\big(KK_{\epsilon_m}(X_m,SB)\to KK_{\epsilon_n}(X_n,SB)\big)=0.
$$
The Mittag-Leffler condition implies that there is $N\geq n$ such that 
\begin{align*}
\bigcap_{m\geq n} \text{Image} & \big(KK_{\epsilon_m}(X_m,SB)\to KK_{\epsilon_n}(X_n,SB)\big) \\ & =\text{Image}\big(KK_{\epsilon_N}(X_N,SB)\to KK_{\epsilon_n}(X_n,SB)\big)
\end{align*}
so we may conclude that the forget control map 
$$
KK_{\epsilon_N}(X_N,SB) \to KK_{\epsilon_n}(X_n,SB)
$$
is zero.  In particular, such an $N$ exists for $n=1$, and we may set $Z=X_N$ and $\epsilon=\epsilon_N$.

Conversely, say $A$ satisfies condition \eqref{a vanish}.  Using Theorem \ref{uct zero}, it suffices to show that if $B$ is a separable $C^*$-algebra with $K_*(B)=0$, then $KK(A,B)=0$.  Let $\pi_2:A\to \mathcal{L}_{S^2B}$ (respectively, $\pi_3:A\to \mathcal{L}_{S^3B}$) be a large representation, and use this to identify $A$ with a $C^*$-subalgebra of $\mathcal{L}_{S^2B}$ (respectively, $\mathcal{L}_{S^3B}$). Using condition \eqref{a vanish} we may construct a good approximation $((X_n,\epsilon_n))_{n=1}^\infty$ for $A$ in the sense of Definition \ref{dir set 0} such that for any $n$ the maps 
\begin{equation}\label{s3 0}
KK_{\epsilon_{n+1}}(X_{n+1},S^3B)\to KK_{\epsilon_n}(X_n,S^3B)
\end{equation}
and
\begin{equation}\label{s2 0}
KK_{\epsilon_{n+1}}(X_{n+1},S^2B)\to KK_{\epsilon_n}(X_n,S^2B) 
\end{equation}
are zero.  As the maps in line \eqref{s3 0} are all zero, the inverse system $(KK_{\epsilon_n}(X_n,S^3B))_{n=1}^\infty$ satisfies the Mittag-Leffler condition, whence by Proposition \ref{gray the} we have that ${\displaystyle \lim_{\leftarrow}\!^1 KK_{\epsilon_n}(X_n,S^3B)=0}$.  On the other hand, the fact that the maps in line \eqref{s2 0} are all zero and the definition of the inverse limit immediately imply that  ${\displaystyle \lim_{\leftarrow}\! KK_{\epsilon_n}(X_n,S^2B)=0}$.  Hence in the short exact sequence 
$$
0\to  {\displaystyle \lim_{\leftarrow}\!^1 KK_{\epsilon_n}(X_n,S^3B)}\to KK(A,S^2B) \to {\displaystyle \lim_{\leftarrow} KK_{\epsilon_n}(X_n,S^2B)} \to  0
$$
from Theorem \ref{milnor ses} the left and right groups are zero, whence $KK(A,S^2B)=0$.  Hence by Bott periodicity, $KK(A,B)=0$ as desired.
\end{proof}

We include the following remark as the comparison to the existing literature might help orient some readers; it also gives a sense of why Corollary \ref{sub sua} is useful (our main use of that corollary will come later in the paper).

\begin{remark}\label{dad rem}
Theorem \ref{uct reform} can be used to deduce a weak version of a theorem of Dadarlat \cite[Theorem 1.1]{Dadarlat:2003tg}.  Dadarlat shows that if $A$ is a separable nuclear $C^*$-algebra such for any finite subset $X$ of $A$ and any $\epsilon>0$, one has a UCT subalgebra $C$ of $A$ such that $x\in_\epsilon C$ for all $x\in X$, then $A$ satisfies the UCT.  Theorem \ref{main} implies the special case of Dadarlat's theorem where the subalgebras $C$ can also be taken nuclear.  

To see this, note first that as a $C^*$-algebra satisfies the UCT (respectively, is nuclear) if and only if its unitization satisfies the UCT (respectively, is nuclear) by \cite[Proposition 2.3 (a)]{Rosenberg:1987bh} (respectively, by \cite[Exercise 2.3.5]{Brown:2008qy}), we may assume that $A$ is unital.  We aim to establish the condition in Theorem \ref{uct reform} part \eqref{a vanish}.  Let then $B$ be a separable $C^*$-algebra with $K_*(B)=0$.  
Using Corollary \ref{sub sua}, there exists a large representation $\pi:A\to \LSB$ such that the restriction of $\pi$ to any unital nuclear $C^*$-subalgebra of $A$ is also large.  Let $X$ be a finite subset of $A_1$, and let $\epsilon>0$.  Let $C$ be a nuclear, unital, UCT $C^*$-subalgebra of $A$ such that $x\in_{\epsilon/5} C$ for all $x\in X$.  Let $X'$ be a finite subset of $C_1$ such that for each $x\in X$ there is $x'\in X'$ such that $\|x-x'\|<2\epsilon/5$.  Then the forget control map 
\begin{equation}\label{dadfc1}
KK_{\epsilon/5}(X',SB)\to KK_{\epsilon}(X,B)
\end{equation}
of Definition \ref{dir set} is defined.  As $C$ satisfies the UCT, and as the restriction of $\pi$ to $C$ is also large, condition \eqref{a vanish} from Theorem \ref{uct reform} gives a finite subset $Y$ of $C_1$ and $\delta>0$ such that the forget control map
\begin{equation}\label{dadfc2}
KK_\delta(Y,SB)\to KK_{\epsilon/5}(X',SB)
\end{equation}
is defined and zero.  Composing the forget control maps in lines \eqref{dadfc1} and \eqref{dadfc2}, we have established the condition from Theorem \ref{uct reform} part \eqref{a vanish} for $A$, and are done.

It would be interesting if one could use these techniques to recover Dadarlat's theorem without the extra nuclearity assumption on the UCT subalgebras.  This would seem to require better control over the representations involved, however: compare Remark \ref{nuc rem} above.
\end{remark}





\subsection{The nuclear case}

In this section, we prove a stronger version of Theorem \ref{uct reform} in the special case that the $C^*$-algebra $A$ is nuclear.  The key ingredient for this is an averaging argument due to Christensen, Sinclair, Smith, White, and Winter \cite[Section 3]{Christensen:2012tt}, which in turn relies on Haagerup's theorem \cite{Haagerup:1983wg} that nuclear $C^*$-algebras are amenable.  

Let us recall some terminology about bimodules.

\begin{definition}\label{bimod def}
Let $A$ be a unital $C^*$-algebra.  An \emph{$A$-bimodule} is a Banach space $E$ equipped with left and right module actions of $A$ such that $1_Ae=e1_A=e$ for all $e\in E$, and such that $\|ae\|_E\leq \|a\|_A\|e\|_E$ and $\|ea\|_E\leq \|e\|_E\|a\|_A$ for all $a\in A$ and $e\in E$.
\end{definition}

The following reformulation of nuclearity is implicit in \cite[Section 3]{Christensen:2012tt}; the reader is encouraged to see that reference for further background.

\begin{lemma}\label{amen lem}
Let $A$ be a unital $C^*$-algebra.  Then the following are equivalent:
\begin{enumerate}[(i)]
\item $A$ is nuclear;
\item \label{na 2} for any $\epsilon>0$ and any finite subset $X$ of $A$, there exist contractions $a_1,...,a_n\in A $ and scalars $t_1,...,t_n\in [0,1]$ such that $\sum_{i=1}^n t_i=1$, such that 
$$
\Bigg\|1_A-\sum_{i=1}^n t_i a_ia_i^*\Bigg\|_A<\epsilon,
$$
and such that for any $A$-bimodule $E$, any $e\in E_1$, and any $x\in X$,
\begin{equation}\label{amen com 2}
\Bigg\|x\Big(\sum_{i=1}^n t_ia_ie a_i^*\Big)-\Big(\sum_{i=1}^n t_ia_ie a_i^*\Big)x\Bigg\|_{E}< \epsilon.
\end{equation}
\end{enumerate}
\end{lemma}

\begin{proof}
We will need to recall the projective tensor product of Banach spaces.  Let $E$ and $F$ be (complex) Banach spaces, and let $E\odot F$ denote their algebraic tensor product (over $\C$).  The \emph{projective norm} of  $g\in E\odot F$ is defined by
\begin{equation}\label{proj norm}
\|g\|:=\inf\sum_{i=1}^n \|e_i\|_E\|f_i\|_F,
\end{equation}
where the infimum is taken over all ways of writing $g$ as a sum $\sum_{i=1}^n e_i\otimes f_i$ of elementary tensors.  The \emph{projective tensor product} of $E$ and $F$, denoted $E\widehat{\otimes} F$, is the completion of $E\odot F$ for the projective norm.  If $A$ is a $C^*$-algebra, we make $A\widehat{\otimes} A$ into an $A$-$A$-bimodule via the actions defined on elementary tensors by 
\begin{equation}\label{bimod act}
a(b\otimes c):=ab\otimes c \quad \text{and}\quad (b\otimes c)a:=b\otimes ca.
\end{equation}
Now, it is shown in \cite[Lemma 3.1]{Christensen:2012tt}\footnote{This is based on several deep ingredients: the key points are the result of Connes \cite[Corollary 2]{Connes:1978ue} that amenability for a $C^*$-algebra implies nuclearity; the converse to this due to Haagerup \cite[Theorem 3.1]{Haagerup:1983wg}; and Johnson's foundational work on amenability and virtual diagonals \cite[Section  1]{Johnson:1972vp}.}  that a unital $C^*$-algebra is nuclear if and only if the following holds:  ``for any $\epsilon>0$ and any finite subset $X$ of $A$, there exist contractions $a_1,...,a_n\in A $ and scalars $t_1,...,t_n\in [0,1]$ such that $\sum_{i=1}^n t_i=1$, such that 
$$
\Bigg\|1_A-\sum_{i=1}^n t_i a_ia_i^*\Bigg\|_A<\epsilon,
$$
and such that 
\begin{equation}\label{amen com 1}
\Bigg\|x\Big(\sum_{i=1}^n t_ia_i\otimes  a_i^*\Big)-\Big(\sum_{i=1}^n t_ia_i\otimes  a_i^*\Big)x\Bigg\|_{A\widehat{\otimes} A}< \epsilon
\end{equation}
for all $x\in X$.''
For the sake of this proof, let us call this the ``CSSWW'' condition.  It suffices for us to show that condition \eqref{na 2} is equivalent to the CSSWW condition.  

First assume $A$ satisfies condition \eqref{na 2} above.  Then taking $E=A\widehat{\otimes} A$ and $e=1_A\otimes 1_A$ shows that $A$ satisfies the CSSWW condition.  Conversely, say $A$ satisfies the CSSWW condition.  Let $X$ be a finite subset of $A$ and let $\epsilon>0$, and let $a_1,...,a_n$ and $t_1,...,t_n$ satisfy the properties in the CSSWW condition with respect to this $X$ and $\epsilon$.  Let $E$ be an $A$-bimodule, and $e\in E_1$.  Consider the map
$$
\pi:A\odot A\to E,\quad a\otimes b \mapsto aeb
$$
from the algebraic tensor product (over $\C$) of $A$ with itself to $E$.  Using the definition of the projective tensor norm (line \eqref{proj norm} above), it is straightforward to check that $\pi$ is contractive for that norm, whence it extends to a contractive linear map $\pi:A\widehat{\otimes} A\to E$.  Moreover, the extended map $\pi$ is clearly an $A$-bimodule map for the bimodule structure on $A\widehat{\otimes} A$ defined in line \eqref{bimod act}.  Applying $\pi$ to the expression inside the norm in line \eqref{amen com 1} therefore implies the inequality in line \eqref{amen com 2}, so we are done. 
\end{proof}

\begin{remark}\label{con rem}
We will only need to apply Lemma \ref{amen lem} in the special case that the bimodule $E$ in part \eqref{na 2} is a $C^*$-algebra containing $A$ as a unital $C^*$-subalgebra, with the bimodule actions defined by left and right multiplication.  The corresponding, formally weaker, variant of condition \eqref{na 2} still implies nuclearity, as we now sketch\footnote{This also gives an approach to the theorem of Connes that amenable $C^*$-algebras are nuclear that is maybe slightly more direct than the original argument from \cite[Corollary 2]{Connes:1978ue}.  However, it still factors through the theorem that injective von Neumann algebras are semi-discrete (see \cite[Theorem 6]{Connes:1976fj} for the case of factors, and \cite{Wassermann:1979tr} for the general case), so cannot really be said to be genuinely simpler.}.  Let $A$ be a unital $C^*$-algebra satisfying the variant of condition \eqref{na 2} from Lemma \ref{amen lem}, where $E$ is a $C^*$-algebra containing $A$ as a unital $C^*$-subalgebra.  Let $\pi:A\to \mathcal{B}(H)$ be an arbitrary unital representation, which we use to make $\mathcal{B}(H)$ an $A$-bimodule.  Let $I$ be the directed set consisting of all pairs $i=(X,\epsilon)$ where $X$ is a finite subset of $A$, and $\epsilon>0$, and where $(X,\epsilon)\leq (Y,\delta)$ if $X\subseteq Y$ and $\delta\leq \epsilon$.  For each $i=(X,\epsilon)\in I$, let $a^{(i)}_1,...,a^{(i)}_{n_i}$ and $t^{(i)}_1,...,t^{(i)}_{n_i}$ have the properties in Lemma \ref{amen lem} condition \eqref{na 2}.  For each $i$, define a ccp map 
$$
\phi_i:\mathcal{B}(H)\to \mathcal{B}(H),\quad b\mapsto \sum_{j=1}^{n_i} t^{(i)}_j\pi(a_j^{(i)})b\pi(a_j^{(i)})^*,
$$
and let $\phi:\mathcal{B}(H)\to \mathcal{B}(H)$ be any point-ultraweak limit point of the net $(\phi_i)$ (such exists by \cite[Theorem 1.3.7]{Brown:2008qy}, for example).  Then one checks that $\phi$ is a conditional expectation from $\mathcal{B}(H)$ onto $\pi(A)'$, whence the latter is injective.  As $\pi$ was arbitrary, this implies that $A$ is nuclear: indeed, applying this to the universal representation $\pi$ implies that $\pi(A)'$ is injective, whence $A^{**}=\pi(A)''$ is injective by \cite[IV.2.2.7]{Blackadar:2006eq}, whence $A$ is nuclear by the main result of \cite{Choi:1977wq}.
\end{remark}

Variants of the next lemma we need are well-known: see for example the lemma on page 332 of \cite{Arveson:1977aa}, which we could have used for a purely qualitative version.  For the sake of concreteness, we give a quantitative\footnote{The estimate it gives is optimal in some sense: to see this consider $C=M_2(\C)$, $x={\tiny\begin{pmatrix} \delta & 0 \\ 0 & 1-\delta\end{pmatrix}}$, and $c={\tiny\begin{pmatrix} 0 & 1 \\ 1 & 0 \end{pmatrix}}$.} version.

\begin{lemma}\label{proj com}
Let $\delta\in [0,1/2)$, and let $x$ be a self-adjoint element in a $C^*$-algebra $C$ with spectrum that does not intersect the interval $(\delta,1-\delta)$.  Let $\chi$ be the characteristic function of $(1/2,\infty)$. Then for any $c\in C$,
$$
\|[\chi(x),c]\|\leq \frac{1}{1-2\delta}\|[x,c]\|.
$$
\end{lemma}

\begin{proof}
Let $N>\|x\|$.  Let $\gamma$ be the positively oriented rectangular contour in the complex plane with vertices at $\frac{1}{2}\pm iN$, and $2N\pm iN$.  Then by the holomorphic functional calculus, $\chi(x)=\frac{1}{2\pi i} \int_\gamma (z-x)^{-1}dz$.  Hence for any $c\in C$, $[\chi(x),c]=\frac{1}{2\pi i} \int_\gamma [(z-x)^{-1},c]dz$.  Applying the formula 
$$
[(z-x)^{-1},c]=(z-x)^{-1}[c,x](z-x)^{-1}
$$
and estimating gives 
\begin{equation}\label{int com}
\|[\chi(x),c]\|\leq\frac{\|[c,x]\|}{2\pi} \int_\gamma \|(z-x)^{-1}\|^2d|z|.
\end{equation}
Let $\gamma_1$ be the side of $\gamma$ described by $\{\frac{1}{2}+it\mid -N\leq t\leq N\}$, and let $\gamma_2$ be the union of the other three sides.  Then for $z$ in the image of $\gamma_2$, the continuous functional calculus implies that $\|(z-x)^{-1}\|\leq (N-\|x\|)^{-1}$.  As the length of $\gamma_2$ is $4N$, we thus see that 
\begin{equation}\label{gamma2}
\int_{\gamma_2}\|(z-x)^{-1}\|^2\|d|z|\leq \frac{4N}{(N-\|x\|)^2}
\end{equation}
On the other hand, for $z=\frac{1}{2}+it$ in the image of $\gamma_1$, the continuous functional calculus gives $\|(z-x)^{-1}\|\leq \big((\frac{1}{2}-\delta)^2+t^2\big)^{-1/2}$, whence 
\begin{equation}\label{gamma1}
\int_{\gamma_2}\|(z-x)^{-1}\|^2d|z|\leq \int_{-N}^N \frac{1}{(\frac{1}{2}-\delta)^2+t^2}dt\leq \int_{-\infty}^\infty \frac{1}{(\frac{1}{2}-\delta)^2+t^2}dt=\frac{\pi}{\frac{1}{2}-\delta}.
\end{equation}
Combining lines \eqref{int com}, \eqref{gamma2}, and \eqref{gamma1} we get 
$$
\|[\chi(x),c]\|\leq\frac{\|[c,x]\|}{2\pi}\Bigg(\frac{4N}{(N-\|x\|)^2}+\frac{\pi}{\frac{1}{2}-\delta}\Bigg).
$$
Letting $N\to\infty$ gives $\|[\chi(x),c]\|\leq \frac{\|[c,x]\|}{1-2\delta}$, which is the claimed estimate.
\end{proof}

The following lemma is our key application of Lemma \ref{amen lem}.

\begin{lemma}\label{amen conse}
Let $\epsilon\in (0,1)$.  Let $B$ be a separable $C^*$-algebra, and let $A$ be a separable, unital, nuclear $C^*$-algebra.  Let $\pi:A\to \LSB$ be a large representation (see Definition \ref{large rep}), and use this to identify $A$ with a $C^*$-subalgebra of $\LSB$.

Let $X$ be a finite subset of $A_1$, and let $(Y,\delta)$ be an element of the set $\mathcal{X}_A$ of Definition \ref{dir set 0} such that $(X,\epsilon)\leq (Y,\delta)$.  Then there exists a finite subset $Z$ of $A_1$ containing $X$ and a homomorphism
$$
\phi_*:KK_{\epsilon/8}(Z,B)\to KK_\delta(Y,B)
$$
such that the following diagram
$$
\xymatrix{ KK_{\epsilon/8}(Z,B) \ar[d]_-{\phi_*} \ar[dr] & \\ KK_\delta(Y,B) \ar[r] & KK_\epsilon(X,B) }
$$
(where the unlabeled maps are forget control maps as in Definition \ref{dir set}) commutes.
\end{lemma}

\begin{proof}
Let $X$, $Y$, and $\delta$ be as in the statement.  If $\delta\geq \epsilon/8$, we may just take $Z=Y$ and $\phi_*$ the forget control map.  Assume then that $\delta<\epsilon/8$.  According to Lemma \ref{amen lem} there exists contractions $a_1,...,a_n\in A$ and $t_1,...,t_n\in [0,1]$ such that $\sum_{i=1}^n t_i=1$, such that 
$$
\Bigg\|1_A-\sum_{i=1}^n t_i a_ia_i^*\Bigg\|_A<\delta /4,
$$
and such that for all $y\in Y$ and $b$ in the unit ball of $\LB$,
\begin{equation}\label{amen com inq}
\Bigg\|y\Big(\sum_{i=1}^n t_ia_ib a_i^*\Big)-\Big(\sum_{i=1}^n t_ia_ib a_i^*\Big)y\Bigg\|_{\LSB}<\delta /4 .
\end{equation}
We set $Z:=X\cup\{a_1^*,...,a_n^*\}$, and claim this works.  

Let $p\in \mathcal{P}_{\epsilon/8}(Z,B)$, let $e\in \LB$ be the neutral projection (see Definition \ref{np even}), and define 
$$
\alpha(p):=\sum_{i=1}^n t_ia_ipa_i^*+\Bigg(e-\sum_{i=1}^n t_ia_i ea_i^*\Bigg)\in \LB.
$$
As the representation is large, we may use the fixed isomorphism $\ell^2\otimes B\cong \C^2\otimes \ell^2\otimes B$ to identify $\LB$ with $M_2(\LB)$ and have that with respect to this identification, operators in $A$ are diagonal matrices, and $e={\tiny\begin{pmatrix} 1 & 0 \\ 0 & 0 \end{pmatrix}}$.  In particular, $e$ commutes with all the $a_i$, and so we have 
\begin{align}\label{p alpha p}
\|p-\alpha(p)\| & \leq \Bigg\|\Big(1-\sum_{i=1}^n t_ia_ia_i^*\Big)p\Bigg\|+\sum_{i=1}^n t_i \|a_i[p,a_i^*]\|+\Bigg\|\Big(1-\sum_{i=1}^n t_ia_ia_i^*\Big)e\Bigg\| 
\nonumber \\ & <\frac{\delta}{4}+\frac{\epsilon}{8}+\frac{\delta}{4}.
\end{align}
As $\delta< \epsilon/8$ and as $\epsilon<1$, we see that $\|p-\alpha(p)\|< \frac{1}{4}$.   As $p$ is a projection, Lemma \ref{spec lem} implies that 
\begin{equation}\label{spec contain}
\text{spectrum}(\alpha(p))\cap (1/4,3/4)=\varnothing.
\end{equation}
Let $\chi$ be the characteristic function of $(1/2,\infty)$, so $\chi$ is continuous on the spectrum of $\alpha(p)$ and we may define $\phi(p):=\chi(\alpha(p))$.  The rest of the proof will be spent showing that the formula $[p]\mapsto [\phi(p)]$ defines a homomorphism 
$$
\phi_*:KK_{\epsilon/6}(Z,B)\to KK_\delta(Y,B)
$$
with the claimed properties.

We first claim that if $p\in \mathcal{P}_{\epsilon/8}(Z,B)$, then $\phi(p)$ is in $\mathcal{P}_\delta(Y,B)$.  Note first that  
$$
\alpha(p)-e=\sum_{i=1}^n t_ia_i(p-e)a_i^*,
$$
which is in $\K_B$.  As $\K_B$ is an ideal in $\LB$, it follows $f(\alpha(p))-f(e)$ is in $\K_B$ for any polynomial $f$.  Letting $(f_n)$ be a sequence of polynomials that converges uniform to $\chi$ on the spectrum of $\alpha(p)$ and letting $n\to\infty$, we see that $\chi(\alpha(p))-e$ is in $\K_B$.  Let now $y\in Y$ and apply the inequality in line \eqref{amen com inq} once with $b=p$ and once with $b=e$ (and use that $[e,y]=0$) to deduce that  
\begin{equation}\label{alpha p com}
\|[\alpha(p),y]\|<\delta/2.
\end{equation}
Lines \eqref{alpha p com}, \eqref{spec contain}, and Lemma \ref{proj com} imply that $\|[\chi(\alpha(p)),y]\|<\delta$, completing the proof that $\phi(p)$ is an element of $\mathcal{P}_\delta(Y,B)$.  Moreover, it is straightforward to see that the assignment 
$$
\mathcal{P}_{\epsilon/8}(Z,B)\to\mathcal{P}_\delta(Y,B),\quad p\mapsto \phi(p)
$$ 
takes homotopies to homotopies and Cuntz sums to Cuntz sums.  Hence we do indeed get a well-defined homomorphism 
$$
\phi_*:KK_{\epsilon/8}(Z,B)\to KK_\delta(Y,B),\quad [p]\mapsto [\phi(p)]
$$
as claimed.

It remains to show that the diagram 
$$
\xymatrix{ KK_{\epsilon/8}(Z,B) \ar[d]_-{\phi_*} \ar[dr] & \\ KK_\delta(Y,B) \ar[r] & KK_\epsilon(X,B) }
$$
commutes.  For this, let $p\in \mathcal{P}_{\epsilon/8}(Z,B)$ represent a class in $KK_{\epsilon/8}(Z,B)$, and for $t\in [0,1]$, define $p_t:=(1-t)p+t\alpha(p)$.  Then by line \eqref{p alpha p}, we have that $\|p_t-p\|<\frac{\epsilon}{8}+\frac{\delta}{2}<\frac{1}{4}$ for all $t\in [0,1]$, so in particular 
\begin{equation}\label{spec contain 2}
\text{spectrum}(p_t)\cap (1/3,3/4)=\varnothing \quad \text{for all} \quad t\in [0,1].
\end{equation}
Hence $\chi(p_t)$ is a well-defined projection for all $t\in [0,1]$.  We claim that $\chi(p_t)$ is an element of $\mathcal{P}_\epsilon(X,B)$ for all $t\in [0,1]$; as $\chi(p_1)=\chi(\alpha(p))$ and $\chi(p_0)=p$, this will complete the proof.

For this last claim, note first that $p_t-e\in \K_B$ for all $t\in [0,1]$, whence (analogously to the case of $\chi(\alpha(p))$ argued above) $\chi(p_t)-e\in \K_B$ for all $t\in [0,1]$.  Morever, for all $z\in Z$, 
$$
\|[p_t,z]\|\leq \|[p_t-p,z]\|+\|[p,z]\|< 2\Big(\frac{\epsilon}{8}+\frac{\delta}{2}\Big)+\frac{\epsilon}{8}<\frac{\epsilon}{2},
$$
where the last inequality used that $\delta< \epsilon/8$.  Hence by line \eqref{spec contain 2} and Lemma \ref{proj com}, $\|[\chi(p_t),z]\|<\epsilon$ for all $z\in Z$, and so in particular for all $z\in X$.  This completes the proof that $\chi(p_t)\in \mathcal{P}_\epsilon(X,B)$ for all $t\in [0,1]$, so we are done.
\end{proof}

\begin{corollary}\label{uct reform 2}
Let $A$ be a separable, unital, nuclear $C^*$-algebra.  The following are equivalent:
\begin{enumerate}[(i)]
\item \label{a uct 2} $A$ satisfies the UCT.
\item \label{a vanish 2} Let $\epsilon\in (0,1)$, and let $B$ be a separable $C^*$-algebra $B$ with $K_*(B)=0$.  Let $\pi:A\to \LSB$ be a large representation, and use this to identify $A$ with a $C^*$-subalgebra of $\LSB$.   Then for any finite subset $X$ of $A_1$ there is a finite subset $Z$ of $A_1$ such that $(X,\epsilon)\leq (Z,\epsilon/8)$ in the sense of Definition \ref{dir set 0}, and so that the forget control map
$$
KK_{\epsilon/8}(Z,SB)\to KK_\epsilon(X,SB)
$$
of Definition \ref{dir set} is zero.
\end{enumerate}
\end{corollary}

\begin{proof}
Using Theorem \ref{uct reform}, it suffices to show that condition \eqref{a vanish} from that theorem implies condition \eqref{a vanish 2} from the current corollary (the converse is immediate).  Let then $\epsilon$, $B$, $\pi$, and $X$ be as in the statement.  Then condition \eqref{a vanish} from Theorem \ref{uct reform} gives $(Y,\delta)\geq (X,\epsilon)$ in the sense of Definition \ref{dir set 0} such that the associated forget control map 
$$
KK_\delta(Y,SB)\to KK_\epsilon(X,SB)
$$
of Definition \ref{dir set} is zero.  Lemma \ref{amen conse} then gives a finite subset $Z$ of $A_1$ containing $X$ and a homomorphism 
$$
\phi_*:KK_{\epsilon/8}(Z,SB)\to KK_\delta(Y,SB),\quad [p]\mapsto [\phi(p)]
$$
such that the following diagram
$$
\xymatrix{ KK_{\epsilon/8}(Z,SB) \ar[d]_-{\phi_*} \ar[dr] & \\ KK_\delta(Y,SB) \ar[r] & KK_\epsilon(X,SB) }
$$
commutes (the unlabeled arrows are forget control maps).  Hence the diagonal forget control map in the above diagram is zero, which is what we wanted to show.
\end{proof}

\section{Flexible models for controlled $KK$-theory}\label{flex sec}

In this section (as throughout), if $B$ is a separable $C^*$-algebra, then $\LB$ and $\K_B$ denote respectively the adjointable and compact operators on the standard Hilbert $B$-module $\ell^2\otimes B$.  For each $n$, we consider $\LB$ as a subalgebra of $M_n(\LB)$ via the ``diagonal inclusion'' $\LB=1_{M_n}\otimes \LB\subseteq M_n\otimes \LB=M_n(\LB)$. 

Our goal in this section is to give flexible models for controlled $KK$-theory that will be useful for computations.  Contrary to the usual conventions of $C^*$-algebra $K$-theory, we base our new even and odd groups on idempotents and invertibles rather than projections and unitaries.  The extra flexibility this allows is very useful for computations.  The main reason for not writing the whole paper using the more flexible model is that we previously established Theorem \ref{milnor ses} in \cite{Willett:2020aa} using the version of controlled $KK$-theory from Definition \ref{basic con kk} above, so need to use that model where we are directly applying Theorem \ref{milnor ses}.  Moreover, we need the results from Section \ref{tech sec} in the current paper (which are also independently needed in Section \ref{bound sec}) to relate the two models.

\subsection{The even case}\label{even sec}

Our goal in this subsection is to define a variant of the controlled $KK$-theory groups of Section \ref{reform sec}, but based on idempotents rather than projections.  For the next definition, we recall that $C^+$ denotes the unitization of a $C^*$-algebra $C$, and that if $a\in M_n(C)$ and $b\in M_m(C)$ are matrices over a $C^*$-algebra, then $a\oplus b$ denotes the matrix ${\tiny \begin{pmatrix} a & 0 \\ 0 & b \end{pmatrix}}$ in $M_{n+m}(C)$.

\begin{definition}\label{alm com 0}
Let $B$ be a separable $C^*$-algebra, let $X$ be a subset\footnote{Unlike Definition \ref{basic con kk}, we do not require $X$ to be ``large'' in the sense of Definition \ref{np even}.  Essentially, largeness is needed to ensure that the sets $KK_\epsilon(X,B)$ of Definition \ref{basic con kk} are groups; we show the sets we define in Definition \ref{alm com 0} are groups by using matrix arguments and a weaker equivalence relation in this definition.} of the unit ball of $\LB$, let $\kappa\geq 1$, let $\epsilon> 0$, and let $n\in \N$.  Define $\mathcal{P}_{n,\kappa,\epsilon}(X,B)$ to be the collection of pairs $(p,q)$ of idempotents in $M_n(\KB^+)$ satisfying the following conditions:
\begin{enumerate}[(i)]
\item \label{def proj} $\|p\|\leq  \kappa$ and $\|q\|\leq \kappa$;
\item \label{def small com} $\|[p,x]\|< \epsilon$ and $\|[q,x]\|< \epsilon$ for all $x\in X$;
\item the classes $[\sigma(p)],[\sigma(q)]\in K_0(\C)$ defined by the images of $p$ and $q$ under the canonical quotient map $\sigma:M_n(\KB^+)\to M_n(\C)$ are the same.
\end{enumerate}
Define  
$$
\mathcal{P}_{\infty,\kappa,\epsilon}(X,B):=\bigsqcup_{n=1}^\infty \mathcal{P}_{n,\kappa,\epsilon}(X,B),
$$
i.e.\ $\mathcal{P}_{\infty,\kappa,\epsilon}(X,B)$ is the \emph{disjoint} union of all the sets $\mathcal{P}_{n,\kappa,\epsilon}(X,B)$.

Equip each $\mathcal{P}_{n,\kappa,\epsilon}(X,B)$ with the norm topology it inherits from $M_n(\LB)\oplus M_n(\LB)$, and equip $\mathcal{P}_{\infty,\kappa,\epsilon}(X,B)$ with the disjoint union topology.  Let $\sim$ be the equivalence relation on $\mathcal{P}_{\infty,\kappa,\epsilon}(X,B)$ generated by the following relations:
\begin{enumerate}[(i)]
\item $(p,q)\sim (p\oplus r,q\oplus r)$ for any element $(r,r)\in \mathcal{P}_{\infty,\kappa,\epsilon}(X,B)$ with both components the same;
\item $(p_1,q_1)\sim (p_2,q_2)$ whenever these elements are in the same path component of $\mathcal{P}_{\infty,\kappa,\epsilon}(X,B)$.\footnote{Equivalently, both are in the same $\mathcal{P}_{n,\kappa,\epsilon}(X,B)$, and are in the same path component of this set.}
\end{enumerate}
Define $KK^0_{\kappa,\epsilon}(X,B)$ to be equal as a set to $\mathcal{P}_{\infty,\kappa,\epsilon}(X,B)/\sim$, and provisionally define a binary operation $+$ on $KK^0_{\kappa,\epsilon}(X,B)$ by $[p_1,q_1]+[p_2,q_2]:=[p_1\oplus q_1,p_2\oplus q_2]$.
\end{definition}

The next lemma is essentially the same as \cite[Lemma A.21]{Willett:2020aa}.

\begin{lemma}\label{kk0 group}
With notation as in Definition \ref{alm com 0}, $KK^0_{\kappa,\epsilon}(X,B)$ is a well-defined abelian group with identity element the class $[0,0]$ of the zero idempotent. 
\end{lemma}

\begin{proof}
Checking directly from the definitions shows that $KK^0_{\kappa,\epsilon}(X,B)$ is a well-defined (associative) monoid with identity element the class $[0,0]$.  A standard rotation homotopy shows that $KK^0_{\kappa,\epsilon}(X,B)$ is commutative.  To complete the proof we need to show that any element $[p,q]$ has an inverse.  We claim that this is given by $[q,p]$.  Indeed, applying the rotation homotopy
$$
\Bigg(\begin{pmatrix} p & 0 \\ 0 & q \end{pmatrix}~,~\begin{pmatrix} \cos(t) & \sin(t) \\ -\sin(t) & \cos(t) \end{pmatrix}\begin{pmatrix} q & 0 \\ 0 & p \end{pmatrix} \begin{pmatrix} \cos(t) & -\sin(t) \\ \sin(t) & \cos(t) \end{pmatrix}\Bigg),\quad t\in [0,\pi/2]
$$ 
shows that $(p\oplus q,q\oplus p)\sim (p\oplus q,p\oplus q)$, and the element $(p\oplus q,p\oplus q)$ is equivalent to $(0,0)$ by definition of the equivalence relation.
\end{proof}

The following lemma gives a useful description of cycles $(p,q)\in \mathcal{P}_{\infty,\kappa,\epsilon}(X,B)$ that define the zero class in $KK^0_{\kappa,\epsilon}(X,B)$ .

\begin{lemma}\label{kk0 zero}
With notation as in Definition \ref{alm com 0}, let $(p,q)\in \mathcal{P}_{n,\kappa,\epsilon}(X,B)$, and assume that $[p,q]=0$ in $KK^0_{\kappa,\epsilon}(X,B)$.   Then there is $m\in \N$ and an element $(s,s)$ of $\mathcal{P}_{n+2m,\kappa,\epsilon}(X,B)$ such that $(p\oplus 1_m\oplus 0_m,q\oplus 1_m\oplus 0_m)$ is in the same path component of $\mathcal{P}_{n+2m,2\kappa,\epsilon}(X,B)$ as $(s,s)$.
\end{lemma}

\begin{proof}
For elements $(p_1,q_1)$ and $(p_2,q_2)$ in $\mathcal{P}_{\infty,\kappa,\epsilon}(X,B)$ let us write: $(p_1,q_1)\to (p_2,q_2)$ if $(p_2,q_2)=(p_1\oplus r,q_1\oplus r)$ for some $(r,r)\in \mathcal{P}_{\infty,\kappa,\epsilon}(X,B)$; $(p_1,q_1)\stackrel{h}{\sim} (p_2,q_2)$ if there is a path connecting these elements; and $(p_1,q_1)\leftarrow (p_2,q_2)$ if $(p_2,q_2)\to (p_1,q_1)$.  Then $[p,q]=0$ means that there is some sequence of moves from $\{\to,\leftarrow,\stackrel{h}{\sim}\}$ starting at $(p,q)$ and finishing at $(0,0)$.  It is not difficult to see the following: any time a move from $\{\to,\leftarrow,\stackrel{h}{\sim}\}$ is consecutively repeated we may replace it by a single move of the same type; any occurrence of ``$\stackrel{h}{\sim} \to$'' may be replaced by an occurrence of ``$\to \stackrel{h}{\sim}$''; any occurrence of ``$\leftarrow \stackrel{h}{\sim}$'' may be replaced by an occurrence of ``$\stackrel{h}{\sim}\leftarrow$''; any occurrence of ``$\leftarrow \to$'' or ``$\leftarrow \stackrel{h}{\sim} \to$'' may be replaced by ``$\to \stackrel{h}{\sim}\leftarrow$'' (we leave the details to the reader in each case).  Using these replacements, we see that our moves relating $(p,q)$ to $(0,0)$ may be assumed to be of the form $(p,q)\to \stackrel{h}{\sim} \leftarrow (0,0)$, or in other words that there are elements $(r,r)$ and $(t,t)$ in $\mathcal{P}_{\infty,\kappa,\epsilon}(X,B)$ such that $(p\oplus r,q\oplus r)$ is homotopic to $(t,t)$.  

To complete the proof, note then that $(p\oplus r \oplus 1-r,q\oplus r\oplus 1-r)$ is homotopic to $(t\oplus 1-r,t\oplus 1-r)$.  For $t\in [0,\pi/2]$, define  
$$
r_t:=\begin{pmatrix} r & 0 \\ 0 & 0 \end{pmatrix} +\begin{pmatrix} \cos(t) & -\sin(t) \\ \sin(t) & \cos(t)\end{pmatrix} \begin{pmatrix} 0 & 0 \\ 0 & 1-r \end{pmatrix}\begin{pmatrix} \cos(t) & \sin(t) \\ -\sin(t) & \cos(t)\end{pmatrix}
$$
so $(r_t)_{t\in [0,\pi/2]}$ is a path connecting $r\oplus 1-r$ and $1\oplus 0$.  One computes that $\|r_t\|\leq 1+\kappa\leq 2\kappa$ for all $t$, and that $\|[r_t,x]\|<\epsilon$ for all $x\in X$.  Hence with $s=t\oplus 1-r$ we get the claimed result.
\end{proof}

We will need a more general variation of Definitions \ref{dir set 0} and \ref{dir set}.

\begin{definition}\label{dir set 2}
Let $C$ be a $C^*$-algebra.  Let $\mathcal{X}'_C$ consist of all triples of the form $(X,\kappa,\epsilon)$ where $X$ is a finite subset of the unit ball of $C$, $\kappa\geq 1$, and $\epsilon>0$.   Put a partial order on $\mathcal{X}_C'$ by $(X,\kappa,\epsilon)\leq (Y,\lambda,\delta)$ if $\delta\leq \epsilon$, $\lambda\leq \kappa$ and if for all $x\in X$ there exists $y\in Y$ with $\|x-y\|\leq \frac{1}{2\lambda}(\epsilon-\delta)$.

Let now $B$ be a separable $C^*$-algebra.  Then if $(X,\kappa,\epsilon)\leq (Y,\lambda,\delta)$ in $\mathcal{X}_{\LB}'$, one checks that for each $n$ we have
\begin{equation}\label{p inclusion 2}
\mathcal{P}_{n\lambda,\delta}(Y,B)\subseteq \mathcal{P}_{n,\kappa,\epsilon}(X,B).
\end{equation}
We call the canonical map 
$$
KK^0_{\lambda,\delta}(Y,B)\to KK_{\kappa,\epsilon}^0(X,B)
$$
induced by the inclusions in line \eqref{p inclusion 2} above a \emph{forget control map}.  
\end{definition}

\subsection{The odd case}\label{odd kk sec}

Our goal in this subsection is to introduce an odd parity version of the controlled $KK$-theory groups of the previous section.  For the statement, recall that $C^+$ denotes the unitization of a $C^*$-algebra $C$.

\begin{definition}\label{alm com 1}
Let $B$ be a separable $C^*$-algebra, let $X$ be a subset of the unit ball of $\LB$, let $\kappa\geq 1$, let $\epsilon> 0$, and let $n\in \N$.  Define $\mathcal{U}_{n,\kappa,\epsilon}(X,B)$ to be the subset of those invertible elements $u$ in $M_n(\K_B^+)$ satisfying the following conditions:
\begin{enumerate}[(i)]
\item $\|u\|\leq  \kappa$ and $\|u^{-1}\|\leq \kappa$;
\item $\|[u,x]\|< \epsilon$ and $\|[u^{-1},x]\|< \epsilon$ for all $x\in X$.
\end{enumerate}
Define  
$$
\mathcal{U}_{\infty,\kappa,\epsilon}(X,B):=\bigsqcup_{n=1}^\infty \mathcal{U}_{n,\kappa,\epsilon}(X,B),
$$
i.e.\ $\mathcal{U}_{\infty,\kappa,\epsilon}(X,B)$ is the \emph{disjoint} union of all the sets $\mathcal{U}_{n,\kappa,\epsilon}(X,B)$.

Equip each $\mathcal{U}_{n,\kappa,\epsilon}(X,B)$ with the norm topology it inherits from $M_n(\LB)$, and equip $\bigsqcup_{n=1}^\infty \mathcal{U}_{n,\kappa,\epsilon}(X,B)$ with the disjoint union topology.     Define an equivalence relation on $\mathcal{U}_{\infty,\kappa,\epsilon}(X,B)$ to be generated by the following relations:
\begin{enumerate}[(i)]
\item for any $k\in \N$, if $1_k\in\mathcal{U}_{k,\kappa,\epsilon}(X,B)$ is the identity element, then $u\sim u\oplus 1_k$;
\item $u_1\sim u_2$ if both are elements of the same path component of $\mathcal{U}_{\infty,2\kappa,\epsilon}(X,B)$.\footnote{Equivalently, both are in the same $\mathcal{U}_{n,2\kappa,\epsilon}(X,B)$, and are in the same path component of this set.  Notice also the switch from $\kappa$ to $2\kappa$ here, which is needed for our proof that $KK^1_{\kappa,\epsilon}(X,B)$ is a group.}
\end{enumerate}
Define $KK^1_{\kappa,\epsilon}(X,B)$ to be $\mathcal{U}_{\infty,\kappa,\epsilon}(X,B)/\sim$, and provisionally define a binary operation $+$ on $KK^1_{\kappa,\epsilon}(X,B)$ by $[u_1]+[u_2]:=[u_1\oplus u_2]$.
\end{definition}

\begin{lemma}\label{kk1 group}
With notation as in Definition \ref{alm com 1}, $KK^1_{\kappa,\epsilon}(X,B)$ is a well-defined abelian group with identity element the class $[1_B]$ of the unit of $B$.
\end{lemma}

\begin{proof}
It is straightforward to check that $KK^1_{\kappa,\epsilon}(X,B)$ is a monoid, and the class $[1]$ is neutral by definition.  A standard rotation homotopy shows that $KK^1_{\kappa,\epsilon}(X,B)$ is commutative.  It remains to show that inverses exist.  We claim that for $u\in \mathcal{U}_{n,\kappa,\epsilon}(X,B)$, the inverse of the class $[u]$ is given by $[u^{-1}]$.  Indeed, consider the homotopy
$$
u_t:=\begin{pmatrix} u & 0 \\ 0 & 1 \end{pmatrix} \begin{pmatrix} \cos(t) & -\sin(t) \\ \sin(t) & \cos(t) \end{pmatrix} \begin{pmatrix} 1 & 0 \\ 0 & u^{-1} \end{pmatrix} \begin{pmatrix} \cos(t) & \sin(t) \\ -\sin(t) & \cos(t) \end{pmatrix}, \quad t\in [0,\pi/2].
$$
This connects $u\oplus u^{-1}$ and $1_{2k}$, so it suffices to show that this passes through $\mathcal{U}_{2n,2\kappa,\epsilon}(X,B)$.  For the commutator condition, we compute that for $a\in X$ and $t\in [0,2\pi]$
$$
[a,u_t]=\begin{pmatrix} [a,u] & 0 \\ 0 & [u^{-1},a]\end{pmatrix} \begin{pmatrix} \cos^2(t) & \cos(t)\sin(t) \\ \cos(t)\sin(t) & -\cos^2(t) \end{pmatrix}.
$$
The scalar matrix on the right has norm $|\cos(t)|$, and the matrix on the left has norm at most $\max\{\|[a,u]\|,\|[a,u^{-1}]\|\}< \epsilon$, so $\|[a,u_t]\|< \epsilon$.  For the norm condition, we compute that 
$$
u_t=\begin{pmatrix} u & 0 \\ 0 & -u^{-1} \end{pmatrix} \begin{pmatrix} \cos^2(t) & \cos(t)\sin(t) \\ \cos(t)\sin(t) & -\cos^2(t) \end{pmatrix} +\begin{pmatrix} \sin^2(t) & -\cos(t)\sin(t) \\ \cos(t)\sin(t) & \sin^2(t) \end{pmatrix}.
$$
The first scalar matrix appearing above has norm $|\cos(t)|$, and the second has norm $|\sin(t)|$.  We thus have that $\|u_t\|\leq \kappa|\cos(t)|+|\sin(t)|$, which is at most\footnote{We suspect the optimal estimate is $\kappa$ -- this is the case if $u$ is normal, for example -- but were unable to do better than $\sqrt{1+\kappa^2}$ in general.} $2\kappa$ as required.
\end{proof}

\begin{definition}\label{dir set 3}
Let $C$ be a $C^*$-algebra, and let $\mathcal{X}'_C$ be the directed set of Definition \ref{dir set 2} above.  Let $B$ be a separable $C^*$-algebra.  Then if $(X,\kappa,\epsilon)\leq (Y,\lambda,\delta)$ in $\mathcal{X}_{\LB}'$, one checks that for each $n$ we have
\begin{equation}\label{u inclusion 2}
\mathcal{U}_{n,\lambda,\delta}(Y,B)\subseteq \mathcal{U}_{n,\kappa,\epsilon}(X,B)
\end{equation}
for all $n$.  We call the canonical map 
$$
KK^1_{\lambda,\delta}(Y,B)\subseteq KK_{\kappa,\epsilon}^1(X,B)
$$
induced by the inclusions in line \eqref{u inclusion 2} above a \emph{forget control map}.  
\end{definition}

\section{Homotopies, similarities, and normalization}\label{tech sec}

In this section (as throughout), if $B$ is a separable $C^*$-algebra, then $\LB$ and $\K_B$ denote respectively the adjointable and compact operators on the standard Hilbert $B$-module $\ell^2\otimes B$.  For each $n$, we consider $\LB$ as a subalgebra of $M_n(\LB)$ via the ``diagonal inclusion'' $\LB=1_{M_n}\otimes \LB\subseteq M_n\otimes \LB=M_n(\LB)$. 

Our goal is to establish some technical lemmas about the controlled $KK$-groups $KK^0_{\kappa,\epsilon}(X,B)$ and $KK^1_{\kappa,\epsilon}(X,B)$ and the underlying sets of cycles $\mathcal{P}_{\infty,\kappa,\epsilon}(X,B)$ and $\mathcal{U}_{\infty,\kappa,\epsilon}(X,B)$ from Definitions \ref{alm com  0} and \ref{alm com 1} respectively.  These are all variants of standard facts from $C^*$-algebra $K$-theory, but the arguments are more involved as we need to do extra work to control commutator estimates.  Some of the material is adapted from the foundational work of Oyono-Oyono and the second author on controlled $K$-theory \cite{Oyono-Oyono:2011fk}; those authors work in the `dual' setting to us in some sense, and similar techniques are often useful.

Most of the results in this section come with explicit estimates.  We have generally not tried to get optimal estimates, but as it might be useful for future work we have tried to point out where one might expect the estimates to be optimal where this is simple to do. 

\subsection{Background on idempotents}\label{idem sec}

In this subsection we look at idempotents in $C^*$-algebras and their relationship to projections.  Most of this is well-known; nonetheless, we give proofs for the sake of completeness where we could not find a good reference.

To establish notation, let us first note that if $p\in \mathcal{B}(H)$ is an idempotent, then with respect to the decomposition $H=\text{Image}(p)\oplus \text{Image}(p)^\perp$, $p$ has a matrix representation  
\begin{equation}\label{idem rep}
p=\begin{pmatrix} 1 & a \\ 0 & 0 \end{pmatrix}
\end{equation}
for some $a\in \mathcal{B}(\text{Image}(p)^\perp,\text{Image}(p))$; conversely, any operator admitting a matrix of this form with respect to some orthogonal direct sum decomposition of the underlying Hilbert space defines an idempotent.

\begin{lemma}\label{id norm}
If $p$ is an idempotent bounded operator on a Hilbert space that is neither zero nor the identity, then $\|1-p\|=\|p\|$ and $\|p-p^*\|\leq \|p\|$.
\end{lemma}

\begin{proof}
Writing $p$ as in line \eqref{idem rep} (and using that neither $\text{Image}(p)$ nor $\text{Image}(p)^\perp$ are the zero subspace), we compute that 
\begin{equation}\label{idem norm}
\|p\|^2=\|pp^*\|=\|1+aa^*\|=1+\|a\|^2
\end{equation}
and moreover that
$$
\|1-p\|^2=\|(1-p)^*(1-p)\|=\|1+a^*a\|=1+\|a\|^2=\|p\|^2. 
$$
Looking now at $p-p^*$, we see that 
$$
(p-p^*)(p-p^*)^*=\begin{pmatrix} 0 & a \\ -a^* & 0 \end{pmatrix}\begin{pmatrix} 0 & -a \\ a^* & 0 \end{pmatrix}  =\begin{pmatrix} aa^* & 0 \\ 0 & a^*a \end{pmatrix}, 
$$
whence $\|p-p^*\|^2=\|a\|^2\leq \|p\|^2$.  
\end{proof}

\begin{corollary}\label{id norm cor}
If $\kappa\geq 1$, and $p$ is any idempotent in a $C^*$-algebra with $\|p\|\leq \kappa$, then $\|1-p\|\leq \kappa$, $\|p-p^*\|\leq \kappa$, and $\|2p-1\|\leq 2\kappa$.
\end{corollary}

\begin{proof}
The estimates for $\|1-p\|$ and $\|p-p^*\|$ are immediate from Lemma \ref{id norm} (and direct checks for the degenerate cases $p=0$ and $p=1$).  The estimate for $2p-1$ follows as $2p-1=p-(1-p)$.
\end{proof}

It will be convenient to formalize a standard construction in $C^*$-algebra $K$-theory for turning idempotents into projections (compare for example \cite[Proposition 4.6.2]{Blackadar:1998yq}).

\begin{definition}\label{idem to proj}
Let $p$ be an idempotent in a $C^*$-algebra $C$.  Define $z:=1+(p-p^*)(p^*-p)\in C^+$, and note that $z\geq 1_{C^+}$ so $z$ is invertible.  Define $r:=pp^*z^{-1}$, which is an element of $C$.  We call $r$ the \emph{projection\footnote{It will be shown to be a projection in the next lemma.} associated to $p$}.
\end{definition}

\begin{remark}\label{ass pro rem}
If $C$ is a concrete $C^*$-algebra and $p$ is an idempotent with matrix representation as in line \eqref{idem rep}, then one computes that the associated projection has matrix representation 
\begin{equation}\label{ass pro rep}
r=\begin{pmatrix} 1 & 0 \\ 0 & 0 \end{pmatrix}
\end{equation}
with respect to the same decomposition of the underlying Hilbert space.  In particular, $r$ is the projection with the same image as the idempotent $p$.
\end{remark}

\begin{lemma}\label{idem to proj props}
Let $p$ be an idempotent in a $C^*$-algebra $C$, and assume that $\|p\|\leq \kappa$ for some $\kappa\geq 1$.  Let $r$ be the projection associated to $p$ as in Definition \ref{idem to proj}, and for $t\in [0,1]$ define $r_t:=(1-t)p+tr$.  Then the following hold:
\begin{enumerate}[(i)]
\item \label{ass pro} The element $r$ is a projection in $C$, and there is an invertible $u\in C^+$ such that $upu^{-1}=r$.  Moreover, $u$ and its inverse have norm at most $1+\|p\|$, and $u$ is connected to the identity through a path of invertibles such that all the invertibles in the path and all of their inverses have norm at most $1+\|p\|$.
\item \label{ip com 1} Each $r_t$ is an idempotent such that $\|r_t\|\leq \kappa$ for all $t$, and the map $t\mapsto r_t$ is $\kappa$-Lipschitz. 
\item \label{ip com 2} For any $c\in C$ and $t\in [0,1]$ we have $\|[r_t,c]\|\leq (1+2t)\|[p,c]\|+t\|[p,c^*]\|$.
\item \label{ass pro lip} The map 
$$
\{p\in C\mid p=p^2\}\to \{p\in C\mid p=p^2=p^*\}
$$
that takes an idempotent to its associated projection is $1$-Lipschitz.
\end{enumerate}
\end{lemma}

\begin{proof}
Part \eqref{ass pro} as in line \eqref{idem rep}, we may write $p={\tiny\begin{pmatrix} 1 & a \\ 0 & 0 \end{pmatrix}}$, and note as in line \eqref{idem norm} that $\|p\|=\sqrt{1+\|a\|^2}$, so in particular $\|a\|\leq \|p\|$.  Using the discussion in Remark \ref{ass pro rem} we see that $u={\tiny\begin{pmatrix} 1 & a \\ 0 & 1 \end{pmatrix}}$ satisfies $upu^{-1}=r$, and that the path $u_t={\tiny\begin{pmatrix} 1 & ta \\ 0 & 1 \end{pmatrix}}$ connects $u$ to the identity through invertibles of norm at most $1+\|ta\|\leq 1+\|p\|$.  The claims on the norms of the inverses follow as ${\tiny\begin{pmatrix} 1 & ta \\ 0 & 1 \end{pmatrix}^{-1}=\begin{pmatrix} 1 & -ta \\ 0 & 1 \end{pmatrix}}$.

(or see for example the proof of \cite[Proposition 4.6.2]{Blackadar:1998yq}).

For part \eqref{ip com 1}, we write $p$ as in line \eqref{idem rep}, note that $\|a\|\leq \kappa$, and also that $r$ has the matrix representation as in line \eqref{ass pro rep}.  This implies the claimed properties.

For part \eqref{ip com 2}, we again write $p$ as a matrix as in line \eqref{idem rep}.  Let $c\in C$, and with respect to the same decomposition of the underlying Hilbert space, let us write 
$$
c=\begin{pmatrix} c_{11} & c_{12} \\ c_{21} & c_{22} \end{pmatrix}.$$
Then one computes that 
\begin{equation}\label{[p,c]}
[p,c]=\begin{pmatrix} ac_{21} & c_{12}+ac_{22}-c_{11}a \\ -c_{21} & -c_{21}a\end{pmatrix}.
\end{equation}
As the conditional expectation that sends a matrix to its diagonal is contractive, we have 
$$
\Bigg\|\begin{pmatrix} ac_{21} & 0 \\ 0 & -c_{21}a\end{pmatrix}\Bigg\|\leq \|[p,c]\|
$$
and combining this with line \eqref{[p,c]} gives
\begin{equation}\label{mat com}
\Bigg\|\begin{pmatrix} 0 & c_{12}+ac_{22}-c_{11}a \\ -c_{21} & 0 \end{pmatrix}\Bigg\|\leq 2\|[p,c]\|.
\end{equation}
One computes that the top right entry of $[p-p^*,c]$ is $ac_{22}-c_{11}a$, whence 
$$
\|ac_{22}-c_{11}a\|\leq \|[p-p^*,c]\|\leq \|[p,c]\|+\|[p,c^*]\|.
$$
This and line \eqref{mat com} together imply that
\begin{equation}\label{mat com 2}
\Bigg\|\begin{pmatrix} 0 & c_{12} \\ -c_{21} & 0 \end{pmatrix}\Bigg\|\leq 3\|[p,c]\|+\|[p,c^*]\|.
\end{equation}
As $r$ has the matrix representation from line \eqref{ass pro rep}, the left hand side of the inequality in line \eqref{mat com 2} equals $\|[r,c]\|$, and so line \eqref{mat com 2} can be rewritten as the inequality $\|[r,c]\|\leq 3\|[p,c]\|+\|[p,c^*]\|$.  As $r_t=(1-t)p+tr$, this implies the claimed estimate.

For part \eqref{ass pro lip} we may assume that $C$ is a concrete $C^*$-algebra.  As noted in Remark \ref{ass pro rem}, the projection $r$ associated to an idempotent $p$ is then simply the orthogonal projection with the same image as $p$.  In this language, part \eqref{ass pro lip} is \cite[Chapter One, Theorem 6.35]{Kato:1980wl}.
\end{proof}

\subsection{From similarities to homotopies}

Our goal in this short subsection is to establish an analogue of the standard $K$-theoretic fact that similar idempotents are homotopic, at least up to to increasing matrix sizes.  Compare for example \cite[Proposition 4.4.1]{Blackadar:1998yq}.

\begin{proposition}\label{sim to hom}
Let $B$ be a separable $C^*$-algebra, let $X$ be a subset of the unit ball of $\LB$, and let $\kappa\geq 1$ and $\epsilon>0$.  Let $(p_0,q)$ and $(p_1,q)$ be elements of $\mathcal{P}_{n,\kappa,\epsilon}(X,B)$, and let $u\in\mathcal{U}_{n,\kappa,\epsilon}(X,B)$ be such that $up_0u^{-1}=p_1$.  Then the elements $(p_0\oplus 0_n,q\oplus 0_n)$ and $(p_1\oplus 0_n,q\oplus 0_n)$ are in the same path component of $\mathcal{P}_{2n,\kappa^3,3\kappa^2\epsilon}(X,B)$, and in particular, $(p_0,q)$ and $(p_1,q)$ define the same class in $KK^0_{\kappa^3,3\kappa^2\epsilon}(X,B)$.

The analogous statement holds with the roles of the first (``\,$p$'') and second (``\,$q$'') components reversed.
\end{proposition}

\begin{proof}
Define 
$$
v_t:=\begin{pmatrix} \cos(t) & -\sin(t) \\ \sin(t) & \cos(t) \end{pmatrix} \begin{pmatrix} 1 & 0 \\ 0 & u \end{pmatrix} \begin{pmatrix} \cos(t) & \sin(t) \\ -\sin(t) & \cos(t) \end{pmatrix} \in M_{2n}(\K_B^+).
$$
Then the path
$$
t\mapsto (v_t(p_0\oplus 0_n)v_t^{-1},q\oplus 0_n),\quad t\in [0,\pi/2]
$$
connects $(p_0\oplus 0_n,q\oplus 0_n)$ to $(p_1\oplus 0_n,q\oplus 0_n)$ through $\mathcal{P}_{2n,\kappa^3,3\kappa^2\epsilon}(X,B)$.  We leave the direct checks involved to the reader.
\end{proof}

\subsection{Normalization}

Our goal in this subsection is to show that cycles for $KK^0_{\kappa,\epsilon}(X,B)$ and $KK^1_{\kappa,\epsilon}(X,B)$ can be assumed to have prescribed ``scalar part'', at least up to some deterioration of $\kappa$ and $\epsilon$.

The following lemma is well-known without the Lipschitz condition\footnote{The constant $3$ appearing in the statement is not optimal: one can see from the proof that $3$ can be replaced with $2+\epsilon$, for any $\epsilon>0$.  We do not know what the optimal constant is.}: see for example \cite[Theorem 4.6.7]{Blackadar:1998yq} or \cite[Corollary 4.1.8]{Higson:2000bs}.

\begin{lemma}\label{lip path}
Let $L>0$.  Then if $(p_t)_{t\in [0,1]}$ is an $L$-Lipschitz path of projections in a unital $C^*$-algebra $C$, there is a $(3L)$-Lipschitz path $(u_t)_{t\in [0,1]}$ of unitaries in $C$ such that $u_0=1$, and such that $p_t=u_tp_0u_t^*$ for all $t\in [0,1]$.
\end{lemma}

We need a preliminary lemma.

\begin{lemma}\label{lip root}
Let $\eta\geq 1$, and let $C$ be a unital $C^*$-algebra.  Then the map 
$$
\{c\in C\mid c\geq \eta^{-1}\}\to C,\quad c\mapsto c^{-1/2}
$$
is $\frac{1}{2}\eta^{3/2}$-Lipschitz\footnote{The constant is optimal in some sense: this follows as the absolute value if the derivative of the function $t\mapsto t^{-1/2}$ on $[\eta^{-1},\infty)$ has maximum value $\frac{1}{2}\eta^{3/2}$.}.
\end{lemma}

\begin{proof}
For any positive real number $t$, one has
$$
t^{-1/2}=\frac{2}{\pi}\int_0^\infty (\lambda^2+t)^{-1}d\lambda,
$$
whence for any positive invertible elements $c,d\in C$ 
\begin{equation}\label{gk int}
c^{-1/2}-d^{-1/2}=\frac{2}{\pi}\int_0^\infty \big((\lambda^2+c)^{-1}-(\lambda^2+d)^{-1}\big) d\lambda.
\end{equation}
Using the formula
$$
(\lambda^2+c)^{-1}-(\lambda^2+d)^{-1}=(\lambda^2+c)^{-1}(d-c)(\lambda^2+d)^{-1}
$$
and assuming that $c\geq \eta^{-1}$ and $d\geq \eta^{-1}$, the continuous functional calculus implies that
$$
\|(\lambda^2+c)^{-1}-(\lambda^2+d)^{-1}\|\leq \|c-d\|(\lambda^2+\eta^{-1})^{-2}.
$$
This inequality and line \eqref{gk int} imply that
$$
\|c^{-1/2}-d^{-1/2}\|\leq \frac{2\|c-d\|}{\pi} \int_0^\infty (\lambda^2+\eta^{-1})^{-2}d\lambda.
$$
The integral on the right hand side equals $(\pi \eta^{3/2})/4$, whence the result.
\end{proof}

\begin{proof}[Proof of Lemma \ref{lip path}]
We first claim that it suffices to show we can choose a $\delta>0$ such that if $[t_1,t_2]$ is a sub-interval of $[0,1]$ of length at most $\delta$, and $t\mapsto p_t$ is a projection-valued $L$-Lipschitz function on $[t_1,t_2]$, then there is a unitary-valued $(3L)$-Lipschitz function $t\mapsto u_t$ on $[t_1,t_2]$ such that $u_0=1$ and $p_t=u_tp_0u_t^*$ for all $t\in [t_1,t_2]$.  Indeed, if we can do this, then let $0=t_0<t_1<\cdots <t_N=1$ be a partition of the interval $[0,1]$ such that each subinterval has length at most $\delta$, and for each $i\in \{0,...,N-1\}$ choose a unitary-valued $(3L)$-Lipschitz function $t\mapsto u^{(i)}_t$ on $[t_i,t_{i+1}]$ such that $u^{(i)}_{t_i}=1$ and $p_t=u_t^{(i)}p_{t_i}(u_t^{(i)})^*$ for all $t\in [t_i,t_{i+1}]$.  The function on $[0,1]$ defined on each subinterval $[t_i,t_{i+1}]$ by 
$$
t\mapsto u^{(i)}_tu^{(i-1)}_{t_{i}}u^{(i-2)}_{t_{i-1}}\cdots u^{(0)}_{t_1}
$$
then has the right properties to establish the lemma.

Let us then establish the statement in the claim.  Let $\epsilon>0$ be small enough that $(1-(2+\epsilon)\epsilon)^{-1/2}+(1+\epsilon)^2(1-(2+\epsilon)\epsilon)^{-3/2}\leq 3$, and let $\delta>0$ be such that if $t,s\in [0,1]$ satisfy $|t-s|\leq \delta$, then $\|p_s-p_t\|<\epsilon$.  Let $[t_1,t_2]$ be an interval of length at most $\delta$.  For $t\in [t_1,t_2]$, define $x_t:=p_tp_{t_1}+(1-p_t)(1-p_{t_1})$ and note that
$$
\|x_t-1\|=\|(2p_t-1)(p_{t_1}-p_t)\|\leq \|2p_t-1\|\|p_{t_1}-p_t\|< \epsilon,
$$
and so each $x_t$ is invertible, $\|x_t\|< 1+\epsilon$, and also $\|x_t^{-1}\|< (1-\epsilon)^{-1}$ by the Neumann series formula for the inverse.  One computes that $x_tp_{t_1}=p_tp_{t_1}=p_tx_t$, and so $x_tp_{t_1}x_t^{-1}=p_t$.  Moreover, $p_{t_1}x_t^*=x_t^*p_t$, and so $p_{t_1}x_t^*x_t=x_t^*p_tx_t=x_t^*x_tp_{t_1}$, i.e.\ $x_t^*x_t$ commutes with $p_{t_1}$.  If we define $w_t:=x_t(x_t^*x_t)^{-1/2}$, we have that $w_t$ is unitary and moreover  
$$
w_tp_{t_1}w_t^{-1}=x_t(x_t^*x_t)^{-1/2}p_{t_1}(x_t^*x_t)^{1/2}x_t^{-1}=x_tp_{t_1}x_t^{-1}=p_t.
$$
It remains to show that the path defined on $[t_1,t_2]$ by $t\mapsto w_t$ is $(3L)$-Lipschitz.

We first note that for $s,t\in [t_1,t_2]$, we have that 
\begin{equation}\label{xs and xt}
\|x_s-x_t\|=\|(p_t-p_s)(2p_{t_1}-1)\|\leq \|p_t-p_s\|\leq L|s-t|
\end{equation}
by assumption that $(p_t)$ is $L$-Lipschitz.  Using that $\|x_t\|< 1+\epsilon$, this implies that for any $s,t\in [t_1,t_2]$
$$
\|x_t^*x_t-x_s^*x_s\|\leq \|x_t^*-x_s^*\|\|x_t\|+\|x_s^*\|\|x_t-x_s\|< 2(1+\epsilon)L|s-t|.
$$ 
Moreover, $\|1-x_t^*x_t\|< (2+\epsilon)\epsilon$, whence $1-(2+\epsilon)\epsilon\leq x_t^*x_t$ and so in particular
\begin{equation}\label{xtxt}
\|(x_t^*x_t)^{-1/2}\|\leq  (1-(2+\epsilon)\epsilon)^{-1/2} \quad \text{for all} \quad t\in [t_1,t_2].
\end{equation}
Hence moreover Lemma \ref{lip root} (with $\eta=(1-(2+\epsilon))^{-1}$) implies that for any $s,t\in [t_1,t_2]$ 
\begin{equation}\label{xxs and xxt}
\|(x_t^*x_t)^{-1/2}-(x_s^*x_s)^{-1/2}\|\leq (1-(2+\epsilon)\epsilon)^{-3/2}(1+\epsilon)L|s-t|.
\end{equation}
Lines \eqref{xs and xt}, \eqref{xxs and xxt}, and \eqref{xtxt} combined with the fact that $\|x_t\|< 1+\epsilon$ for all $t\in [t_1,t_2]$ implies that for any $s,t\in [t_1,t_2]$
\begin{align*}
\|w_t-w_s\|\leq &  \|x_t-x_s\|\|(x_t^*x_t)^{-1/2}\|+\|x_s\|\|(x_t^*x_t)^{-1/2}-(x_s^*x_s)^{-1/2}\| \\
\leq & (1-(2+\epsilon)\epsilon)^{-1/2}L|s-t|+(1+\epsilon)^2(1-(2+\epsilon)\epsilon)^{-3/2}L|s-t|
\end{align*}
which implies the desired estimate by choice of $\epsilon$.
\end{proof}

For the statement of the next definition, recall that for $l\in \{1,...,n\}$, we let $1_l\in M_n(\C)$ be the rank $l$ projection with $l$ ones in the top-left part of the diagonal and zeros elsewhere.

\begin{definition}\label{nize proj}
With notation as in Definition \ref{alm com 0}, define 
$$
\mathcal{P}^1_{n,\kappa,\epsilon}(X,B):=\left\{\begin{array}{l|l} (p,q)\in \mathcal{P}_{n,\kappa,\epsilon}(X,B) &  \exists l\in \N \text{ such that }  (p,q)-(1_l,1_l) \\ & \text{is in } M_n(\K_B)\oplus M_n(\K_B) \end{array}\right\}.
$$
Define $\mathcal{P}^1_{\infty,\kappa,\epsilon}(X,B)$ to be the disjoint union of these sets as $n$ ranges over $\N$.
\end{definition}

Here is the first of our main goals for this subsection: it allows control of the ``scalar part'' of cycles for $KK^0_{\kappa,\epsilon}(X,B)$.

\begin{proposition}\label{proj same image}
Let $B$ be a separable $C^*$-algebra.  Let $X$ be a self-adjoint\footnote{We mean here that $X=X^*$, not the stronger assumption that every $x\in X$ is self-adjoint.} subset of the unit ball of $\LB$, let $\epsilon> 0$, let $\kappa\geq 1$, and let $n\in \N$.
\begin{enumerate}[(i)]
\item \label{same np} Any element $\mathcal{P}_{n,\kappa,\epsilon}(X,B)$ is in the same path component of  $\mathcal{P}_{n,4\kappa^3,\epsilon}(X,B)$ as an element of $\mathcal{P}^1_{n,4\kappa^3,\epsilon}(X,B)$\footnote{If $\kappa=1$, one can replace $4\kappa^3$ with $1$ in the statement: we leave the details to the reader.}.
\item \label{hom np}  If two elements $(p_0,q_0)$ and $(p_1,q_1)$ of $\mathcal{P}^1_{n,\kappa,\epsilon}(X,B)$ are connected by a path in $\mathcal{P}_{n,\kappa,\epsilon}(X,B)$, then they are connected by a path in $\mathcal{P}^1_{n,\kappa,4\epsilon}(X,B)$.  Moreover, if $L\geq 1$ is such that there is an $L$-Lipschitz path in $\mathcal{P}_{n,\kappa,\epsilon}(X,B)$ connecting $(p_0,q_0)$ and $(p_1,q_1)$, then there is a $(20\kappa L)$-Lipschitz  path in $\mathcal{P}^1_{n,\kappa,4\epsilon}(X,B)$ connecting $(p_0,q_0)$ and $(p_1,q_1)$.
\end{enumerate}
\end{proposition}

\begin{proof}[Proof of Proposition \ref{proj same image}]
For part \eqref{same np}, assume that $(p,q)$ is an element of $\mathcal{P}_{n,\kappa,\epsilon}(X,B)$.  Hence by definition of $\mathcal{P}_{n,\kappa,\epsilon}(X,B)$, if $\K_B^+$ is the unitization of $\K_B$ and $\sigma:M_n(\K_B^+)\to M_n(\C)$ is the canonical quotient map then the classes $[\sigma(p)]$ and $[\sigma(q)]$ in $K_0(\C)$ are the same, so in particular the idempotents $\sigma(p)$ and $\sigma(q)$ have the same rank. Using Lemma \ref{idem to proj props} part \eqref{ass pro}, there are paths of invertibles $(u_t)_{t\in [0,1]}$ and $(v_t)_{t\in [0,1]}$ in $M_n(\C)$ and projections $r,s$ such that $u_1=v_1$ is the identity, such that $u_0ru_0^{-1}=\sigma(p)$, such that $v_0sv_0^{-1}=\sigma(q)$, and such that the norms of all the $u_t$, all the $v_t$ and their inverses are all at most $1+\kappa\leq 2\kappa$.  On the other hand, $r$ and $s$ have the same rank, whence there are paths of unitaries $(u_t)_{t\in [1,2]}$ and $(v_t)_{t\in [0,1]}$ in $M_n(\C)$ such that $u_1=v_1$ is the identity, and such that $u_2ru_2^*=1_l$, and $v_2sv_2^*=1_l$.  As scalar matrices commute with $X$, the path $((u_tpu_t^{-1},v_tqv_t^{-1}))_{t\in [0,2]}$ passes through $\mathcal{P}_{n,4\kappa^3,\epsilon}(X,B)$, and connects $(p,q)$ to an element of $\mathcal{P}^1_{n,4\kappa^3,\epsilon}(X,B)$ as required.



For part \eqref{hom np}, we just look at the statement involving Lipschitz paths; the case of general continuous paths follows (in a simpler way) from the same arguments, and is left to the reader.  Assume that $(p_0,q_0)$ and $(p_1,q_1)$ are elements of $\mathcal{P}^1_{n,\kappa,\epsilon}(X,B)$ that are connected by an $L$-Lipschitz path that passes through $\mathcal{P}_{n,\kappa.\epsilon}(X,B)$.  In particular there exists $l\in \N$ such that $\sigma(p_0)=\sigma(q_0)=1_l=\sigma(p_1)=\sigma(q_1)$.  Let $r_0$ be the projection associated to $p_0$ as in Definition \ref{idem to proj}.  As in Lemma \ref{idem to proj props}, part \eqref{ip com 1}, the path defined for $t\in [0,1]$ by $t\mapsto (1-t)p_0+tr_0$ is $\kappa$-Lipschitz and connects $p_0$ and $r_0$ through idempotents of norm at most $\kappa$.  Moreover, Lemma \ref{idem to proj props}, part \eqref{ip com 2} implies that for all $x\in X$ and all $t\in [0,1]$
$$
\|[(1-t)p_0+tr_0,x]\|\leq (1+2t)\|[p_0,x]\|+t\|[p_0,x^*]\|.
$$ 
As $X=X^*$, this implies that $\|[(1-t)p_0+tr_0,x]\|<4\epsilon$ for all $x\in X$, and all $t\in [0,1]$.  Note also that $\sigma((1-t)p_0+tr_0)=1_l$ for all $t$.  Similarly, we get $s_0$ which has the same properties with respect to $q_0$.  We have thus shown that $(p_0,q_0)$ is connected to the element $(r_0,s_0)$ via a $\kappa$-Lipschitz path in $\mathcal{P}^1_{n,\kappa,4\epsilon}(X,B)$.  Completely analogously, $(p_1,q_1)$ is connected to its associated projection $(r_1,s_1)$ via a $\kappa$-Lipschitz path in $\mathcal{P}^1_{n,\kappa,4\epsilon}(X,B)$.  Moreover, using Lemma \ref{idem to proj props}, part \eqref{ass pro lip}, we have that $(r_0,s_0)$ and $(r_1,s_1)$ are connected by an $L$-Lipschitz path of projections in $\mathcal{P}_{n,1,4\epsilon}(X,B)$, say $((r_t,s_t))_{t\in [0,1]}$.

Now, consider the path $(\sigma(r_t),\sigma(s_t))_{t\in [0,1]}$ in $M_n(\C)\oplus M_n(\C)$, which is also $L$-Lipschitz.  Lemma \ref{lip path} gives $(3L)$-Lipschitz paths $(u_t)_{t\in [0,1]}$ and $(v_t)_{t\in [0,1]}$ of unitaries in $M_n(\C)$ such that $\sigma(r_t)=u_t\sigma(r_0)u_t^*$ and $\sigma(s_t)=v_t\sigma(s_0)v_t^*$ for all $t\in [0,1]$.  The path $((u_t^*r_tu_t,v_t^*s_tv_t))_{t\in [0,1]}$ then passes through $\mathcal{P}^1_{n,1,4\epsilon}(X,B)$, is $(6L)$-Lipschitz, and connects $(r_0,s_0)$ to $(u_1^*r_1u_1,v_1^*s_1v_1)$.

Summarizing where we are, we have the following paths 
\begin{enumerate}[(i)]
\item A $\kappa$-Lipschitz path through $\mathcal{P}^1_{n,\kappa,4\epsilon}(X,B)$, parametrized by $[0,1]$, and that connects $(p_0,q_0)$ and $(r_0,s_0)$.
\item A $(6L)$-Lipschitz path through $\mathcal{P}^1_{n,1,4\epsilon}(X,B)$, parametrized by $[0,1]$, and that connects $(r_0,s_0)$ and $(u_1^*r_1u_1,v_1^*s_1v_1)$.
\item A $\kappa$-Lipschitz path through $\mathcal{P}^1_{n,\kappa,4\epsilon}(X,B)$, parametrized by $[0,1]$, and that connects $(p_1,q_1)$ and $(r_1,s_1)$.
\end{enumerate}  
We claim that there is a $2\pi$-Lipschitz path passing through $\mathcal{P}^1_{n,1,4\epsilon}(X,B)$, parametrized by $[0,1]$ and connecting $(u_1^*r_1u_1,v_1^*s_1v_1)$ and $(r_1,s_1)$.  Concatenating this new path with the three paths above (and using that $\kappa\geq 1$ and that $L\geq 1$), and rescaling the two $\kappa$-Lipschitz paths by $1/12$, the $6L$-Lipschitz path by $4/12$, and the $6\pi$-Lipschitz by $6/12$, this will give us a $(20\kappa L)$-Lipschitz path connecting $(p_0,q_0)$ and $(p_1,q_1)$ through $\mathcal{P}^1_{n,1,4\epsilon}(X,B)$, which will complete the proof.

To establish the claim note that $u_1$ commutes with $1_l$, and is therefore connected to the identity in $M_n(\C)$ via a $\pi$-Lipschitz path of unitaries that all commute with $1_l$, say $(u_t)_{t\in [1,2]}$.  Similarly, we get a $\pi$-Lipschitz path $(v_t)_{t\in [1,2]}$ with the same properties with respect to $v_1$.  The path $((u_t^*r_1u_t,v_t^*s_1v_t))_{t\in [1,2]}$ then passes through $\mathcal{P}^1_{n,1,4\epsilon}(X,B)$, is $2\pi$-Lipschitz, and connects $(u_1^*r_1u_1,v_1^*s_1v_1)$ to $(r_1,s_1)$, so we are done.
\end{proof}

We now move on to results that let us prescribe the ``scalar part'' of cycles for $KK^1$, which is much simpler.

\begin{definition}\label{nize uni}
With notation as in Definition \ref{alm com 1}, define 
$$
\mathcal{U}^1_{n,\kappa,\epsilon}(X,B):=\{u\in \mathcal{U}_{n,\kappa,\epsilon}(X,B)\mid u-1\in M_n(\K_B)\}.
$$
Define $\mathcal{U}^1_{\infty,\kappa,\epsilon}(X,B)$ to be the disjoint union of these sets as $n$ ranges over $\N$.
\end{definition}

We need a slight variant of the well-known fact that the group of invertibles in a $C^*$-algebra deform retracts onto the group of unitaries.

\begin{lemma}\label{ukappa con lem}
Let $\kappa\geq 1$, let $C$ be a unital $C^*$-algebra, and let $C^{-1}_\kappa$ be the set of invertible elements $u\in C$ such that $\|u\|\leq \kappa$ and $\|u^{-1}\|\leq \kappa$.  Then the unitary group of $C$ is a deformation retract of $C^{-1}_\kappa$.  In particular, $M_n(\C)_\kappa^{-1}$ is connected.
\end{lemma}

\begin{proof}
Let $u\in C^{-1}_\kappa$, and for $t\in [0,1/2]$ define $u_t:=u(u^*u)^{-t}$.  This is a homotopy between the identity $u\mapsto u_0$ on $C^{-1}_\kappa$ and the map $u\mapsto u_{1/2}$; the latter is a retraction of $C^{-1}_\kappa$ onto the unitary group of $C$, giving the first part.  In particular, it follows that $C_\kappa^{-1}$ is connected if and only if $C^{-1}_1$ is connected; as the unitary group of $M_n(\C)$ is connected, this gives the last statement.
\end{proof}

\begin{proposition}\label{uni lem}
Let $B$ be a separable $C^*$-space, let $X$ be a subset of the unit ball of $\LB$, let $\epsilon> 0$, let $\kappa\geq 1$, and let $n\in \N$.
\begin{enumerate}[(i)]
\item \label{same nu} Any element $v\in \mathcal{U}_{n,\kappa,\epsilon}(X,B)$ is in the same path component of  $\mathcal{U}_{n,\kappa^2,\kappa\epsilon}(X,B)$ as an element of $\mathcal{U}^1_{n,\kappa^2,\kappa\epsilon}(X,B)$.
\item \label{hom nu} If two elements $v_0,v_1\in \mathcal{U}^1_{n,\kappa,\epsilon}(X,B)$ are in the same path component of $\mathcal{U}_{n,\kappa,\epsilon}(X,B)$, then they are in the same path component of $\mathcal{U}^1_{n,\kappa^2,\kappa\epsilon}(X,B)$.
\end{enumerate}
\end{proposition}

\begin{proof}
For part \eqref{same nu}, let $\K_B^+$ be the unitization of $\K_B$, let $\sigma:M_n(\K_B^+)\to M_n(\C)$ be the canonical quotient map, and set $w=\sigma(u^{-1})$.  Using Lemma \ref{ukappa con lem}, there is a path $(w_t)_{t\in [0,1]}$ of invertibles connecting $w=w_1$ to the identity and all with norm at most $\kappa$.  Then the path $(w_tv)_{t\in [0,1]}$ is in $\mathcal{U}_{n,\kappa^2,\kappa\epsilon}(X,B)$ and connects $v$ to the element $u:=w_1v$, which satisfies $\sigma(u)=1$, and so $1-u\in M_n(\K_B)$.

For part \eqref{hom nu}, let $(v_t)_{t\in [0,1]}$ be a path in $\mathcal{U}_{n,\kappa,\epsilon}(X,B)$ connecting $v_0$ and $v_1$.  Let $w_t=\sigma(v_t^{-1})$, and note that $w_0=w_1=1$.  Moreover, $\|w_t\|\leq \kappa$ for all $t$.  Then $u_t:=w_tv_t$ is a path connecting $v_0$ and $v_1$ in $\mathcal{U}^1_{n,\kappa^2,\kappa\epsilon}(X,B)$ as required.
\end{proof}

\subsection{From homotopies to similarities}

Our goal in this subsection is to establish a controlled variant of the fact that homotopic idempotents are similar: compare for example \cite[Proposition 4.3.2]{Blackadar:1998yq}.  This requires some work, as we need to control the ``speed'' of the homotopy in order to control the commutator estimates for the invertible element appearing in the similarity.  The final target is Proposition \ref{hom to sim} below; the other results build up to it.

\begin{lemma}\label{lip lem}
Let $\kappa\geq 1$, and let $p_0$ and $p_1$ be idempotents in a $C^*$-algebra $C$ with norm at most $\kappa$, and such that $\|p_0-p_1\|\leq 1/(12\kappa^2)$.  Then there is a path $(p_t)_{t\in [0,1]}$ of idempotents connecting $p_0$ and $p_1$, and with the following properties:
\begin{enumerate}[(i)]
\item \label{norm} each $p_t$ is an idempotent in $C$ of norm at most $2\kappa$;
\item \label{commute} for all $c\in C$ and $t\in [0,1]$, 
$$\|[c,p_t]\|\leq 21\kappa^2\max_{i=0,1}\|[c,p_i]\|;$$
\item \label{lip} the function $t\mapsto p_t$ is $1$-Lipschitz.
\end{enumerate}
\end{lemma}

\begin{proof}
For each $t\in [0,1]$, define $r_t:=(1-t)p_0+tp_1\in C$, and define $u_t:=(1-r_t)(1-p_0)+r_tp_0\in C^+$.  Corollary \ref{id norm cor} implies that $\|2p_0-1\|\leq 2\kappa$, whence 
$$
\|1-u_t\|=\|(2p_0-1)(p_0-r_t)\|\leq 2\kappa\|p_0-p_1\|\leq 1/6
$$
In particular, $u_t$ is invertible, $\|u_t\|\leq 7/6$, and $\|u_t^{-1}\|\leq 6/5$ by the Neumann series formula of the inverse.  Define $p_t:=u_tp_0u_t^{-1}$, which is an idempotent in $C$.  We claim that the path $(p_t)_{t\in [0,1]}$ has the desired properties.  Note first that $r_0=p_0$, whence $u_0=1$, and so the path $(p_t)_{t\in [0,1]}$ does start at the original $p_0$.  On the other hand, $u_1p_0=r_1p_0=p_1p_0=p_1u_1$, whence $u_1p_0u_1^{-1}=p_1$.  Thus the path $(p_t)$ does connect $p_0$ and $p_1$. 

For part \eqref{norm}, note that as $u_tp_0=r_tp_0$, we get 
$$
\|p_t\|=\|r_tp_0u_t^{-1}\|\leq \|(r_t-p_0)p_0u_t^{-1}\|+\|p_0u_t^{-1}\|\leq \frac{1}{12\kappa^2}\kappa\frac{6}{5}+\kappa\frac{6}{5}\leq 2\kappa.
$$
For part \eqref{commute}, let $\delta=\max_{i=0,1}\|[c,p_i]\|$.  We compute using the identity $1-u_t=(2p_0-1)(p_0-r_t)$ that 
\begin{align*}
\|[u_t,c]\|=\|[1-u_t,c]\| & \leq \|[2p_0-1,c]\|\|p_0-r_t\|+\|2p_0-1\|\|[p_0-r_t,c]\| \\ & \leq 2\|[p_0,c]\|\|p_0-r_t\|+\|2p_0-1\|(\|[p_0,c]\|+\|[r_t,c]\|).
\end{align*}
Using that $\|2p_0-1\|\leq 2\kappa$ again, this implies that
$$
\|[u_t,c]\|\leq 2\delta\frac{1}{12\kappa^2}+2\kappa\cdot 2\delta= \Big(4\kappa+\frac{1}{6\kappa^2}\Big)\delta.
$$
Hence also 
$$
\|[u_t^{-1},c]\|=\|u_t^{-1}[c,u_t]u_t^{-1}\|\leq \frac{36}{25}\Big(4\kappa+\frac{1}{6\kappa^2}\Big)\delta\|c\|
$$
and so 
\begin{align*}
\|[p_t,c]\| & =\|[u_tp_0u_t^{-1},c]\| \\ & \leq \|[u_t,c]\|\|p_0\|\|u_t^{-1}\|+\|u_t\|\|[p_0,c]\|\|u_t^{-1}\|+\|u_t\|\|p_0\|\|[u_t^{-1},c]\| \\
&  \leq \Big(4\kappa+\frac{1}{6\kappa^2}\Big)\delta\kappa\frac{6}{5}+\frac{7}{5}\delta+\frac{7}{6}\kappa \frac{36}{25}\Big(4\kappa+\frac{1}{6\kappa^2}\Big)\delta \\
& \leq 21\kappa^2\delta
\end{align*}
as claimed.  Finally, for part \eqref{lip}, we again use that $\|2p_0-1\|\leq 2\kappa$ to compute that for any $s,t\in [0,1]$, 
\begin{align*}
\|u_s-u_t\| & =\|(2p_0-1)(r_s-r_t)\|\leq \|2p_0-1\||s-t|\|p_0-p_1\|\leq2\kappa|s-t|\frac{1}{12\kappa^2} \\ & = \frac{1}{6\kappa}|s-t|
\end{align*}
and so 
$$
\|u_s^{-1}-u_t^{-1}\|=\|u_t^{-1}(u_t-u_s)u_s^{-1}\|\leq \frac{36}{25}\frac{1}{6\kappa}|s-t|=\frac{6}{25\kappa}|s-t|.
$$
Hence 
\begin{align*}
\|p_t-p_s\| & \leq \|(u_t-u_s)p_0u_t^{-1}\|+\|u_sp_0(u_t^{-1}-u_s^{-1})\| \\ & \leq \frac{1}{6\kappa}|s-t|\kappa\frac{6}{5}+\frac{7}{6}\kappa\frac{6}{25\kappa}|s-t| \\
& \leq |s-t|
\end{align*}
as claimed.
\end{proof}

The next lemma gives universal control over the ``speed'' of a homotopy between idempotents (at the price of moving to larger matrices).  The basic idea is not new: see for example \cite[Proposition 1.31]{Oyono-Oyono:2011fk}.  We give a complete proof, however, as we need to incorporate commutator estimates and work with idempotents rather than projections.

\begin{lemma}\label{short hom}
Let $B$ be a separable $C^*$-algebra, let $X$ be a subset of the unit ball of $\LB$, let $\epsilon> 0$, and let $n\in \N$.  Let $(p_0,q_0)$ and $(p_1,q_1)$ be elements of the same path component of $\mathcal{P}_{n,\kappa,\epsilon}(X,B)$.  Then there is $k\in \N$ and a homotopy $((r_t,s_t))_{t\in [0,1]}$ in $\mathcal{P}_{(2k+1)n,2\kappa,21\kappa^2\epsilon}(X,B)$ such that $(r_i,s_i)=(p_i\oplus 1_{nk}\oplus 0_{nk},q_i\oplus 1_{nk}\oplus 0_{nk})$ for $i\in \{0,1\}$, and such that the map $t\mapsto (r_t,s_t)$ is $(16\kappa)$-Lipschitz.
\end{lemma}

\begin{proof}
Let $((p_t,q_t))_{t\in [0,1]}$ be an arbitrary homotopy in $\mathcal{P}_{n,\kappa,\epsilon}(X,B)$ connecting $(p_0,q_0)$ and $(p_1,q_1)$.  Let $\delta>0$ be such that if $s,t\in [0,1]$ satisfy $|s-t|\leq \delta$, then $\|p_s-p_t\|\leq 1/(12\kappa^2)$ and $\|q_s-q_t\|\leq 1/(12\kappa^2)$.  Let $0=t_0<t_1<....<t_k=1$ be a sequence of points in $[0,1]$ such that $t_{i+1}-t_i\leq \delta$ for all $i$.  We claim that this $k$ works, and to show this we build an appropriate homotopy by concatenating the various steps below.  
\begin{enumerate}[(i)]
\item Connect $(p_0\oplus 1_{nk}\oplus 0_{nk},q_0\oplus 1_{nk}\oplus 0_{nk})$ to 
$$
\big(\,p_0 \oplus  \underbrace{(1_n\oplus 0_n)\oplus \cdots \oplus (1_n\oplus 0_n)}_{k \text{ times}}\,,\,q_0 \oplus  \underbrace{(1_n\oplus 0_n)\oplus \cdots \oplus (1_n\oplus 0_n)}_{k \text{ times}}\,\big)
$$
via a $2$-Lipschitz rotation homotopy parametrized by $[0,\pi/2]$ and passing through $\mathcal{P}_{(2k+1)n,\kappa,\epsilon}(X,B)$.
\item \label{split step} In the $i^\text{th}$ `block' $1_n\oplus 0_n$, use the homotopy 
$$
\begin{pmatrix} 1-p_{t_i} & 0 \\ 0 & 0 \end{pmatrix} +\begin{pmatrix} \cos(t) & -\sin(t) \\ \sin(t) & \cos(t) \end{pmatrix} \begin{pmatrix} 0 & 0 \\ 0 & p_{t_i} \end{pmatrix} \begin{pmatrix} \cos(t) & \sin(t) \\ -\sin(t) & \cos(t) \end{pmatrix}
$$
(parametrized by $t\in [0,\pi/2]$) to connect $1_n\oplus 0_n$ to $1-p_{t_i}\oplus p_{t_i}$, and similarly for $q$.  In order to compute commutator estimates, note that rearranging gives that the homotopy above is the same as  
$$
\begin{pmatrix} 1 & 0 \\ 0 & 0 \end{pmatrix}+
\begin{pmatrix} p_{t_i} & 0 \\ 0 & p_{t_i} \end{pmatrix} \begin{pmatrix} -\cos^2(t) & -\sin(t)\cos(t) \\ -\sin(t)\cos(t) & \cos^2(t)\end{pmatrix},\quad t\in [0,\pi/2].
$$
The scalar matrix appearing on the right above has norm $|\cos(t)|$, whence every element in this homotopy has norm at most $2\kappa$.  Hence our homotopy connects the result of the previous stage to 
$$
(p_0 \oplus  1-p_{t_1}\oplus p_{t_1} \oplus \cdots \oplus 1-p_{t_k} \oplus p_{t_k},q_0 \oplus  1-q_{t_1}\oplus q_{t_1} \oplus \cdots \oplus 1-q_{t_k} \oplus q_{t_k})
$$
through $\mathcal{P}_{(2k+1)n,2\kappa,\epsilon}(X,B)$, and is $2\kappa$-Lipschitz.
\item From Corollary \ref{id norm cor}, each idempotent $1-p_{t_i}$ has norm at most $\kappa$.  For each $i\in \{1,...,k\}$, using that $\|(1-p_{t_i})-(1-p_{t_{i-1}})\|\leq 1/(12\kappa^2)$, Lemma \ref{lip lem} gives a path of idempotents connecting $1-p_{t_i}$ and $1-p_{t_{i-1}}$ and with the following properties: it is $1$-Lipschitz; it consists of idempotents of norm at most $2\kappa$; each idempotent $r$ in the path satisfies $\|[r,x]\|\leq 21\kappa^2\epsilon$ for all $x\in X$.  We get similar paths with respect to the elements $1-q_{t_i}$, and use these paths to connect the result of the previous stage to 
$$
(p_0 \oplus  1-p_{t_0}\oplus p_{t_1} \oplus \cdots \oplus 1-p_{t_{k-1}} \oplus p_{t_k},q_0 \oplus  1-q_{t_0}\oplus q_{t_1} \oplus \cdots \oplus 1-q_{t_{k-1}} \oplus q_{t_k}).
$$
via a $1$-Lipschitz path in $\mathcal{P}_{(2k+1)n,2\kappa,21\kappa^2\epsilon}(X,B)$.
\item Use an analog of the homotopy in step \eqref{split step} in each block of the form $p_{t_i}\oplus 1-p_{t_i}$ (and similarly for $q$) to connect the result of the previous stage to 
$$
\big(\underbrace{(1_n\oplus 0_n)\oplus \cdots \oplus (1_n\oplus 0_n)}_{k \text{ times}}\oplus p_{t_k},\underbrace{(1_n\oplus 0_n)\oplus \cdots \oplus (1_n\oplus 0_n)}_{k \text{ times}}\oplus q_{t_k}\big).
$$
This passes through $\mathcal{P}_{(2k+1)n,2\kappa,\epsilon}(X,B)$, and is $2\kappa$-Lipschitz.
\item Finally, noting that $p_{t_k}=p_1$ and $q_{t_k}=q_1$, use a rotation homotopy parametrized by $[0,\pi/2]$ to connect the result of the previous stage to $(p_1\oplus 1_{nk}\oplus 0_{nk},q_1\oplus 1_{nk}\oplus 0_{nk})$.  This passes through $\mathcal{P}_{(2k+1)n,\kappa,\epsilon}(X,B)$ and is $2\kappa$-Lipschitz.
\end{enumerate}
Concatenating the five homotopies above gives a $2\kappa$-Lipschitz homotopy, parametrized by $[0,2\pi+1]$, that passes through $\mathcal{P}_{(2k+1)n,2\kappa,\epsilon}(X,B)$ and connects $(p_0\oplus 1_{nk}\oplus 0_{nk},q_0\oplus 1_{nk}\oplus 0_{nk})$ and $(p_1\oplus 1_{nk}\oplus 0_{nk},q_1\oplus 1_{nk}\oplus 0_{nk})$.  Reparametrizing by $[0,1]$, we get a $(16\kappa)$-Lipschitz homotopy as required.
\end{proof}

Before we get to the main result of this subsection, we give one more elementary lemma; we record it as it will be used multiple times below.

\begin{lemma}\label{com prod}
Say $x$ and $y_1,...,y_n$ are elements of a $C^*$-algebra such that $\|[x,y_i]\|\leq \delta$ and $\|y_i\|\leq m$ for all $i$.  Then if $y:=y_1y_2\cdots y_n$, we have $\|[x,y]\|\leq nm^{n-1}\delta$.
\end{lemma}

\begin{proof}
This follows from the formula 
$$
[x,y]=\sum_{i=1}^n \Big(\prod_{1\leq j<i}y_j\Big) [x,y_i]\Big(\prod_{i<j\leq n}y_j\Big),
$$
which itself follows from induction on $n$ and the usual Leibniz formula $[x,y_1y_2]=y_1[x,y_2]+[x,y_1]y_2$.
\end{proof}

Here is the main result of this subsection.  The basic idea of the proof is contained in \cite[Corollary 1.32]{Oyono-Oyono:2011fk}, but as usual we need to do more work in order to get our estimates.

\begin{proposition}\label{hom to sim}
Let $B$ be a separable $C^*$-algebra, let $X$ be a self-adjoint subset of the unit ball of $\LB$, let $\kappa\geq 1$, and let $\epsilon> 0$.  Let $M=2^{(100\kappa)^3}$.  With notation as in Definition \ref{nize proj}, let $n\in \N$, and let $(p,q)$ be in the same path component of $\mathcal{P}^1_{n,\kappa,\epsilon}(X,B)$ as an element $(r,r)$ with both entries the same.  Then there is $m\in \N$ and (with notation as in Definition \ref{nize uni}) an element $u\in \mathcal{U}^1_{n+2m,M,M\epsilon}(X,B)$ such that 
$$
u(p\oplus 1_m\oplus 0_m)u^{-1}=q\oplus 1_m\oplus 0_m.
$$ 
\end{proposition}

\begin{proof}
Let $k\in \N$ be as in the conclusion of Lemma \ref{short hom}, so there exists a $(16\kappa)$-Lipschitz homotopy in $\mathcal{P}_{(2k+1)n,2\kappa,21\kappa^2\epsilon}(X,B)$ between $(p\oplus 1_{nk}\oplus 0_{nk},q\oplus 1_{nk}\oplus 0_{nk})$ and $(r\oplus 1_{nk}\oplus 0_{nk},r\oplus 1_{nk}\oplus 0_{nk})$.  Set $m=kn$.  Proposition \ref{proj same image} gives a $(20\kappa\cdot 16\kappa)$-Lipschitz path $((p_t,q_t))_{t\in [0,1]}$ passing through $\mathcal{P}^1_{n+2m,2\kappa,84\kappa^2\epsilon}(X,B)$ that connects $(p\oplus 1_{nk}\oplus 0_{nk},q\oplus 1_{nk}\oplus 0_{nk})$ and $(r\oplus 1_{nk}\oplus 0_{nk},r\oplus 1_{nk}\oplus 0_{nk})$.  To simplify notation, note this path is $(2^{9}\kappa^2)$-Lipschitz, and that it passes through $\mathcal{P}^1_{n+2m,2\kappa,2^{7}\kappa^2\epsilon}(X,B)$

Define $N:=\lceil 2^{13}\kappa^3\rceil$ (where $\lceil y\rceil$ is the least integer at least as large as $y$), and define $t_i=i/N$ for $i\in \{0,...,N\}$.  As the path $((p_t,q_t))_{t\in [0,1]}$ is $(2^{9}\kappa^2)$-Lipschitz, for any $i\in\{1,...,N\}$, $\|p_{t_i}-p_{t_{i-1}}\|\leq (16\kappa)^{-1}$.  For $i\in \{1,...,N\}$, define $v_i:=p_{t_{i-1}}p_{t_{i}}+(1-p_{t_{i-1}})(1-p_{t_{i}})$.  As $\|p_{t_{i}}\|\leq 2\kappa$ for all $i$,  Corollary \ref{id norm cor} implies that 
\begin{equation}\label{2p-1}
\|2p_{t_{i}}-1\|\leq 4\kappa
\end{equation}
for all $i$, and so
$$
\|1-v_i\|=\|(2p_{t_{i-1}}-1)(p_{t_{i-1}}-p_{t_{i}})\|\leq 4\kappa\cdot (16\kappa)^{-1}\leq 1/2.
$$  
It follows that each $v_i$ is invertible, $\|v_i\|\leq 2$, and (by the Neumann series formula for the inverse) $\|v_i^{-1}\|\leq 2$.  Note also that as the homotopy $((p_t,q_t))_{t\in [0,1]}$ passes through $\mathcal{P}^1_{(2k+1)n,2\kappa,2^{7}\kappa^2\epsilon}(X,B)$ all the elements $p_{t_i}$ must have the same ``scalar part'' (i.e.\ the same image under the canonical map $M_{n+2m}(\K_B^+)\to M_{n+2m}(\C)$), and so the elements $v_i$ must satisfy $1-v_i\in M_{n+2m}(\K_B)$. Moreover, for $x\in X$, using line \eqref{2p-1} again we see that
\begin{align*}
\|[v_i,x]\| & =\|[v_i-1,x]\| \\ & =\|[(2p_{t_{i-1}}-1)(p_{t_{i-1}}-p_{t_{i}}),x]\| \\ 
& \leq 2\|[p_{t_{i-1}},x]\|(\|p_{t_{i-1}}\|+\|p_{t_{i}}\|)+\|2p_{t_{i-1}}-1\|(\|[p_{t_{i-1}},x]\|+\|[p_{t_{i}},x]\|) \\
& \leq  12\kappa\cdot 2^{7}\kappa^2\epsilon.
\end{align*}
Hence moreover 
$$
\|[v_i^{-1},x]\|=\|v_i^{-1}[x,v_i]v_i^{-1}\|\leq  4\cdot 12\kappa\cdot 2^{7}\kappa^2\epsilon\leq 2^{13}\kappa^3\epsilon.
$$ 
At this point we have that each $v_i$ is an element of $\mathcal{U}^1_{n+2m,2, 2^{13}\kappa^3\epsilon}$.  

Note also that $v_ip_{t_i}=p_{t_{i-1}}p_{t_{i}}=p_{t_{i-1}}v_i$, and so $v_ip_{t_i}v_i^{-1}=p_{t_{i-1}}$ for each $i$.  Define $v$ to be the product $v_1v_2\cdots v_N$, so $v$ satisfies $v^{-1}p_0v=p_1$, or in other words $v^{-1}(p\oplus 1_{m}\oplus 0_{m})v=r\oplus 1_{m}\oplus 0_{m}$.  Note that $1-v\in M_{n+2m}(\K_B)$.  As $\|v_i\|\leq 2$ and $\|v_i^{-1}\|\leq 2$ for each $i$, we have that $\|v\|\leq 2^{N}$ and similarly $\|v^{-1}\|\leq 2^{N}$.  Moreover, for any $x\in X$, Lemma \ref{com prod} gives $\|[v,x]\|\leq N2^{N-1}\cdot  2^{13}\kappa^3 \epsilon$ and similarly $\|[v^{-1},x]\|\leq N2^{N-1}\cdot  2^{13}\kappa^3\epsilon$.  Applying the same construction with $(q_t)$ in place of $(p_t)$, we get an invertible element $w$ such that $w^{-1}(q\oplus 1_{m}\oplus 0_{m})w=r\oplus 1_{m}\oplus 0_{m}$, such that $1-w\in M_{n+2m}(\K_B)$, such that $\|w\|\leq 2^N$, $\|w^{-1}\|\leq 2^N$, and such that $\|[w,x]\|\leq N2^{N-1}\cdot  2^{13}\kappa^3 \epsilon$ and $\|[w^{-1},x]\|\leq N2^{N-1}\cdot  2^{13}\kappa^3 \epsilon$ for all $x\in X$.  Define $u=wv^{-1}$.  As $N=\lceil 2^{13}\kappa^3\rceil$, this has the claimed properties.
\end{proof}

\section{Reformulating the UCT II}\label{fd sec}

In this section (as throughout), if $B$ is a separable $C^*$-algebra, then $\LB$ and $\K_B$ denote respectively the adjointable and compact operators on the standard Hilbert $B$-module $\ell^2\otimes B$.  For each $n$, we consider $\LB$ as a subalgebra of $M_n(\LB)$ via the ``diagonal inclusion'' $\LB=1_{M_n}\otimes \LB\subseteq M_n\otimes \LB=M_n(\LB)$. 

Our goal in this section is to reformulate the vanishing results on the UCT of Section \ref{reform sec} in terms of the groups $KK^i_{\kappa,\epsilon}(X,B)$ of Section \ref{flex sec}.  We look at the even ($i=0$) and odd ($i=1$) cases separately.

\subsection{The even case}

\begin{lemma}\label{kappa to epsilon 0}
Let $\kappa\geq 1$ and $\epsilon > 0$.  Let $B$ be a separable $C^*$-algebra, and let $X$ be a self-adjoint subset of the unit ball of $\LB$.  Then there is a homomorphism $\psi_*:KK^*_{\kappa,\epsilon/4}(X,B)\to KK^*_{1,\epsilon}(X,B)$ such that the diagrams
\begin{equation}\label{kappa to epsilon 0 2}
\xymatrix{ KK^0_{1,\epsilon}(X,B) \ar[dr] &  \\ 
KK^0_{\kappa,\epsilon/4}(X,B) \ar[u]_-{\psi_*} \ar[r] & KK^0_{\kappa,\epsilon}(X,B)  }
\end{equation}
and 
\begin{equation}\label{kappa to epsilon 0 1}
\xymatrix{ KK^0_{1,\epsilon/4}(X,B) \ar[r] \ar[dr] &  KK^0_{1,\epsilon}(X,B) \\ 
&  KK^0_{\kappa,\epsilon/4}(X,B) \ar[u]^-{\psi_*}  }
\end{equation}
commute, where the unlabeled arrows are the forget control maps of Definition \ref{dir set 2}.
\end{lemma}

\begin{proof}
Let $(p,q)$ be an element of $\mathcal{P}_{n,\kappa,\epsilon/4}(X,B)$.  Let $r$ and $s$ be the projections associated to $p$ and $q$ respectively as in Definition \ref{idem to proj}.  Using Lemma \ref{idem to proj props} parts \eqref{ass pro} and \eqref{ip com 2} we may define a map
$$
\psi:\mathcal{P}_{n,\kappa,\epsilon/4}(X,B)\to \mathcal{P}_{n,1,\epsilon}(X,B),\quad (p,q)\mapsto (r,s).
$$
Allowing $n$ to vary, and noting that the process of taking associated projections takes homotopies to homotopies (by part \eqref{ass pro lip} of Lemma \ref{idem to proj props}) and block sums to block sums, we get a well-defined homomorphism
$$
\psi_*:KK^0_{\kappa,\epsilon/4}(X,B)\to KK^0_{1,\epsilon}(X,B).
$$ 

To check commutativity of the diagram in line \eqref{kappa to epsilon 0 2}, it suffices to show that if $(r,s)\in \mathcal{P}_{n,1,\epsilon}(X,B)$ is the pair of projections associated to $(p,q)\in \mathcal{P}_{n,\kappa,\epsilon/4}(X,B)$ as above, then $(r,s)$ and $(p,q)$ are in the same path component of $\mathcal{P}_{n,\kappa,\epsilon}(X,B)$.  This follows from parts \eqref{ip com 1} and \eqref{ip com 2} of Lemma \ref{idem to proj props}.  Commutativity of the diagram in line \eqref{kappa to epsilon 0 1} is immediate: if $(p,q)$ is in $\mathcal{P}_{n,1,\epsilon}(X,B)$ for some $n$, then $p$ and $q$ are themselves projections, so equal their associated projections.
\end{proof}

The following lemma records some results from \cite[Section A.3]{Willett:2020aa} that we will need.  For the statement, recall the notion of a unitally strongly absorbing representation from Definition \ref{usa def} above.

\begin{lemma}\label{a3 lem}
In the statement of this lemma, all unlabeled arrows are forget control maps as in Definitions \ref{dir set} and \ref{dir set 2}.  Let $A$ be a separable unital $C^*$-algebra, and let $B$ be a separable $C^*$-algebra.  Let $\pi:A\to \LB$ be a strongly unitally absorbing representation of $A$, which we use to identify $A$ with a $C^*$-subalgebra of $\LB$.  

Let $\epsilon>0$, and let $X$ be a finite subset of $A_1$.  Then there exist homomorphisms
$$
\alpha:KK^0_{1,\epsilon}(X,B)\to KK_{5\epsilon}(X,B)
$$
and 
$$
\beta: KK_{\epsilon}(X,B)\to KK^0_{1,\epsilon}(X,B)
$$
that are natural with respect to forget control maps: more precisely if $(X,\epsilon)\leq (Y,\delta)$ in $\mathcal{X}_A$ as in Definition \ref{dir set 0} then the diagrams
$$
\xymatrix{ KK^0_{1,\delta}(Y,B) \ar[d]^-\beta \ar[r] & KK^0_{1,\epsilon}(X,B)\ar[d]^-\beta  \\ 
KK^0_{\delta}(Y,B) \ar[r] & KK^0_{\epsilon}(X,B) } \quad \text{and}\quad \xymatrix{ KK^0_{1,5\delta}(Y,B) \ar[r] & KK^0_{1,5\epsilon}(X,B) \\ 
KK^0_{\delta}(Y,B) \ar[r] \ar[u]^-\alpha & KK^0_{\epsilon}(X,B) \ar[u]^-\alpha} 
$$
are defined and commute.

Moreover, the diagrams
$$
\xymatrix{ KK^0_{1,\epsilon}(X,B) \ar[d]^-\alpha \ar[r] & KK^0_{1,5\epsilon}(X,B) \\
KK_{5\epsilon}(X,B) \ar[ur]_-{\beta} & }
$$
and 
$$
\xymatrix{ KK_{\epsilon}(X,B) \ar[d]^-\beta \ar[r] & KK_{5\epsilon}(X,B) \\
KK_{1,\epsilon}^0(X,B) \ar[ur]_-{\alpha} & }
$$
commute.
\end{lemma}

\begin{proof}
Let $\pi:A\to M_2(\LB)$ be (the amplification of) our fixed representation.
In the language of \cite[Appendix A.2]{Willett:2020aa}, the groups $KK_{\epsilon}(X,B)$ are the same as the groups that are called there $KK_\epsilon^{\pi,p}(X,B)$, while in the language of \cite[Appendix A.3]{Willett:2020aa}, the groups $KK_{1,\epsilon}^0(X,B)$ would there be called $KK_\epsilon^{\pi_0,m}(X,B)$. The lemma thus follows from the arguments of \cite[Lemmas A.22, A.23, and A.24]{Willett:2020aa}
\end{proof}

We are now able to deduce a version of Corollary \ref{uct reform 2} for the groups of Definition \ref{alm com 0}.

\begin{corollary}\label{vanish cor 2}
Let $A$ be a separable, unital, nuclear $C^*$-algebra.  The following are equivalent:
\begin{enumerate}[(i)]
\item \label{a uct 3} $A$ satisfies the UCT.
\item \label{a vanish 3} Let $\kappa\geq 1$ and $\epsilon\in (0,1)$.  Let $B$ be a separable $C^*$-algebra with $K_*(B)=0$.  Let $\pi:A\to \LSB$ be a strongly unitally absorbing representation, which we use to identify $A$ with a $C^*$-subalgebra of $\LSB$.  Then for any finite subset $X$ of $A_1$, there is a finite subset $Z$ of $A_1$ such that $(X,\kappa,\epsilon)\leq (Z,\kappa,\epsilon/160)$ in the sense of Definition \ref{dir set 2}, and such that the forget control map
$$
KK^0_{\kappa,\epsilon/160}(Z,SB)\to KK^0_{\kappa,\epsilon}(X,SB)
$$
of Definition \ref{dir set 2} is zero.
\item \label{weak a vanish} There exist $\kappa\geq 1$ and $\nu\geq \kappa$ with the following property.  Let $\gamma>0$, let $B$ be a separable $C^*$-algebra with $K_*(B)=0$, and let $X$ be a finite subset of $A_1$.  Let $\pi:A\to \LSB$ be a strongly unitally absorbing representation, which we use to identify $A$ with a $C^*$-subalgebra of $\LSB$.  Then there is $\epsilon>0$ and a finite subset $Z$ of $A_1$ such that $(X,\nu,\gamma)\leq (Z,\kappa,\epsilon)$ in the sense of Definition \ref{dir set 2}, and such that the forget control map
$$
KK^0_{\kappa,\epsilon}(Z,SB)\to KK_{\nu,\gamma}^0(X,SB)
$$
of Definition \ref{dir set 2} is zero.
\end{enumerate}
\end{corollary} 

\begin{proof}
In the following proof, all unlabeled arrows are forget control maps as in Definition \ref{dir set}, or Definition \ref{dir set 2}.  Assume first that condition \eqref{a uct 3} from the statement holds, and let $\kappa\geq 1$ and $\epsilon>0$; we may assume moreover that $\epsilon<1$.  Let a finite subset $X$ be given as in condition \eqref{a vanish 3}.  Then by the equivalence from Corollary \ref{uct reform 2}, there is a finite subset $Z$ of $A_1$ such that the forget control map
$$
KK_{\epsilon/8}(Z,SB)\to KK_{\epsilon}(X,SB)
$$
is zero.  Replacing $Z$ by $Z\cup Z^*$ if necessary, we may assume that $Z$ is self-adjoint.   Lemma \ref{a3 lem} gives a commutative diagram
$$
\xymatrix{ KK_{\epsilon/8}(Z,SB)  \ar[r]^-0  & KK_{\epsilon}(X,SB)\ar[d]^-\beta  \\
KK^0_{1,\epsilon/40}(Z,SB) \ar[r] \ar[u]^-\alpha & KK^0_{1,\epsilon}(X,SB) },
$$
whence the bottom horizontal map is zero.  On the other hand, Lemma \ref{kappa to epsilon 0} (see in particular line \eqref{kappa to epsilon 0 2}) gives a map $\psi_*$ such that the bottom triangle in the diagram below 
$$
\xymatrix{ KK^0_{1,\epsilon/40}(Z,SB) \ar[r]^-0 \ar[dr] & KK^0_{1,\epsilon}(X,SB) \ar[d] \\
KK^0_{\kappa,\epsilon/160}(Z,SB) \ar[u]^-{\psi_*} \ar[r] & KK_{\kappa,\epsilon}(X,SB)} 
$$
commutes.  The top triangle also commutes as all the maps involved are forget control maps, whence the bottom horizontal map is zero.  This gives us condition \eqref{a vanish 3} from the statement.  

Condition \eqref{a vanish 3} clearly implies condition \eqref{weak a vanish}, so it remains to show that condition \eqref{weak a vanish} implies condition \eqref{a uct 3}.  For this, it suffices to establish condition \eqref{a vanish} from Theorem \ref{uct reform}, so let $\gamma>0$ and a finite subset $X$ of $A_1$ be given.  Then according to condition \eqref{weak a vanish} there are $\nu\geq \kappa\geq 1$, $\epsilon>0$ and a finite subset $Z$ of $A_1$ such that the forget control map 
$$
KK^0_{\kappa,\epsilon}(Z,SB)\to KK^0_{\nu,\gamma/20}(X,SB)
$$ 
is defined and zero. Replacing $Z$ with $Z\cup Z^*$ if necessary, we may assume $Z$ is self-adjoint.  Using Lemma \ref{kappa to epsilon 0} (see in particular line \eqref{kappa to epsilon 0 1}) there is a map $\psi_*$ such that the top right triangle in the diagram below comutes
$$
\xymatrix{ KK^0_{1,\epsilon}(Z,SB) \ar[r] \ar[d] & KK_{1,\gamma/20}(X,SB)  \ar[d] \ar[dr] \ar[r] & KK_{1,\gamma/5}(X,SB) \\
KK^0_{\kappa,\epsilon}(Z,SB) \ar[r]_-0 & KK_{\nu,\gamma/20}(X,SB) \ar[r] &  KK_{\nu,\gamma/20}^0(X,SB) \ar[u]_-{\psi_*} }.
$$
The rest of the diagram also commutes, as all the arrows are forget control maps, whence the composition 
$$
\xymatrix{ KK^0_{1,\epsilon}(Z,SB) \ar[r] & KK_{1,\gamma/20}(X,SB)  \ar[r] & KK_{1,\gamma/5}(X,SB)}
$$
of the two top horizontal maps is zero.  Using Lemma \ref{a3 lem}, there is a commutative diagram
$$
\xymatrix{ KK_{\epsilon}(Z,SB) \ar[r] \ar[d]_-\beta  & KK_{\gamma}(X,SB) \\
KK^0_{1,\epsilon}(Z,SB)  \ar[r]_-0 & KK_{1,\gamma/5}(X,SB) \ar[u]_-\alpha }.
$$
The top horizontal map is therefore zero; this is the conclusion we need for Theorem \ref{uct reform}, condition \eqref{a vanish} so we are done.
\end{proof}

\subsection{The odd case}

For the statement of the next lemma, consider the Hilbert module $\ell^2\otimes SB$ associated to the suspension $SB=C_0((0,1),B)$ of a separable $C^*$-algebra $B$.   Let $C_{sb}(X,M(C))$ denote the $C^*$-algebra of bounded and strictly continuous functions from a locally compact space 
$X$ to the multiplier algebra $M(C)$ of a $C^*$-algebra $C$.  For any $C^*$-algebra $C$ there are canonical identifications $\mathcal{L}_C=M(C\otimes \K)$ (see for example \cite[Theorem 2.4]{Lance:1995ys}) and $M(C_0(X,C))=C_{sb}(X,M(C))$ (see for example \cite[Corollary 3.4]{Akemann:1973aa}).  Hence there is a canonical identification 
\begin{equation}\label{lsb}
\mathcal{L}_{SB}=C_{sb}((0,1),\LB).
\end{equation}
We identify $\LB=\mathcal{L}(\ell^2\otimes B)$ with a $C^*$-subalgebra of $\LSB=\mathcal{L}(\ell^2\otimes B\otimes C_0(0,1))$ via the $*$-homomorphism $a\mapsto a\otimes 1_{C_0(0,1)}$.  We recall also that $\K_B^+$ denotes the unitization of $\K_B$.

\begin{lemma}\label{susp b lem}
Let $B$ be a separable $C^*$-algebra.  Let $\kappa\geq 1$, $\epsilon>0$, and let $X$ be a subset of the unit ball of $\LB$.  Then:
\begin{enumerate}[(i)]
\item \label{proj susp} Elements of $\mathcal{P}_{n,\kappa,\epsilon}(X,SB)$ (see Definition \ref{alm com 0} identify canonically with continuous paths $(p_t,q_t)_{t\in [0,1]}$ of idempotents in $M_n(\K_B^+)\oplus M_n(\K_B^+)$ satisfying the following conditions:
\begin{enumerate}[(a)]
\item for all $t\in [0,1]$, $\|p_t\|\leq \kappa$ and $\|q_t\|\leq \kappa$;
\item for all $t\in [0,1]$ and all $x\in X$, $\|[p_t,x]\|<\epsilon$ and $\|[q_t,x]\|<\epsilon$,
\item there are $p,q\in M_n(\C)$ such that $p_0=p_1=p$, $q_0=q_1=q$ and such that if $\sigma:M_n(\K_B^+)\to M_n(\C)$ is the canonical quotient map then $\sigma(p_t)=p$ and $\sigma(q_t)=q$ for all $t\in [0,1]$.
\end{enumerate}
Moreover, the element $(p,q)$ is in the subset $\mathcal{P}_{n,\kappa,\epsilon}^1(X,SB)$ of Definition \ref{nize proj} if and only if $p$ and $q$ are equal to $1_l$ for some $l\in \N$.
\item \label{uni susp} Elements of $\mathcal{U}_{n,\kappa,\epsilon}(X,SB)$ (see Definition \ref{alm com 1}) identify with continuous paths $(u_t)_{t\in [0,1]}$ of invertibles in $M_n(\K_B^+)$ satisfying the following conditions:
\begin{enumerate}[(a)]
\item for all $t\in [0,1]$, $\|u_t\|\leq \kappa$ and $\|u_t^{-1}\|\leq \kappa$;
\item for all $t\in [0,1]$ and all $x\in X$, $\|[u_t,x]\|<\epsilon$ and $\|[u_t^{-1},x]\|<\epsilon$;
\item there is $u\in M_n(\C)$ such that $u_0=u_1=u$ and such that if $\sigma:M_n(\K_B^+)\to M_n(\C)$ is the canonical quotient map then $\sigma(u_t)=u$ for all $t\in [0,1]$.
\end{enumerate}
Moreover, the element is in the subset $\mathcal{U}_{n,\kappa,\epsilon}^1(X,SB)$ of Definition \ref{nize uni} if and only if $u$ is the identity.
\end{enumerate}
\end{lemma}

\begin{proof}
We have a canonical identification 
$$
\K_{SB}^+=\{f\in C([0,1],\K_B^+)\mid \sigma(f(t))=f(0)=f(1) \text{ for all }t\in [0,1]\}.
$$
Part \eqref{proj susp} follows directly by comparing this with Definitions \ref{alm com 0} and \ref{nize proj}; similarly, part \eqref{uni susp} follows from comparing this with Definitions \ref{alm com 1} and \ref{nize uni}.  We leave the details to the reader.
\end{proof}

\begin{lemma}\label{uct reform odd case}
For any $\kappa\geq 1$ there exists a positive constant $M_1$ with the following property.  Let $\epsilon>0$, let $A$ be a separable, unital, nuclear $C^*$-algebra that satisfies the UCT, and let $B$ be a separable $C^*$-algebra with $K_*(B)=0$.  Let $\pi:A\to \LSB$ be a strongly unitally absorbing representation that factors through the subalgebra $\mathcal{B}(\ell^2)$ (such exists by Lemma \ref{scalar sa exist}), and use this to identify $A$ with a $C^*$-subalgebra of $\LSB$.

Then for any finite subset $X$ of $A_1$ there exists a finite subset $Z$ of $A_1$ such that the forget control map 
$$
KK_{\kappa,\epsilon}^1(Z,SB) \to KK_{M_1,M_1\epsilon}^1(X,SB)
$$
of Definition \ref{dir set 3} is defined and zero.  
\end{lemma}

\begin{proof}
We claim $M_1=2^{(200\kappa^8)^3}\cdot 320\kappa^7$ works.  Using Corollary \ref{vanish cor 2} there is a finite subset $Z$ of $A_1$ such that the forget control map
\begin{equation}\label{0 forget}
KK^0_{\kappa^8,2\kappa^6\epsilon}(Z,SB)\to KK^0_{\kappa^8,320\kappa^6\epsilon}(X,SB)
\end{equation}
of Definition \ref{dir set 2} is zero.  We claim this set $Z$ works.  

Let $u$ be an arbitrary element of $\mathcal{U}_{n,\kappa,\epsilon}(Z,B)$.  Using Proposition \ref{uni lem} part \eqref{same nu}, and with notation as in Definition \ref{nize uni}, there is an element $v$ of $\mathcal{U}^1_{n,\kappa^2,\kappa\epsilon}(Z,B)$ in the same path component of $\mathcal{U}_{n,\kappa^2,\kappa\epsilon}(Z,B)$ as $u$.  Define now a path $(v_t)_{t\in [0,1]}$ by 
\begin{equation}\label{v path}
v_t:=\begin{pmatrix} \cos(\pi t/2) & -\sin(\pi t/2) \\ \sin(\pi t/2) & \cos(\pi t/2) \end{pmatrix} \begin{pmatrix} 1 & 0 \\ 0 & v \end{pmatrix} \begin{pmatrix} \cos(\pi t/2) & \sin(\pi t/2) \\ -\sin(\pi t/2) & \cos(\pi t/2) \end{pmatrix}\begin{pmatrix} v^{-1} & 0 \\ 0 & 1 \end{pmatrix}.
\end{equation}
Note that each $v_t$ is an element of $\mathcal{U}^1_{2n,\kappa^4,2\kappa^3\epsilon}(Z,B)$.  Define  
$$
p_t:=v_t\begin{pmatrix} 1 & 0 \\ 0 & 0 \end{pmatrix} v_t^{-1}.
$$
Write $\underline{p}$ for the path $(p_t)$, and note that according to Lemma \ref{susp b lem} part \eqref{proj susp}, we may identify the pair $(\underline{p},1_n\oplus 0_n)$ with (using the notation of Definition \ref{nize proj}) an element of $\mathcal{P}_{2n,\kappa^8,2\kappa^7\epsilon}^1(Z,SB)$, and therefore also a class $[\underline{p},1_n\oplus 0_n]\in KK^0_{\kappa^8,2\kappa^7\epsilon}(Z,SB)$.  By assumption, the forget control map in line \eqref{0 forget} is zero, and therefore the image of $[\underline{p},1_n\oplus 0_n]$ in $KK^0_{\kappa^8,320\kappa^7\epsilon}(X,SB)$ is zero.  For notational simplicity, at this point let us define $\epsilon_1:=320\kappa^7\epsilon$.  

Now, Lemma \ref{kk0 zero} gives $m\in \N$ and $(s,s)\in \mathcal{P}_{2(n+m),2\kappa^8,\epsilon_1}(X,SB)$ such that $(\underline{p}\oplus 1_m\oplus 0_m,1_n\oplus 0_n\oplus 1_m\oplus 0_m)$ and $(s,s)$ are in the same path component of the set $\mathcal{P}_{2(n+m),2\kappa^8,\epsilon_1}(X,SB)$.   Let $x$ be a unitary matrix in $M_{2(n+m)}(\C)$ such that $x(1_n\oplus 0_n\oplus 1_m\oplus 0_m)x^*=1_{n+m}\oplus 0_{n+m}$.  As $x$ is connected to the identity through unitaries, the element $(x(\underline{p}\oplus 1_m\oplus 0_m)x^*,1_{n+m}\oplus 0_{n+m})$ is also homotopic to $(s,s)$ in $\mathcal{P}_{2(n+m),2\kappa^8,\epsilon_1}(X,SB)$; moreover (with notation as in Definition \ref{nize proj}), it is in $\mathcal{P}_{2(n+m),2\kappa^8,\epsilon_1}^1(X,SB)$.  We may now apply Proposition \ref{hom to sim} to see that if $M=2^{(200\kappa^8)^3}$ then there is $k\in \N$ and an element $\underline{w}$ of $\mathcal{U}^1_{2(n+m+k),M,M\epsilon_1}(X,SB)$ such that 
$$
\underline{w}(x(\underline{p}\oplus 1_m\oplus 0_m)x^*\oplus 1_k\oplus 0_k)\underline{w}^{-1}=1_{n+m}\oplus 0_{n+m}\oplus 1_k\oplus 0_k.
$$
Write $\underline{v}$ for the path defined in line \eqref{v path} above, which naturally defines an element of $\mathcal{L}_{SB}$ using the identification in line \eqref{lsb}.  Then if we define 
$$
\underline{y}:=\underline{w}(x\oplus 1_{2k})(\underline{v}\oplus 1_{2(m+k)})\in \mathcal{L}_{SB},
$$
we have 
$$
\underline{y}(1_n\oplus 0_n\oplus 1_m\oplus 0_m\oplus 1_k\oplus 0_k)\underline{y}^{-1}=1_n\oplus 0_n\oplus 1_m\oplus 0_m\oplus 1_k\oplus 0_k.
$$
In other words, the element $\underline{y}$ commutes with $1_n\oplus 0_n\oplus 1_m\oplus 0_m\oplus 1_k\oplus 0_k$.  Define 
$$
\underline{z}:=(1_n\oplus 0_n\oplus 1_m\oplus 0_m\oplus 1_k\oplus 0_k)\underline{y}(1_n\oplus 0_n\oplus 1_m\oplus 0_m\oplus 1_k\oplus 0_k).
$$
Using Lemma \ref{susp b lem} part \eqref{uni susp}, we may think of $\underline{z}$ as a path $(z_t)_{t\in [0,1]}$ in $\mathcal{U}_{n+m+k,M,M\epsilon_1}(X,B)$.  Now, write $\underline{w}$ as a path $(w_t)_{t\in [0,1]}$, and note that as $\underline{w}$ is in $\mathcal{U}^1_{2(n+m+k),M,M\epsilon_1}(X,SB)$, then by Lemma \ref{susp b lem} part \eqref{uni susp}, $w_0=w_1=1_{2(n+m)}$.  Moreover, $v_0=1_{2n}$ by definition.  Hence $z_0=x\oplus 1_k$.  On the other hand $v_1=u\oplus u^{-1}\oplus 1_{2(m+k)}$ and so $z_1=(x\oplus 1_k)(u\oplus 1_{m+k})$.   Hence $(x\oplus 1_k)^*\underline{z}$ defines a homotopy in $\mathcal{U}_{n+m+k,M,M\epsilon_1}(X,B)$ between $1_{n+m+k}$ and $u\oplus 1_{m+k}$.  This implies $[u]$ maps to zero in $KK^1_{M,M\epsilon_1}(X,SB)$, which completes the proof.
\end{proof}

\section{A Mayer-Vietoris boundary map}\label{bound sec}

In this section (as throughout), if $B$ is a separable $C^*$-algebra, then $\LB$ and $\K_B$ denote respectively the adjointable and compact operators on the standard Hilbert $B$-module $\ell^2\otimes B$.  For each $n$, we consider $\LB$ as a subalgebra of $M_n(\LB)$ via the ``diagonal inclusion'' $\LB=1_{M_n}\otimes \LB\subseteq M_n\otimes \LB=M_n(\LB)$. 

Our goal in this section is to construct and analyse a``Mayer-Vietoris boundary map'' in controlled $KK$-theory.  The main results of the section prove the existence of this boundary map (Proposition \ref{boundary map 2}) and show it has an exactness property (Proposition \ref{bound exact}).  These results are the technical heart of the paper.

\subsection{Existence}

Here is the construction of the boundary map.

\begin{proposition}\label{boundary map 2}
Define an increasing function $N_0:[1,\infty)\to[0,\infty)$ by the formula $N_0(\kappa)=2^{27}\kappa^{24}$.  This function has the following properties.

Let $\kappa\geq 1$, let $N_0=N_0(\kappa)$, let $\epsilon> 0$, let $B$ be a separable $C^*$-algebra, and let $X$ be a subset of the unit ball of $\LB$.    Let $h\in \LB$ be a positive contraction such that $\|[h,x]\|< \epsilon$ for all $x\in X$.  Then there is a homomorphism 
$$
\partial: KK_{\kappa,\epsilon}^1(h(1-h)X\cup\{h\},B)\to KK^0_{N_0,N_0\epsilon}(X\cup\{h\},B)
$$
defined by applying the following process to a class from $KK_{\kappa,\epsilon}^1(h(1-h)X\cup\{h\},B)$: 
\begin{enumerate}[(i)]
\item Choose a representative $w\in \mathcal{U}_{n,\kappa,\epsilon}(h(1-h)X\cup\{h\},B)$ for the class, and use Proposition \ref{uni lem} part \eqref{same nu} to find an element  $u\in \mathcal{U}^1_{n,\kappa^2,\kappa\epsilon}(h(1-h)X\cup\{h\},B)$ that is in the same path component as $w$ in $\mathcal{U}_{n,\kappa^2,\kappa\epsilon}(h(1-h)X\cup\{h\},B)$.
\item Define 
\begin{equation}\label{c and d}
c=c(u,h):=hu+(1-h),\quad d=d(u,h):=hu^{-1}+(1-h)
\end{equation}
in $M_n(\LB)$, and 
\begin{equation}\label{v def}
v=v(u,h):=\begin{pmatrix} 1 & c \\  0 & 1 \end{pmatrix}\begin{pmatrix} 1 & 0 \\  -d & 1 \end{pmatrix}\begin{pmatrix} 1 & c \\  0 & 1 \end{pmatrix}\begin{pmatrix} 0 & 1 \\  -1 & 0 \end{pmatrix}\in M_{2n}(\LB).
\end{equation}
\item Define  
$$
\partial[w]:= \Bigg[ v\begin{pmatrix} 1 & 0 \\ 0 & 0 \end{pmatrix} v^{-1}\,,\,\begin{pmatrix} 1 & 0 \\ 0 & 0 \end{pmatrix}\Bigg].
$$
\end{enumerate}
Moreover, the boundary map is ``natural with respect to forget control maps": precisely, if for some $\kappa\leq \lambda$ and $\epsilon\leq \delta$, the boundary maps 
$$
\partial: KK_{\kappa,\epsilon}^1(h(1-h)X\cup\{h\},B)\to KK^0_{N_0(\kappa),N_0(\kappa)\epsilon}(X\cup\{h\},B)
$$
and 
$$
\partial: KK_{\lambda,\delta}^1(h(1-h)X\cup\{h\},B)\to KK^0_{N_0(\lambda),N_0(\lambda)\delta}(X\cup\{h\},B)
$$
both exist, then the diagram 
$$
\xymatrix{ KK_{\kappa,\epsilon}^1(h(1-h)X\cup\{h\},B)\ar[r]^-\partial \ar[d] &  KK^0_{N_0(\kappa),N_0(\kappa)\epsilon}(X\cup\{h\},B) \ar[d] \\
KK_{\lambda,\delta}^1(h(1-h)X\cup\{h\},B) \ar[r]^-\partial &  KK^0_{N_0(\lambda),N_0(\lambda)\delta}(X\cup\{h\},B) }
$$
(with vertical maps the forget control maps of Definitions \ref{dir set 2} and \ref{dir set 3}) commutes.
\end{proposition}

In order to make the proof more palatable, we split off some computations as lemmas.  The proofs of these lemmas are elementary, but the second one is quite lengthy.  We record them for the sake of completeness, but recommend the reader skips the proofs.

\begin{lemma}\label{ab lem}
Let $B$ be a separable $C^*$-algebra.  Let $u\in M_n(\LB)$ be an invertible element such that $1-u\in M_n(\K_B)$, and let $h\in \LB$ be a positive contraction.   Then the elements $c=c(u,h)$ and $d=d(u,h)$ from line \eqref{c and d} above have the following properties.
\begin{enumerate}[(i)]
\item\label{in id} The elements $cd-1$ and $dc-1$ are in $M_n(\K_B)$.  
\item\label{close norm} If $\kappa\geq 1$ and $\epsilon> 0$  are such that $\|u\|\leq \kappa$, $\|u^{-1}\|\leq \kappa$, $\|[h,u]\|< \epsilon$, and $\|[h,u^{-1}]\|< \epsilon$, then $cd-1$ and $dc-1$ are both closer than $(\kappa+1)\epsilon$ to $h(1-h)(u+u^{-1}-2)$.  
\end{enumerate}
\end{lemma}

\begin{proof}
We just look at the case of $cd-1$ for both parts \eqref{in id} and \eqref{close norm}; the case of $dc-1$ is similar.  Note first that because $1-u$ is in $M_n(\K_B)$ and $M_n(\K_B)$ is an ideal in $M_n(\LB)$, we must have that $1-u^{-1}$ is in $M_n(\K_B)$ also.   We compute that
\begin{align}
cd-1 & =huhu^{-1}+(1-h)hu^{-1}+hu(1-h) -2h+h^2 \nonumber \\
& =h^2+hu[h,u^{-1}]+h(1-h)u^{-1} \nonumber \\ & \quad \quad \quad+h(1-h)u +[h,u](1-h)-2h+h^2 \label{cd-1}.
\end{align}
Using that $u$ and $u^{-1}$ equal $1$ modulo the ideal $M_n(\K_B)$, we compute that this equals $0$ modulo $M_n(\K_B)$.  Hence $cd-1$ is in $M_n(\K_B)$

Looking at part \eqref{close norm}, note that the terms $hu[h,u^{-1}]$ and $[h,u](1-h)$ in line \eqref{cd-1} above have norms at most $\kappa\epsilon$ and $\epsilon$ respectively.  Hence $cd-1$ is within $(\kappa+1)\epsilon$ of $h^2+h(1-h)u^{-1}+h(1-h)u-2h+h^2$, which equals $h(1-h)(u+u^{-1}-2)$.  
\end{proof}

\begin{lemma}\label{boundary map}
Let $B$ be a separable $C^*$-algebra.  Let $\kappa\geq 1$, $\epsilon> 0$, and let $X$ be a subset of the unit ball of $\LB$.  Let $h\in \LB$ be a positive contraction such that $\|[h,x]\|< \epsilon$ for all $x\in X$, and let $u$ be an element of the set $\mathcal{U}^1_{n,\kappa,\epsilon}(h(1-h)X\cup\{h\},B)$ from Definition \ref{nize uni}.  Let $c=c(u,h)$ and $d=d(u,h)$ be as in line \eqref{c and d} above, and let $v=v(u,h)$ be as in line \eqref{v def}.  

Then $\|v\|\leq (\kappa+2)^3$, $\|v^{-1}\|\leq (\kappa+2)^3$, and the pair 
$$
\Bigg(\,v\begin{pmatrix} 1 & 0 \\ 0 & 0 \end{pmatrix} v^{-1}\,,\,\begin{pmatrix} 1 & 0 \\ 0 & 0 \end{pmatrix}\,\Bigg)
$$
is an element of $\mathcal{P}^1_{2n,3^6\kappa^6,2^{16}\kappa^5\epsilon}(X\cup\{h\},B)$ from Definition \ref{nize proj}.  
\end{lemma}

\begin{proof}
From the definition of $v$ in line \eqref{v def} above,
\begin{equation}\label{v def 2}
v=\begin{pmatrix} c(dc-2) & 1-cd \\ dc-1 & -d \end{pmatrix}
\end{equation}
and 
$$
v^{-1}=\begin{pmatrix} 0 & -1 \\  1 & 0 \end{pmatrix}\begin{pmatrix} 1 & -c \\  0 & 1 \end{pmatrix}\begin{pmatrix} 1 & 0 \\  d & 1 \end{pmatrix}\begin{pmatrix} 1 & -c \\  0 & 1 \end{pmatrix}=\begin{pmatrix} -d & dc-1 \\ 1-cd & c(dc-2) \end{pmatrix}.
$$
Hence
$$
v\begin{pmatrix} 1 & 0 \\ 0 & 0 \end{pmatrix} v^{-1} = \begin{pmatrix} cd(2-cd) & c(dc-2)(dc-1) \\ (1-dc)d & (dc-1)^2 \end{pmatrix}
$$
and so 
\begin{equation}\label{bigg diff}
v\begin{pmatrix} 1 & 0 \\ 0 & 0 \end{pmatrix} v^{-1} - \begin{pmatrix} 1 & 0 \\ 0 & 0 \end{pmatrix} = \begin{pmatrix} -(cd-1)^2 & (cd-1)c(dc-2) \\ (1-dc)d & (dc-1)^2 \end{pmatrix}.
\end{equation}
This formula, part \eqref{in id} of Lemma \ref{ab lem}, and the fact that $M_n(\K_B)$ is an ideal in $M_n(\LB)$ imply that
$$
v\begin{pmatrix} 1 & 0 \\ 0 & 0 \end{pmatrix} v^{-1} - \begin{pmatrix} 1 & 0 \\ 0 & 0 \end{pmatrix}\in M_{2n}(\K_B),
$$ 
whence $v\begin{pmatrix} 1 & 0 \\ 0 & 0 \end{pmatrix} v^{-1}$ is in $M_{2n}(\K_B^+)$, and $v\begin{pmatrix} 1 & 0 \\ 0 & 0 \end{pmatrix} v^{-1}$ and $\begin{pmatrix} 1 & 0 \\ 0 & 0 \end{pmatrix}$ have the same image under the image of the canonical quotient map $\sigma: M_{2n}(\K_B^+)\to M_{2n}(\C)$.  Note moreover that $\|v\|\leq (\kappa+2)^3$ and $\|v^{-1}\|\leq (\kappa+2)^3$ from the formula for $v$ (whence also $v^{-1}$) as a product of four matrices in line \eqref{v def}.  As $\kappa\geq 1$, this implies that  
$$
\Bigg\|v\begin{pmatrix} 1 & 0 \\ 0 & 0 \end{pmatrix} v^{-1} \Bigg\|\leq (\kappa+2)^6\leq 3^6\kappa^6.
$$
To complete the proof that the pair 
$$
\Bigg(\,v\begin{pmatrix} 1 & 0 \\ 0 & 0 \end{pmatrix} v^{-1}\,,\,\begin{pmatrix} 1 & 0 \\ 0 & 0 \end{pmatrix}\,\Bigg)
$$
defines an element of $\mathcal{P}^1_{2n,3^6\kappa^6,2^{16}\kappa^5\epsilon}(X,B)$ it remains to check the relevant commutator estimates, i.e.\ condition \eqref{def small com} from Definition \ref{alm com 0} with $x$ in $X\cup \{h\}$ and $\epsilon$ replaced by $2^{16}\kappa^5\epsilon$.  As $\begin{pmatrix} 1 & 0 \\ 0 & 0 \end{pmatrix}$ (and indeed, any scalar matrix) commutes with elements of $X\cup \{h\}$ exactly, it suffices to show that 
\begin{equation}\label{v com x}
\Bigg\|\Bigg[x,v\begin{pmatrix} 1 & 0 \\ 0 & 0 \end{pmatrix} v^{-1} - \begin{pmatrix} 1 & 0 \\ 0 & 0 \end{pmatrix}\Bigg]\Bigg\| \leq 2^{16}\kappa^5\epsilon.
\end{equation}
for all $x\in X\cup \{h\}$.  We focus on the case when $x$ is in $X$: the case when $x=h$ follows from similar (and much simpler) estimates that we leave to the reader.

Working towards the estimate in line \eqref{v com x}, we compute that the element in line \eqref{bigg diff} equals 
\begin{equation}\label{bd prod}
\begin{pmatrix} cd-1 & 0 \\ 0 & dc-1\end{pmatrix} \begin{pmatrix} 1-cd & c(dc-2) \\ -d & dc-1 \end{pmatrix}.
\end{equation}
The second matrix above satisfies
\begin{align*}
\Bigg\| \begin{pmatrix} 1-cd & c(dc-2) \\ -d & dc-1 \end{pmatrix} \Bigg\| & \leq \|1-cd\|+\|c\|\|dc-2\|+\|d\|+\|dc-1\|  \\
& \leq ((\kappa+1)^2+1)+(\kappa+1)((\kappa+1)^2+2) \\ & \quad \quad \quad\quad\quad\quad\quad \quad +(\kappa+1)+((\kappa+1)^2+1) .
\end{align*}
As $\kappa+1\geq1$, we therefore see that 
\begin{equation}\label{mat norm}
\Bigg\| \begin{pmatrix} 1-cd & c(dc-2) \\ -d & dc-1 \end{pmatrix} \Bigg\|\leq 8(\kappa+1)^4.
\end{equation}
On the other hand, using part \eqref{close norm} of Lemma \ref{ab lem}, the first matrix in line \eqref{bd prod} above is closer than $\epsilon(\kappa+1)$ to $h(1-h)(u+u^{-1}-2)$ (we identify this as usual with the diagonal matrix with both entries equal to $h(1-h)(u+u^{-1}-2)$).  Hence the difference in line \eqref{bigg diff} is closer than $8(\kappa+1)^5\epsilon$ to 
$$
h(1-h)(u+u^{-1}-2)\begin{pmatrix} 1-cd & c(dc-2) \\ -d & dc-1 \end{pmatrix}.
$$
Hence for $x\in X$, 
\begin{align}\label{x com 1}
\Bigg\|\Bigg[x&,v\begin{pmatrix} 1 & 0 \\ 0 & 0 \end{pmatrix} v^{-1} - \begin{pmatrix} 1 & 0 \\ 0 & 0 \end{pmatrix}\Bigg]\Bigg\| \nonumber \\ & < 16(\kappa+1)^5\epsilon+\Bigg\|\Bigg[x,h(1-h)(u+u^{-1}-2)\begin{pmatrix} 1-cd & c(dc-2) \\ -d & dc-1 \end{pmatrix}\Bigg]\Bigg\|.
\end{align}
As $\|[x,h]\|< \epsilon$, we have $\|[x,h(1-h)]\|< 2\epsilon$; combining this with line \eqref{mat norm} gives
\begin{align*}
\Bigg\|\Bigg[x & ,h(1-h)(u+u^{-1}-2)\begin{pmatrix} 1-cd & c(dc-2) \\ -d & dc-1 \end{pmatrix}\Bigg]\Bigg\|   \\ & < 2\epsilon\cdot 8(\kappa+1)^5+\Bigg\|h(1-h)\Bigg[x,(u+u^{-1}-2)\begin{pmatrix} 1-cd & c(dc-2) \\ -d & dc-1 \end{pmatrix}\Bigg]\Bigg\|.
\end{align*}
Combining this with line \eqref{x com 1} gives 
\begin{align}\label{x com 2}
\Bigg\|\Bigg[x&,v\begin{pmatrix} 1 & 0 \\ 0 & 0 \end{pmatrix} v^{-1} - \begin{pmatrix} 1 & 0 \\ 0 & 0 \end{pmatrix}\Bigg]\Bigg\| \nonumber \\ & < 32(\kappa+1)^5\epsilon+\Bigg\|h(1-h)\Bigg[x,(u+u^{-1}-2)\begin{pmatrix} 1-cd & c(dc-2) \\ -d & dc-1 \end{pmatrix}\Bigg]\Bigg\|.
\end{align}
Every entry of the matrix $(u+u^{-1}-2)\begin{pmatrix} 1-cd & c(dc-2) \\ -d & dc-1 \end{pmatrix}$ can be written as a sum of at most $30$ terms, each of which is a product of at most $5$ elements from the set $\{u,u^{-1},h,1\}$, each of which has norm at most $\kappa$.  As $\|[h(1-h)x,y]\|< \epsilon$ for all $y\in \{u,u^{-1},h,1\}$, Lemma \ref{com prod} gives 
\begin{align}\label{x com 3}
\Bigg\|\Bigg[h(1-h)x,(u+u^{-1}-2)\begin{pmatrix} 1-cd & c(dc-2) \\ -d & dc-1 \end{pmatrix}\Bigg]\Bigg\|< 4\cdot 30\cdot 5\cdot \kappa^4\epsilon.
\end{align}
On the other hand, $\|[h(1-h),y]\|< 2\epsilon$ for all $y\in \{u,u^{-1},h,1\}$, whence 
\begin{align}\label{x com 4}
\Bigg\|\Bigg[h(1-h),(u+u^{-1}-2)\begin{pmatrix} 1-cd & c(dc-2) \\ -d & dc-1 \end{pmatrix}\Bigg]x\Bigg\|< 4\cdot 30\cdot 5\cdot \kappa^4\epsilon.
\end{align}
Finally, note that 
\begin{align*}
h(1-h) & \Bigg[x,(u+u^{-1}-2)\begin{pmatrix} 1-cd & c(dc-2) \\ -d & dc-1 \end{pmatrix}\Bigg] \\ & = \Bigg[h(1-h)x,(u+u^{-1}-2)\begin{pmatrix} 1-cd & c(dc-2) \\ -d & dc-1 \end{pmatrix}\Bigg] \\ & \quad \quad \quad +\Bigg[h(1-h),(u+u^{-1}-2)\begin{pmatrix} 1-cd & c(dc-2) \\ -d & dc-1 \end{pmatrix}\Bigg]x ,
\end{align*}
so combining lines \eqref{x com 2}, \eqref{x com 3}, and \eqref{x com 4} implies
\begin{align*}
\Bigg\|\Bigg[x&,v\begin{pmatrix} 1 & 0 \\ 0 & 0 \end{pmatrix} v^{-1} - \begin{pmatrix} 1 & 0 \\ 0 & 0 \end{pmatrix}\Bigg]\Bigg\| < 1232(\kappa+1)^5\epsilon.
\end{align*}
Recalling that $\kappa\geq 1$, this is enough for the estimate in line \eqref{v com x}.
\end{proof}

We are now ready for the proof of Proposition \ref{boundary map 2}.

\begin{proof}[Proof of Proposition \ref{boundary map 2}]
Assume that $w\in \mathcal{U}_{n,\kappa,\epsilon}(h(1-h)X\cup\{h\},B)$, and let $u\in \mathcal{U}^1_{n,\kappa^2,\kappa\epsilon}(h(1-h)X\cup\{h\},B)$ be in the same path component as $w$ in $\mathcal{U}_{n,\kappa^2,\kappa\epsilon}(h(1-h)X\cup\{h\},B)$; $u$ is guaranteed to exist by Proposition \ref{uni lem} part \eqref{same nu}.  Define $v:=v(u,h)$ as in line \eqref{v def}, so Lemma \ref{boundary map} gives an element
$$
\Bigg(\,v\begin{pmatrix} 1 & 0 \\ 0 & 0 \end{pmatrix} v^{-1}\,,\,\begin{pmatrix} 1 & 0 \\ 0 & 0 \end{pmatrix}\,\Bigg)\in \mathcal{P}_{2n,3^6\kappa^{12},2^{16}\kappa^{11}\epsilon}(X\cup\{h\},B).
$$  
Moreover, if $u_0:=u$, and $u_1$ is another choice of element in  $\mathcal{U}^1_{n,\kappa^2,\kappa\epsilon}(h(1-h)X\cup\{h\},B)$ that is connected to $w$ in $\mathcal{U}_{n,\kappa^2,\kappa\epsilon}(h(1-h)X\cup\{h\},B)$ then Proposition \ref{uni lem} part \eqref{hom nu} implies that there is a homotopy $(u_t)_{t\in [0,1]}$ that connects $u_0$ and $u_1$ through $\mathcal{U}^1_{n,\kappa^4,\kappa\epsilon}(h(1-h)X\cup\{h\},B)$.  Let $v_t:=v(u_t,h)$ be as in line \eqref{v def}.  Then Lemma \ref{boundary map} implies that the path 
$$
t\mapsto \Bigg(\,v_t\begin{pmatrix} 1 & 0 \\ 0 & 0 \end{pmatrix} v_t^{-1}\,,\,\begin{pmatrix} 1 & 0 \\ 0 & 0 \end{pmatrix}\,\Bigg),\quad t\in [0,1]
$$
has image in $\mathcal{P}^1_{2n,3^6\kappa^{24},2^{16}\kappa^{21}\epsilon}(X\cup\{h\},B)$.  In particular, the class $\partial[w]\in KK^0_{3^6\kappa^{24},2^{16}\kappa^{21}\epsilon}(X\cup\{h\},B)$ does not depend on the choice of $u$, so at this point we have a well-defined set map 
$$
\mathcal{U}_{n,\kappa,\epsilon}(h(1-h)X\cup\{h\},B)\to KK^0_{3^6\kappa^{24},2^{16}\kappa^{21}\epsilon}(X\cup\{h\},B).
$$ 
We next claim that this map sends block sums on the left to sums on the right.

For this, assume that $w_1$ and $w_2$ are elements of $\mathcal{U}_{n,\kappa,\epsilon}(h(1-h)X\cup\{h\},B)$.  Let $u_1$ and $u_2$ be elements of $\mathcal{U}^1_{n,\kappa^2,\kappa\epsilon}(h(1-h)X\cup\{h\},B)$ that are connected to $w_1$ and $w_2$ respectively in $\mathcal{U}^1_{n,\kappa^2,\kappa\epsilon}(h(1-h)X\cup\{h\},B)$.  For $i\in \{1,2\}$ let $v_i=v(u_i,h)$ be as in line \eqref{v def}, and let $v:=v(u_1\oplus u_2,h)\in M_{4n}(\LB)$.  Then the pairs 
$$
\Bigg( v_1\begin{pmatrix} 1_n & 0 \\ 0 & 0 \end{pmatrix} v_1^{-1} \oplus v_2\begin{pmatrix} 1_n & 0 \\ 0 & 0 \end{pmatrix} v_2^{-1} ~,~\begin{pmatrix} 1_n & 0 \\ 0 & 0 \end{pmatrix} \oplus \begin{pmatrix} 1_n & 0 \\ 0 & 0 \end{pmatrix}  \Bigg)
$$
and 
$$
\Bigg( v\begin{pmatrix} 1_{2n} & 0 \\ 0 & 0 \end{pmatrix} v^{-1}~,~\begin{pmatrix} 1_{2n} & 0 \\ 0 & 0 \end{pmatrix}\Bigg)
$$
in $M_{4n}(\K_B^+)\oplus M_{4n}(\K_B^+)$ differ by conjugation by the same (scalar) permutation matrix in each component, and so define the same class in $KK^0_{3^6\kappa^{24},2^{16}\kappa^{21}\epsilon}(X\cup\{h\},B)$.

At this point, we have a semigroup homomorphism 
$$
\mathcal{U}_{n,\kappa,\epsilon}(h(1-h)X\cup\{h\},B)\to KK^0_{3^6\kappa^{24},2^{16}\kappa^{21}\epsilon}(X\cup\{h\},B).
$$ 
We claim that it respects the equivalence relation defining $KK_{\kappa,\epsilon}^1(h(1-h)X\cup\{h\},B)$.  First, we check that $w\oplus 1_k$ goes to the same class as $w$.  As we already know we have a semigroup homomorphism, it suffices to show that $1_k$ goes to zero in $KK^0_{3^6\kappa^{24},2^{16}\kappa^{20}\epsilon}(X\cup\{h\},B)$.  For this, note that if $v:=v(1_k,h)$ is as in line \eqref{v def}, then $v=1_{2k}$, whence the image of $1_k$ in $KK^0_{3^6\kappa^{24},2^{16}\kappa^{21}\epsilon}(X\cup\{h\},B)$ is the class $[1_k\oplus 0_k,1_k\oplus 0_k]$, which is zero by definition.  

Let us now show that elements of $\mathcal{U}_{n,\kappa,\epsilon}(h(1-h)X\cup\{h\},B)$ that are homotopic through $\mathcal{U}_{n,2\kappa,\epsilon}(h(1-h)X\cup\{h\},B)$ go to the same class.  For this, say that $w_0$ and $w_1$ are homotopic through $\mathcal{U}_{n,2\kappa,\epsilon}(h(1-h)X\cup\{h\},B)$.  Choose $u_0$ and $u_1$ in $\mathcal{U}^1_{n,\kappa^2,\kappa\epsilon}(h(1-h)X\cup\{h\},B)$ that are connected to $w_0$ and $w_1$ respectively in $\mathcal{U}_{n,\kappa^2,\kappa\epsilon}(h(1-h)X\cup\{h\},B)$ as in Proposition \ref{uni lem} part \eqref{same nu}.  Using Proposition \ref{uni lem} part \eqref{hom nu}, $u_0$ and $u_1$ are connected by a homotopy $(u_t)_{t\in [0,1]}$ in $\mathcal{U}^1_{n,4\kappa^4,2\kappa\epsilon}(h(1-h)X\cup\{h\},B)$.   Let $v_t:=v(u_t,h)$ be as in line \eqref{v def}.  Then Lemma \ref{boundary map} implies that the path 
$$
\Bigg(\,v_t\begin{pmatrix} 1 & 0 \\ 0 & 0 \end{pmatrix} v_t^{-1}\,,\,\begin{pmatrix} 1 & 0 \\ 0 & 0 \end{pmatrix}\,\Bigg)
$$
defines a homotopy between the images of $w_0$ and $w_1$ in $\mathcal{P}^1_{2n,3^{14}\kappa^{24},2^{27}\kappa^{21}\epsilon}(X\cup\{h\},B)$.  We thus see that $N_0(\kappa):=2^{27}\kappa^{24}$ has the desired property, and we are done with the existence of $\partial$.  

As the formulas for the boundary map $\partial$ do not depend on the constants $\kappa$ and $\epsilon$ the naturality statement is clear.
\end{proof}

\subsection{Exactness}

We now turn to the exactness property of the boundary map.  In order to state this, we need two lemmas.

\begin{lemma}\label{cut proj lem}
Let $B$ be a separable $C^*$-algebra.  Let $X$ and $Y$ be subsets of the unit ball of $\LB$, $\epsilon> 0$ and $\kappa\geq 1$.  Let $h\in \LB$ be a positive contraction such that $\|[h,x]\|< \epsilon$ for all $x\in X$.  With notation as in Definition \ref{alm com 0}, let $(p,q)\in \mathcal{P}_{n,\kappa,\epsilon}(X\cup Y\cup\{h\},B)$ (respectively, with notation as in Definition \ref{nize proj}, let $(p,q)\in \mathcal{P}^{(1)}_{n,\kappa,\epsilon}(X\cup Y\cup\{h\},B)$).  Then 
$$
(p,q)\in \mathcal{P}_{n,\kappa,2\epsilon}(hX\cup Y\cup\{h\},B)
$$ 
(respectively, 
$$
(p,q)\in \mathcal{P}^1_{n,\kappa,2\epsilon}(hX\cup Y\cup\{h\},B) ) ~\text{)}.
$$
In particular, there are homomorphisms 
$$
\eta_h:KK^0_{\kappa,\epsilon}(X\cup Y\cup\{h\},B)\to KK^0_{\kappa,2\epsilon}(hX\cup Y\cup \{h\},B)
$$
and 
$$
\eta_{1-h}:KK^0_{\kappa,\epsilon}(X\cup Y\cup\{h\},B)\to KK^0_{\kappa,2\epsilon}((1-h)X\cup Y\cup \{h\},B)
$$
induced by the identity map on cycles $(p,q)$.
\end{lemma}

\begin{proof}
We compute that for $x\in X$,
$$
\|[p,hx]\|\leq \|h\|\|[p,x]\|+\|[p,h]\|\|x\|< \epsilon+\epsilon
$$
These estimates hold similarly for $q$ so $(p,q)\in \mathcal{P}^1_{n,\kappa,2\epsilon}(hX\cup Y\cup\{h\},B)$.  As the identity map on cycles takes homotopies to homotopies, and block sums to block sums, existence of the homomorphism $\eta_h$ is clear.  Existence of $\eta_{1-h}$ follows on noting that the assumptions on $h$ also holds for $1-h$.  
\end{proof}

We leave the direct checks needed for the proof of the next lemma for the reader.

\begin{lemma}\label{forget close}
Let $B$ be a separable $C^*$-algebra.  Let $X$ and $Y$ be subsets of the unit ball of $\LB$, $\epsilon> 0$ and $\kappa\geq 1$.  Assume moreover that there is $\delta>0$ such that for all $y\in Y$, $x\in_\delta X$.  Then for any $\gamma\geq \kappa\delta+\epsilon$ and $\lambda\geq \kappa$, the forget control map of Definition \ref{dir set 2}
$$
KK^0_{\kappa,\epsilon}(X,B)\to KK_{\lambda,\gamma}(Y,B)
$$
is well-defined. \qed
\end{lemma}

The next proposition is the exactness property of the Mayer-Vietoris boundary map that we are aiming for.   We refer the reader to Subsection \ref{strat subsec} for motivation behind the statement.  For the statement, recall that for an element $x$ and subset $Y$ of a metric space, and for $\epsilon>0$, we write ``$x\in_\epsilon S$'' to mean that there is $y\in Y$ with $d(x,y)<\epsilon$.  Moreover, in the statement below, all unlabeled arrows between controlled $KK$-groups are the forget control maps of Definition \ref{dir set 2} or Definition \ref{dir set 3}.

\begin{proposition}\label{bound exact}
The increasing functions $N_1,N_2:[1,\infty)\to[1,\infty)$ defined by 
$$
N_1(\lambda)=2^{9000000\lambda^3} \quad \text{and}\quad N_2(\mu)=2^37\mu^{25}.
$$
satisfy the following properties.

Let $\kappa\geq 1$, and let $\epsilon>0$.  Let $\lambda\geq \kappa$, and let $\delta\geq 3\kappa\epsilon$.  Let $N_1:=N_1(\lambda)$, and let $\mu\geq N_1$ and $\gamma\geq N_1\delta$.  With notation as in Proposition \ref{boundary map 2}, define $N_0:=N_0(\mu)$, and let $N_2:=N_2(\mu)$.  

Let $B$ be a separable $C^*$-algebra, and let $X$ be a self-adjoint subset of the unit ball of $\LB$.  Let $h\in \LB$ be a positive contraction such that $\|[h,x]\|< \epsilon$ for all $x\in X$.  Let $Y_h$, $Y_{1-h}$, and $Y$ be self-adjoint subsets of the unit ball of $\LB$ such that $y\in_\epsilon Y_h$ and $y\in_\epsilon Y_{1-h}$ for all $y\in Y$.  With notation as in Definition \ref{nize proj}, let $(p,q)$ be an element of $\mathcal{P}^1_{n,\kappa,\epsilon}(X\cup Y_h\cup Y_{1-h}\cup\{h\},B)$.  With $\eta_h$ and $\eta_{1-h}$ as in Lemma \ref{cut proj lem}, and suing Lemma \ref{forget close} to define the right hand vertical maps in each case, assume that the images of $[p,q]$ under the maps
\begin{equation}\label{eta h}
\xymatrix{ KK^0_{\kappa,\epsilon}(X\cup Y_h\cup Y_{1-h} \cup\{h\},B) \ar[d] & \\ KK^0_{\kappa,\epsilon}(X\cup Y_h\cup\{h\},B) \ar[r]^-{\eta_h} & KK^0_{\kappa,2\epsilon}(hX\cup Y_h\cup\{h\},B) \ar[d] \\
& KK^0_{\lambda,\delta}(hX\cup Y\cup\{h\},B)}
\end{equation}
and 
\begin{equation}\label{eta 1-h}
\xymatrix{ KK^0_{\kappa,\epsilon}(X\cup Y_h\cup Y_{1-h} \cup\{h\},B) \ar[d] &\\ KK^0_{\kappa,\epsilon}(X\cup Y_{1-h}\cup\{h\},B) \ar[r]^-{\eta_{1-h}} & KK^0_{\kappa,2\epsilon}((1-h)X\cup Y_{1-h}\cup\{h\},B) \ar[d] \\
& KK^0_{\lambda,\delta}(hX\cup Y\cup\{h\},B)}
\end{equation}
are zero.

Then with notation as in Definition \ref{nize uni}, there exists an element 
$$
u\in \mathcal{U}^1_{\infty,N_1,N_1\delta}(h(1-h)X\cup\{h\}\cup Y,B)
$$ 
such that in the diagram below
$$
\xymatrix{ KK^1_{N_1,N_1\delta}(h(1-h)X\cup\{h\}\cup Y,B) \ar[d] & KK^0_{\kappa,\epsilon}(X\cup Y_h\cup Y_{1-h}\cup\{h\},B)\ar[d] \\
 KK^1_{\mu,\gamma}(h(1-h)X\cup \{h\},B) \ar[r]^-\partial & KK^0_{N_0,N_0\gamma}(X\cup \{h\},B) \ar[d] \\
 & KK^0_{N_2,N_2\gamma}(X\cup \{h\},B)}
$$
the images of the classes $[u]\in KK^1_{N_1,N_1\delta}(h(1-h)X\cup\{h\}\cup Y)$ and $[p,q]\in KK^0_{\kappa,\epsilon}(X\cup Y_h\cup Y_{1-h}\cup\{h\},B)$ in the bottom right group $KK^0_{N_2,N_2\gamma}(X\cup \{h\},B)$ are the same.
\end{proposition}

Just as for Proposition \ref{boundary map 2}, to make the argument more palatable, we split off some computations as two technical lemmas.  As in that earlier case, the arguments we give for these lemmas are elementary, but quite lengthy (in fact, much longer than the earlier ones).  We record them for the sake of completeness, but again recommend that the reader skips the proofs.

\begin{lemma}\label{u exist}
Let $B$ be a separable $C^*$-algebra.  Let $\nu\geq 1$ and let $\gamma> 0$.  Let $X$ and $Y$ be self-adjoint subsets of the unit ball of $\LB$.  Let $h\in \LB$ be a positive contraction such that $\|[h,x]\|< \gamma$ for all $x\in X$.  Let $(p,q)\in \mathcal{P}^1_{n,\nu,\gamma}(X\cup Y\cup\{h\},B)$ (see Definition \ref{nize proj} for notation), and let $u_h\in \mathcal{U}^1_{n,\nu,\gamma}(hX\cup\{h\}\cup Y,B)$ and $u_{1-h}\in \mathcal{U}^1_{n,\nu,\gamma}((1-h)X\cup \{h\}\cup Y,B)$ (see Definition \ref{nize uni} for notation).

Then the element 
\begin{equation}\label{uhp def}
u:=u_{1-h}(1-p)+u_hp
\end{equation}
is in $\mathcal{U}^1_{n,2\nu^2,10\nu \gamma}(h(1-h)X\cup\{h\}\cup Y,B)$.
\end{lemma}

\begin{proof}
We split the computations into the points labeled \eqref{u where 1}, \eqref{u where 2}, \eqref{u where 3}, \eqref{u where 4}, and \eqref{u where 5} below.
\begin{enumerate}[(i)]
\item \label{u where 1}  As $u_h-1\in M_{n}(\K_B)$ and $u_{1-h}-1\in M_{n}(\K_B)$, we compute from line \eqref{uhp def} that $u-1\in M_{n}(\K_B)$.
\item \label{u where 2} Note that 
\begin{equation}\label{1-p norm}
\|1-p\|\leq \nu
\end{equation}
by Corollary \ref{id norm cor}.  Hence $\max\{\|u_h\|,\|u_{1-h}\|,\|p\|,\|1-p\|\}\leq \nu$, and so by line \eqref{uhp def}, $\|u\|\leq 2\nu^2$. 
\item \label{u where 3} Let $y\in Y$.  Then by definition, $\|[a,y]\|<\gamma$ for all $a\in \{u_h,u_{1-h},p,1-p\}$.  Hence the definition of $u$ from line \eqref{uhp def} implies that $\|[y,u]\|$ is bounded above by
\begin{align*}
& \|[y,u_{1-h}]\|\|1-p\|+\|u_{1-h}\|\|[y,1-p]\| + \|[y,u_{h}]\|\|p\|+\|u_{h}\|\|[y,p]\| \\
& < 4\nu\gamma. \\
\end{align*}
\item \label{u where 4} Using the definition of $u$ from line \eqref{uhp def} and the assumptions on $u_h$, $u_{1-h}$ and $p$ directly together with line \eqref{1-p norm} implies that 
\begin{align*}
\|[u,h]\| & \leq \|[h,u_{1-h}]\|\|1-p\|+\|u_{1-h}\|\|[h,1-p]\| + \|[h,u_{h}]\|\|p\|+\|u_{h}\|\|[h,p]\| \\
& <4\nu\gamma
\end{align*}
\item \label{u where 5} Let $x\in X$ and note that
$$
[h(1-h)x,u_h]=(1-h)[hx,u_h]+[h,u_h](1-h)x.
$$
As $\|[hx,u_h]\|< \gamma$, as $\|[h,u_h]\|< \gamma$, as $h$ is a positive contraction, and as $x$ is a contraction, we get 
\begin{equation}\label{h 1-h x 1}
\|[h(1-h)x,u_h]\|\leq \|[hx,u_h]\|\|1-h\|+\|hx\||\|[1-h,u_h]\|< 2\gamma.
\end{equation}
Completely analogously, we see that 
\begin{equation}\label{h 1-h x 2}
\|[h(1-h)x,u_{1-h}]\|< 2\gamma.
\end{equation}
We see also that 
\begin{align*}
\|[h(1-h)x,p]\| & \leq \|[x,p]\|\|h(1-h)\|+\|[1-h,p]\|\|hx\|+\|[h,p]\|\|(1-h)x\| \\ 
& <3\gamma.
\end{align*}
Combining this with lines \eqref{1-p norm}, \eqref{h 1-h x 1}, \eqref{h 1-h x 2}, we get
\begin{align*}
\|[h(1-h)x,u]\| & \leq \|[h(1-h)x,u_{1-h}]\|\|1-p\|+\|u_{1-h}\|\|[h(1-h)x,1-p\| \\ & \quad  \quad \quad  \quad \quad +\|[h(1-h)x,u_{h}]\|\|p\|+\|u_{h}\|\|[h(1-h)x,p\| \\
& < 2 \nu\gamma+3\nu\gamma +2 \nu\gamma+3\nu\gamma \\
& =10\nu\gamma.
\end{align*}
\end{enumerate}
Putting the points \eqref{u where 1}, \eqref{u where 2}, \eqref{u where 3}, \eqref{u where 4}, and \eqref{u where 5} above together (and using that $\nu\geq 1$) we conclude that, $u$ is an element of $\mathcal{U}^1_{n,2\nu^2,10\nu\gamma}(h(1-h)X\cup\{h\}\cup Y,B)$ as claimed. 
\end{proof}

\begin{lemma}\label{vw lem}
With assumptions as in Lemma \ref{u exist}, let    
$$
u:=u_{1-h}(1-p)+u_hp\in \mathcal{U}^1_{n,2\nu^2,10\nu \gamma}(h(1-h)X\cup\{h\}\cup Y,B)
$$
be the element considered there.  Let $v:=v(u,h)$ be as in line \eqref{v def} above, and define 
$$
w:=\begin{pmatrix} u_{1-h}(1-p) & -q \\ p & (1-p)u_{1-h}^{-1}\end{pmatrix}\in M_{2n}(\LB).
$$
Then $w$ is invertible, and $vw^{-1}$ is in $\mathcal{U}_{2n,(2\nu)^8,2^{37}\nu^{25}\gamma}(X\cup\{h\},B)$.
\end{lemma}

\begin{proof}
Using the assumptions on $\|p\|$, $\|u_{1-h}\|$, $\|u_{1-h}^{-1}\|$ and line \eqref{1-p norm} to estimate $\|1-p\|$, we have 
$$
\|w\|\leq \|u_{1-h}(1-p)\|+\|q\|+\|p\|+\|(1-p)u_{1-h}^{-1}\|\leq 4\nu^2.
$$ 
A direct computation shows that $w$ is invertible with inverse 
\begin{equation}\label{w inv def}
w^{-1}=\begin{pmatrix} (1-p)u_{1-h}^{-1} & p \\ -q & u_{1-h}(1-p)\end{pmatrix}.
\end{equation}
This satisfies the same norm estimate as $w$, and so we get the norm estimates 
\begin{equation}\label{w norm}
\|w\|\leq (2\nu)^2 \quad\text{and}\quad\|w^{-1}\|\leq (2\nu)^2.
\end{equation}
Lemma \ref{boundary map} and the fact that $\|u\|\leq 2\nu^2$ implies that $\|v\|\leq (2\nu^2+2)^3$ and $\|v^{-1}\|\leq (2\nu^2+2)^3$.   As $\nu\geq 1$, we thus see that 
\begin{equation}\label{v norm} 
\|v\|\leq (2\nu)^6 \quad\text{and}\quad  \|v^{-1}\|\leq (2\nu)^6.
\end{equation}
Lines \eqref{w norm} and \eqref{v norm} then imply 
\begin{equation}\label{vw norm} 
\|vw^{-1}\|\leq (2\nu)^8 \quad\text{and}\quad  \|wv^{-1}\|\leq (2\nu)^8.
\end{equation}

To complete the proof, we need to show that for all $x\in X\cup \{h\}$, we have $\|[vw^{-1},x]\|< 2^{37}\nu^{25}\gamma$ and $\|[wv^{-1},x]\|< 2^{37}\nu^{25}\gamma$.  We focus first on the case of $vw^{-1}$, and look first at $[h,vw^{-1}]$.  

Let $c:=hu+(1-h)$ and $d:=hu^{-1}+(1-h)$ be as in line \eqref{c and d}.  It will be technically convenient to define 
\begin{equation}\label{S def}
S:=\{h,1-h,p,q,1-p,1-q,u_h,u_{h}^{-1},u_{1-h},u_{1-h}^{-1},u,u^{-1},c,d\},
\end{equation}
and to define $S^n$ to be the set of all products of at most $n$ elements from $S$.  Note that for every $s\in S$ we have $\|s\|\leq (2\nu)^2$, and $\|[s,h]\|< 10\nu\gamma$.  Hence by Lemma \ref{com prod}, for all $n\in \N$ we have 
\begin{equation}\label{S est}
s\in S^n \quad \Rightarrow \quad \|[h,s]\|< n (2\nu)^{2(n-1)} 10\nu\gamma.
\end{equation}
Using the formula in line \eqref{v def 2} above, 
$$
[h,v]=\begin{pmatrix} [cdc,h]-2[c,h] & [cd,h] \\ [h,dc] & [d,h] \end{pmatrix}.
$$
and so 
$$
\|[h,v]\|\leq \|[cdc,h]\|+2\|[c,h]\|+\|[cd,h]\|+\|[h,dc]\|+\|[d,h]\|.
$$
Each summand on the right hand side above is of the form $\|[h,s]\|$ where $s\in S^3$ for $S$ as in line \eqref{S def}.  Hence line \eqref{S est} implies that 
\begin{equation}\label{[h,v]}
\|[h,v]\|< 6\cdot 3\cdot (2\nu)^{4} \cdot 10\nu\gamma\leq 2^{11}\nu^5 \gamma
\end{equation}
We also compute that 
$$
[h,w^{-1}]=\begin{pmatrix} [h,(1-p)u_{1-h}^{-1}] & [h,p] \\ [q,h] & [h,u_{1-h}(1-p)]\end{pmatrix},
$$
whence 
\begin{align*}
\|[h,w^{-1}]\|\leq   \|[h,(1-p)u_{1-h}^{-1}]\| +\|[h,p]\| +\|[q,h]\| +\|[h,u_{1-h}(1-p)]\|
\end{align*}
Each commutator appearing above is of the form $[h,s]$ for some $s\in S^2$ as in line \eqref{S def}, whence line \eqref{S est} gives
\begin{equation}\label{[h,w]}
\|[h,w^{-1}]\|< 4\cdot (2\nu)^2\cdot 10\nu\gamma\leq 2^7\nu^3\gamma.
\end{equation}
On the other hand,  
$$
\|[h,vw^{-1}]\|\leq \|[h,v]\|\|w^{-1}\|+\|v\|\|[h,w^{-1}]\|.
$$  
Combining this with lines \eqref{w norm}, \eqref{v norm}, \eqref{[h,v]}, and \eqref{[h,w]}, as well as that $\nu \geq 1$, we see that  
\begin{equation}\label{[h,vw]}
\|[h,vw^{-1}]\|< 2^{11}\nu^5\gamma\cdot (2\nu)^2+(2\nu)^6\cdot 2^7\nu^3\gamma\leq 2^{14}\nu^{9}\gamma.
\end{equation}

Now let us look at $[x,vw^{-1}]$ for $x\in X$.  The definition of $v$ from line \eqref{v def} gives
\begin{align*}
vw^{-1} & =\begin{pmatrix} c(dc-1) & 1-cd \\  dc-1 & 0 \end{pmatrix}w^{-1} -\begin{pmatrix} c & 0 \\ 0 & d \end{pmatrix}w^{-1} \\
& = \begin{pmatrix} cd-1 & 0 \\  0 & dc-1 \end{pmatrix}\begin{pmatrix} c & -1 \\  1 & 0 \end{pmatrix}w^{-1}-\begin{pmatrix} c & 0 \\ 0 & d \end{pmatrix}w^{-1}.
\end{align*}
Hence the formula for $w^{-1}$ from line \eqref{w inv def} gives
\begin{align}\label{y1 y2 y3}
vw^{-1} = &\underbrace{\begin{pmatrix} cd-1 & 0 \\  0 & dc-1 \end{pmatrix}\begin{pmatrix} c(1-p)u_{1-h}^{-1} & cp-u_{1-h}(1-p) \\  (1-p)u_{1-h}^{-1} & p \end{pmatrix}}_{y_1} \nonumber \\ &  \quad \quad-\underbrace{h\begin{pmatrix} 1-q & u_hp \\ -u_h^{-1}q & 1-p \end{pmatrix}}_{y_2} \nonumber \\ & \quad \quad \quad\quad- \underbrace{(1-h)\begin{pmatrix} (1-p)u_{1-h}^{-1} & p \\ -q & u_{1-h}(1-p)\end{pmatrix}}_{y_3}.
\end{align}
We now estimate $\|[vw^{-1},x]\|$ for each $x\in X$ by looking at each of the terms $y_1$, $y_2$, and $y_3$ separately.
\begin{enumerate}[(i)]
\item\label{y_1} First, we look at $y_1$ from line \eqref{y1 y2 y3}.  Let $x\in X$.  Lemma \ref{ab lem} implies that 
\begin{equation}\label{close nice}
\Bigg\|\begin{pmatrix} cd-1 & 0 \\  0 & dc-1 \end{pmatrix}-h(1-h)(u+u^{-1}-2)\Bigg\|< (\nu+1)\gamma
\end{equation}
(where, as usual, we identify $h(1-h)(u+u^{-1}-2)$ with the corresponding diagonal matrix).  Let 
\begin{equation}\label{z1 def}
z_1:=\begin{pmatrix} c(1-p)u_{1-h}^{-1} & cp-u_{1-h}(1-p) \\  (1-p)u_{1-h}^{-1} & p \end{pmatrix}.
\end{equation}
As in line \eqref{1-p norm}, $\|1-p\|\leq \nu$, whence using that $\nu\geq 1$,
\begin{align}\label{z1 norm}
\|z_1\| & \leq \|c\|\|1-p\|\|u_{1-h}^{-1}\|+\|c\|\|p\|+\|u_{1-h}\|\|1-p\|+\|1-p\|\|u_{1-h}^{-1}\|+\|p\| \nonumber \\
& \leq (2\nu^2+1)\nu^2+(2\nu^2+1)\nu+\nu^2+\nu^2+\nu \nonumber \\
& \leq 9\nu^4.
\end{align}
Combining this with line \eqref{close nice}, we see that 
\begin{align*}
\|y_1-& h(1-h)(u+u^{-1}-2)z_1\|  \\ & \leq \Bigg\|\begin{pmatrix} cd-1 & 0 \\  0 & dc-1 \end{pmatrix}-h(1-h)(u+u^{-1}-2)\Bigg\|\|z_1\| \\ & < 9\nu^4(\nu+1)\gamma \leq (2\nu)^5\gamma.
\end{align*}
As $\|x\|\leq 1$, this implies that 
\begin{align*}
\|[x,y_1]\| & \leq \|[x,y_1-h(1-h)(u+u^{-1}-2)z_1]\| \\ & \quad \quad \quad \quad \quad \quad +\|[x,h(1-h)(u+u^{-1}-2)z_1]\| \\
& < (2\nu)^5\gamma +\|[x,h(1-h)(u+u^{-1}-2)z_1]\|.
\end{align*}
Hence we see that 
\begin{align}\label{x y1}
\|[x,y_1]\| & < (2\nu)^5\gamma+\|[[x,h(1-h)],(u+u^{-1}-2)z_1]\| \nonumber\\ &\quad \quad\quad \quad \quad \quad+\|[h(1-h)x,(u+u^{-1}-2)z_1]\|\nonumber \\ & \quad \quad\quad \quad\quad \quad\quad\quad+\|[h(1-h),(u+u^{-1}-2)z_1]x\|.
\end{align}
Looking at line \eqref{z1 def}, every entry of the matrix $(u+u^{-1}-2)z_1$ is a sum of at most $8$ elements from the set $S^4$, where $S$ is as in line \eqref{S def}.  Hence by line \eqref{S est}, we see that 
\begin{equation}\label{h(1-h),(u+u^{-1}-2)z_1}
\|[h(1-h),(u+u^{-1}-2)z_1]\|< 4\cdot 2\cdot 8\cdot 4\cdot (2\nu)^6\cdot 12\nu^2\gamma\leq 2^{18}\nu^8\gamma.
\end{equation}
We have $\|[x,h(1-h)]\|< 2\gamma$, and line \eqref{z1 norm} implies 
$$
\|(u+u^{-1}-2)z_1\|\leq (4\nu^2+2)\cdot 9\nu^4\leq 2^6\nu^6,
$$ 
whence
\begin{equation}\label{x,h(1-h)],(u+u^{-1}-2)z_1}
\|[[x,h(1-h)],(u+u^{-1}-2)z_1]\| \leq 2^8\nu^6\gamma.
\end{equation}
Combining lines \eqref{x y1}, \eqref{h(1-h),(u+u^{-1}-2)z_1}, and \eqref{x,h(1-h)],(u+u^{-1}-2)z_1} thus implies that 
\begin{equation}\label{x y1 2}
\|[x,y_1]\|\leq 2^{19}\nu^8\gamma+\|[h(1-h)x,(u+u^{-1}-2)z_1]\|.
\end{equation}
Note now that for every element $s\in S$ we have that at least one of the following holds: (a) $\|[s,x]\|< 16\nu^2\gamma$ for all $x\in X$; or (b) $\|[s,(1-h)x]\|< 16\nu^2\gamma$ for all $x\in X$; or (c) $\|[s,(1-h)x]\|< 16\nu^2\gamma$ for all $x\in X$; or (d) $\|[s,h(1-h)x]\|< 16\nu^2\gamma$ for all $x\in X$.  In any of these cases, using that $\|[s,h]\|\leq 12\nu^2\gamma$ for any $s\in S$, we get that for any $s\in S$ and $x\in X$, $\|[s,h(1-h)x]\|< 40\nu^2\gamma$.  Applying Lemma \ref{com prod}, we therefore see that 
\begin{equation}\label{S est 2}
s\in S^n \quad \Rightarrow \quad \|[h(1-h)x,s]\|< n (2\nu)^{2(n-1)} 40\nu^2\gamma.
\end{equation}
As we have observed above already, every entry in the matrix $(u+u^{-1}-2)z_1$ is a sum of at most $8$ elements from the set $S^4$, where $S$ is as in line \eqref{S def}.  From line \eqref{S est 2} we therefore see that
$$
\|[h(1-h)x,(u+u^{-1}-2)z_1]\|< 4\cdot 4\cdot (2\nu)^4\cdot 40\nu^2\gamma\leq 2^{14}\nu ^6\gamma.
$$
Combining this with line \eqref{x y1 2} above therefore implies 
$$
\|[x,y_1]\|< 2^{20}\nu^8\gamma.
$$
\item \label{y_2} Now we look at the element $y_2$ from line \eqref{y1 y2 y3} above.  If $x\in X$, we see that 
\begin{equation}\label{x y_2}
[x,y_2]=\Bigg[xh,\begin{pmatrix} 1-q & u_hp \\ -u_h^{-1}q & 1-p \end{pmatrix}\Bigg]+\Bigg[h,\begin{pmatrix} 1-q & u_hp \\ -u_h^{-1}q & 1-p \end{pmatrix}\Bigg]x.
\end{equation}
We have that 
$$
\Bigg[h,\begin{pmatrix} 1-q & u_hp \\ -u_h^{-1}q & 1-p \end{pmatrix}\Bigg]=\begin{pmatrix} [q,h] & [h,u_hp] \\ [u_h^{-1}q,h] & [p,h] \end{pmatrix}.
$$
Each entry in the matrix on the right is the commutator of $h$ with an element of $S^2$, where $S$ is as in line \eqref{S def} above.  Hence by line \eqref{S est}, we see that 
$$
\Bigg\|\Bigg[h,\begin{pmatrix} 1-q & u_hp \\ -u_h^{-1}q & 1-p \end{pmatrix}\Bigg]\Bigg\|< 4\cdot 2\cdot (2\nu)^2\cdot 12\nu^2\gamma\leq 2^9\nu^4\gamma.
$$
Combining this with line \eqref{x y_2} gives 
\begin{equation}\label{x y_2 2}
\|[x,y_2]\|< \Bigg\|\Bigg[xh,\begin{pmatrix} 1-q & u_hp \\ -u_h^{-1}q & 1-p \end{pmatrix}\Bigg]\Bigg\|+2^9\nu^4\gamma.
\end{equation}
On the other hand 
\begin{align}\label{xh com}
\Bigg[xh,\begin{pmatrix} 1-q & u_hp \\ -u_h^{-1}q & 1-p \end{pmatrix}\Bigg] & =\Bigg[[x,h],\begin{pmatrix} 1-q & u_hp \\ -u_h^{-1}q & 1-p \end{pmatrix}\Bigg] \nonumber \\ & \quad \quad\quad\quad\quad +\Bigg[hx,\begin{pmatrix} 1-q & u_hp \\ -u_h^{-1}q & 1-p \end{pmatrix}\Bigg].
\end{align}
As $\|[h,x]\|< \gamma$, we have 
$$
\Bigg\|\Bigg[[x,h],\begin{pmatrix} 1-q & u_hp \\ -u_h^{-1}q & 1-p \end{pmatrix}\Bigg]\Bigg\|\leq 2\gamma\Bigg\| \begin{pmatrix} 1-q & u_hp \\ -u_h^{-1}q & 1-p \end{pmatrix}\Bigg\|.
$$
As $\|1-p\|\leq \nu$ and $\|1-q\|\leq \nu$ by Corollary \ref{id norm cor}, every entry of the matrix on the right has norm at most $\nu^2$, and so 
$$
\Bigg\|\Bigg[[x,h],\begin{pmatrix} 1-q & u_hp \\ -u_h^{-1}q & 1-p \end{pmatrix}\Bigg]\Bigg\|< 2^3\nu^2\gamma.
$$
Hence line \eqref{xh com} implies that 
\begin{equation}\label{xh com 2}
\Bigg\|\Bigg[xh,\begin{pmatrix} 1-q & u_hp \\ -u_h^{-1}q & 1-p \end{pmatrix}\Bigg]\Bigg\|< \Bigg\|\Bigg[hx,\begin{pmatrix} 1-q & u_hp \\ -u_h^{-1}q & 1-p \end{pmatrix}\Bigg]\Bigg\|+2^3\nu^2\gamma.
\end{equation}
The commutator appearing on the right above equals 
$$
\begin{pmatrix} [q,hx] & [hx,u_h]p+u_h[hx,p] \\ [u_h^{-1},hx]q-u_h^{-1}[hx,q] & [p,hx] \end{pmatrix}.
$$
Using that $u_h\in \mathcal{U}^1_{n,\nu,\gamma}(hX,B)$, and applying Lemma \ref{cut proj lem}, the norm of each entry above is at most $2\nu\gamma$, whence 
$$
\Bigg\|\Bigg[hx,\begin{pmatrix} 1-q & u_hp \\ -u_h^{-1}q & 1-p \end{pmatrix}\Bigg]\Bigg\|< 2^3\nu\gamma.
$$
Combining this with lines \eqref{x y_2 2} and \eqref{xh com 2} therefore implies that 
$$
\|[x,y_2]\|< 2^{10}\nu^4\gamma.
$$
\item \label{y_3} Finally, we look at $y_3$ from line \eqref{y1 y2 y3}.  This can be handled very similarly to the case of $y_2$, giving the estimate $\|[x,y_3]\|< 2^{10}\nu^4\gamma$ for all $x\in X$; we leave the details to the reader.
\end{enumerate}
Putting together the concluding estimates of points \eqref{y_1}, \eqref{y_2}, and \eqref{y_2} above, we see that $\|[x,vw^{-1}]\|< 2^{21}\nu^8\gamma$ for all $x\in X$.  Combining this with line \eqref{[h,vw]}, we see that 
\begin{equation}\label{[x,vw]}
\|[x,vw^{-1}]\|< 2^{21}\nu^{9}\gamma
\end{equation}
for all $x\in X\cup \{h\}$.  

To complete the proof, let us estimate $\|[x,wv^{-1}]\|$ for $x\in X\cup \{h\}$.  Using the formula $[x,wv^{-1}]=wv^{-1}[vw^{-1},x]wv^{-1}$, we see that 
$$
\|[x,wv^{-1}]\|\leq \|wv^{-1}\|\|[vw^{-1},x]\|\|wv^{-1}\|.
$$
Lines \eqref{[x,vw]} and \eqref{vw norm} therefore imply that 
$$
\|[x,wv^{-1}]\|\leq 2^{37}\nu^{25}\gamma
$$
and we are finally done.
\end{proof}

Finally, we are ready for the proof of Proposition \ref{bound exact}.

\begin{proof}[Proof of Proposition \ref{bound exact}]
With notation as in the statement, let $(p,q)\in \mathcal{P}^{1}_{n,\kappa,\epsilon}(X\cup Y_h\cup Y_{1-h}\cup \{h\},B)$, and assume that the images of $[p,q]$ in $KK^0_{\lambda,\delta}(hX\cup Y\cup \{h\},B)$ and $KK^0_{\lambda,\delta}((1-h)X\cup Y\cup \{h\},B)$ under the maps in lines \eqref{eta h} and \eqref{eta 1-h} are zero. 

Note first that the map in line \eqref{eta h} is induced by the identity map on cycles, so Lemma \ref{kk0 zero} applied to the cycle $(p,q)$ in $\mathcal{P}_{n,\lambda,\delta}(hX\cup Y\cup \{h\},B)$ implies that there exists $k\in \N$ such that $(p\oplus 1_k\oplus 0_k,q\oplus 1_k\oplus 0_k)$ is in the same path component of $\mathcal{P}_{n+2k,2\lambda,\delta}(hX\cup Y\cup \{h\},B)$ as an element of the form $(r,r)$.  Replacing $(r,r)$ with $(yry^*,yry^*)$ for some appropriate unitary $y\in M_{n+2k}(\C)$ and using that the unitary group of $M_{n+2k}(\C)$ is connected, we may assume that $(r,r)$ is in $\mathcal{P}^1_{n+2k,2\lambda,\delta}(hX\cup Y_h\cup \{h\},B)$ (see Definition \ref{nize proj} for notation).  Moreover, as $(p,q)\in \mathcal{P}^{1}_{n,\lambda,\delta}(X\cup Y_h\cup Y_{1-h}\cup \{h\},B)$ there is a unitary $z\in M_{n+2k}(\C)$ such that  $(z(p\oplus 1_k\oplus 0_k)z^*,z(q\oplus 1_k\oplus 0_k)z^*)$ is in $\mathcal{P}^1_{n,\lambda,\delta}(hX\cup Y\cup\{h\},B)$.   As the elements $(r,r)$ and $(z(p\oplus 1_k\oplus 0_k)z^*,z(q\oplus 1_k\oplus 0_k)z^*)$ of $\mathcal{P}^1_{n,2\lambda,\delta}(hX\cup Y\cup\{h\},B)$ are connected by a path $\mathcal{P}_{n,2\lambda,\delta}(hX\cup Y\cup\{h\},B)$, we may use Proposition \ref{proj same image} part \eqref{hom np} to connect them by a path in $\mathcal{P}^1_{n,2\lambda,4\delta}(hX\cup Y\cup\{h\},B)$.  Precisely analogously (increasing $k$ if necessary), we may assume that $(z(p\oplus 1_k\oplus 0_k)z^*,z(q\oplus 1_k\oplus 0_k)z^*)$ is in the same path component of $\mathcal{P}^1_{n,2\lambda,4\delta}((1-h)X\cup Y_{1-h}\cup\{h\},B)$ as an element of the form $(s,s)$.

For notational simplicity, write $m=n+2k$, and let $M:=4\cdot2^{(200\lambda)^3}$.  Then (with notation as in Definition \ref{nize uni}), Lemma \ref{hom to sim} gives $j\in \N$ and elements 
$$
u_h\in \mathcal{U}^1_{m+2j,M,M\delta}(hX\cup\{h\}\cup Y,B)
$$ 
and 
$$
u_{1-h}\in \mathcal{U}^1_{m+2j,M,M\delta}((1-h)X\cup \{h\}\cup Y,B)
$$ 
such that 
\begin{equation}\label{bigconj h}
u_h(z(p\oplus 1_k\oplus 0_k)z^* \oplus 1_j\oplus 0_j)u_h^{-1}=z(q\oplus 1_k\oplus 0_k)z^*\oplus 1_j\oplus 0_j
\end{equation}
and 
\begin{equation}\label{bigconj 1-h}
u_{1-h}(z(p\oplus 1_k\oplus 0_k)z^*\oplus 1_j\oplus 0_j)u_{1-h}^{-1}=z(q\oplus 1_k\oplus 0_k)z^*\oplus 1_j\oplus 0_j.
\end{equation}
For notational simplicity, rename $z(p\oplus 1_k\oplus 0_k)z^*\oplus 1_j \oplus 0_j$ and $z(p\oplus 1_k\oplus 0_k)z^*\oplus 1_j\oplus 0_j$ as $p$ and $q$ respectively and rewrite $m+2j$ as $n$: if the conclusion of the proposition holds for this new pair then it also holds for the original pair thanks to the definition of the controlled $KK^0$ groups (see Definition \ref{alm com 0}), so this makes no real difference.  In this new language, lines \eqref{bigconj h} and \eqref{bigconj 1-h} can be rewritten $u_hpu_h^{-1}=q$ and $u_{1-h}pu_{1-h}^{-1}=q$ respectively.  

Define now 
\begin{equation}\label{pre u}
u:=u_{1-h}(1-p)+u_hp,
\end{equation}
which we claim has the properties in the statement.  Using Lemma \ref{u exist} with $\nu=M$ and $\gamma=M\delta$, we see that (with notation as in Definition \ref{nize uni}), $u$ is an element of $\mathcal{U}^1_{n,2M^2,10M^2 \delta}(h(1-h)X\cup\{h\}\cup Y,B)$.  Recalling that $M=4\cdot 2^{(200\lambda)^3}$, we see that $N_1(\lambda)=2^{9000000\lambda^3}$ has the desired property.  

To complete the proof, it remains to show that if $N_2=N_2(\mu)=2^{252000000\mu^3}$, then $\partial[u]=[p,q]$ in $KK^0_{N_2,N_2\gamma}(X\cup \{h\},B)$.

Now, $v:=v(u,h)$ is as in line \eqref{v def}, we have
$$
\partial[u]=\Bigg[ v\begin{pmatrix} 1 & 0 \\ 0 & 0 \end{pmatrix} v^{-1}\,,\,\begin{pmatrix} 1 & 0 \\ 0 & 0 \end{pmatrix}\Bigg].
$$
Define now 
\begin{equation}\label{w def}
w:=\begin{pmatrix} u_{1-h}(1-p) & -q \\ p & (1-p)u_{1-h}^{-1}\end{pmatrix}\in M_{2n}(\LB).
\end{equation}
Applying Lemma \ref{vw lem} with $\nu=M$ and $\gamma=M\delta$, we see that $w$ is in $\mathcal{U}_{2n,(2M)^8,2^{37}M^{25}\delta}(X\cup\{h\},B)$.  For notational simplicity, set $M_1:=2^{37}M^{25}$.  Proposition \ref{sim to hom} implies that in $KK^0_{M_1^3, 3M_1^3\delta}(X\cup\{h\},B)$
\begin{align*}
\partial[u]& =\Bigg[ v\begin{pmatrix} 1 & 0 \\ 0 & 0 \end{pmatrix} v^{-1}\,,\,\begin{pmatrix} 1 & 0 \\ 0 & 0 \end{pmatrix}\Bigg] \\
& =\Bigg[ (vw^{-1})^{-1}v\begin{pmatrix} 1 & 0 \\ 0 & 0 \end{pmatrix} v^{-1}(vw^{-1})\,,\,\begin{pmatrix} 1 & 0 \\ 0 & 0 \end{pmatrix}\Bigg] \\
&= \Bigg[ w\begin{pmatrix} 1 & 0 \\ 0 & 0 \end{pmatrix} w^{-1}\,,\,\begin{pmatrix} 1 & 0 \\ 0 & 0 \end{pmatrix}\Bigg].
\end{align*}
Computing, we see that 
$$
w\begin{pmatrix} 1 & 0 \\ 0 & 0 \end{pmatrix} w^{-1}=\begin{pmatrix} 1-q & 0 \\ 0 & p \end{pmatrix},
$$
whence
\begin{equation}\label{pen u}
\partial[u]=\Bigg[ \begin{pmatrix} 1-q & 0 \\ 0 & p \end{pmatrix}\,,\,\begin{pmatrix} 1 & 0 \\ 0 & 0 \end{pmatrix}\Bigg]
\end{equation}
in the group $KK^0_{M_1^3, 3M_1^3\delta}(X\cup\{h\},B)$.  

Note now that the matrix $\begin{pmatrix} 1-q & q \\ q & 1-q\end{pmatrix}\in M_{2n}(\K_B^+)$ has norm at most $2\lambda$ (as $\|q\|\leq \kappa\leq \lambda$, and so $\|1-q\|\leq \lambda$ by Corollary \ref{id norm cor}), and that it satisfies 
$$
\Bigg\|\Bigg[ x,\begin{pmatrix} 1-q & q \\ q & 1-q\end{pmatrix}\Bigg]\Bigg\|<\epsilon<\delta
$$
for all $x\in X\cup\{h\}$.  Hence $\begin{pmatrix} 1-q & q \\ q & 1-q\end{pmatrix}\in \mathcal{U}_{2n,2\lambda,\delta}(X\cup\{h\},B)$.   Applying Proposition \ref{sim to hom} again and using that $\lambda\leq M_1$, the identity 
$$
\begin{pmatrix} 1-q & q \\ q & 1-q\end{pmatrix}\begin{pmatrix} 1 & 0 \\ 0 & 0 \end{pmatrix}\begin{pmatrix} 1-q & q \\ q & 1-q\end{pmatrix}=\begin{pmatrix} 1-q & 0 \\ 0 & q \end{pmatrix}
$$
shows that the class on the right hand side of line \eqref{pen u} is the same as the class
$$
\Bigg[ \begin{pmatrix} 1-q & 0 \\ 0 & p \end{pmatrix}\,,\,\begin{pmatrix} 1-q & 0 \\ 0 & q \end{pmatrix}\Bigg]
$$
in $KK^0_{M_1^6, 9M_1^9\delta}(X\cup\{h\},B)$.  Using a rotation homotopy, this is the same as $[p,q]$ by definition of $KK^0_{M_1^6, 9M_1^9\delta}(X\cup\{h\},B)$; recalling that $M_1:=2^{37}M^{25}$, $M=4\cdot 2^{(200\lambda)^3}$, and that $\mu\geq 2^{9000000\lambda^3}$ we see that $N_2(\mu)=2^{37}\mu^{25}$ indeed has the right properties.
\end{proof}

\section{Main theorems}\label{dec sec}

In this section (as throughout), if $B$ is a separable $C^*$-algebra, then $\LB$ and $\K_B$ are respectively the adjointable and compact operators on the standard Hilbert $B$-module $\ell^2\otimes B$.  We identify $\LB$ with the ``diagonal subalgebra'' $1_{M_n}\otimes \LB\subseteq M_n\otimes \LB=M_n(\LB)$ for each $n$.  

In this section we prove our main result: the class of separable and nuclear $C^*$-algebras with the UCT is closed under decomposability.

\subsection{Two technical `local' controlled vanishing results}

In order to make the structure of the proof of Theorem \ref{main} as clear as we can, in this subsection we split off two `local' technical results.  These are based on our work in Sections \ref{fd sec} and \ref{bound sec}; given the material in these earlier sections, at this point the proofs are essentially book-keeping. 



The next result is the first key technical ingredient we need: it is based on the material from Section \ref{fd sec}.  For the statement, recall that if $x$ and $S$ are respectively an element and subset of a metric space, and $\epsilon>0$, then ``$x\in_\epsilon S$'' means that there is $s\in S$ with $d(x,s)<\epsilon$.

\begin{proposition}\label{uct incl}
There exists a function $M:[1,\infty)\to [1,\infty)$ with the following property.  Let $\kappa\geq 1$, and let $M:=M(\kappa)$.  Let $B$ be a separable $C^*$-algebra such that $K_*(B)=0$.  Let $\epsilon>0$, and let $X$ be a finite subset of the unit ball of $\LSB$.  Let $F\subseteq \LSB$ be a separable, nuclear, unital $C^*$-subalgebra of $\LSB$ such that the identity representation $F\to \LSB$ is strongly unitally absorbing (see Definition \ref{usa def}), such that for all $x\in X$, $x\in_\epsilon F$, and such that $F$ satisfies the UCT.

Then for each $i\in \{0,1\}$ there exists a finite subset $Z$ of $F_1$ such that the forget control map 
$$
KK^i_{\kappa,\epsilon}(Z,B)\to KK^i_{M,M\epsilon}(X,B)
$$
of Definition \ref{dir set 2} (for $i=0$) or Definition \ref{dir set 3} (for $i=1$) is zero.
\end{proposition}

\begin{proof}
Let us focus on the case of $i=0$ first.  Let $Y$ be a finite subset of $F_1$ such that for all $x\in X$ there exists $y\in Y$ with $\|x-y\|< \epsilon$.  Then for any $n$, any $\delta>0$, we see that with notation in Definition \ref{alm com 0}
$$
\mathcal{P}_{n,\kappa,\delta}(Y,SB)\subseteq \mathcal{P}_{n,\kappa,\delta+2\kappa\epsilon}(X,SB).
$$
Hence the forget control map 
\begin{equation}\label{prelim forget}
KK_{\kappa,\delta}^0(Y,SB)\to KK^0_{\kappa,\delta+2\kappa\epsilon}(X,SB)
\end{equation}
is defined.  On the other hand, Corollary \ref{vanish cor 2} implies that there is a finite subset $Z$ of $F_1$ such that the forget control map 
$$
KK_{\kappa,\epsilon}^0(Z,SB)\to KK^0_{\kappa,160\epsilon}(Y,SB)
$$
is defined and zero.  Taking $\delta=160\epsilon$, and composing this with the forget control map in line \eqref{prelim forget} above, we see that the forget control map
$$
KK_{\kappa,\epsilon}^0(Z,SB)\to KK^0_{\kappa,(160+2\kappa)\epsilon}(Y,SB)
$$
is well-defined and zero.  We are therefore done in the case $i=0$: any function $M$ satisfying $M(\kappa)\geq 160+2\kappa$ will work.  The case of $i=1$ is similar (although requiring a much larger $M(\kappa)$), using Lemma \ref{uct reform odd case} in place of Corollary \ref{vanish cor 2}.
\end{proof}

The second key technical result we need is as follows: it is based on the material from Section \ref{bound sec}.  

\begin{proposition}\label{mv vanish technical}
Let $X$ be a finite subset of the unit ball of $\LB$, let $\epsilon>0$, and let $\kappa\geq 1$.  Assume there exists a positive contraction $h\in \LB$, finite self-adjoint subsets $Z_h$, $Z_{1-h}$, and $Z_{h(1-h)}$ of the unit ball of $\LB$, and $\lambda,\mu\geq 1$ and $\delta,\gamma> 0$ with the following properties:
\begin{enumerate}[(i)]
\item \label{mv final com} $\|[h,x]\|<\epsilon$ for all $x\in X$;
\item \label{mv final in} for each $z\in Z_{h(1-h)}$, $z\in_\epsilon Z_h$ and $z\in_\epsilon Z_{1-h}$;
\item \label{zh1-h ass} with $N_1:=N_1(\lambda)$ as in Proposition \ref{bound exact}, the forget control map 
$$
KK^1_{N_1,N_1\delta}(h(1-h)X\cup\{h\}\cup Z_{h(1-h)})\to KK^1_{\mu,\gamma}(h(1-h)X\cup\{h\},B)
$$
os Definition \ref{dir set 3} is defined and zero;
\item \label{zh ass} the forget control map
\begin{align*}
KK^0_{4\kappa^2,2\epsilon}(Z_h & \cup hX\cup\{h\},B) \\&   \to KK^0_{\lambda,\delta}(hX\cup\{h\}\cup Z_{h(1-h)},B)
\end{align*}
of Definition \ref{dir set 2} is defined and zero;
\item \label{z1-h ass} the forget control map
\begin{align*}
KK^0_{4\kappa^2,2\epsilon} (Z_{1-h}&\cup (1-h)X\cup\{h\},B) \\&   \to KK^0_{\lambda,\delta}((1-h)X\cup\{h\}\cup Z_{h(1-h)},B)
\end{align*}
of Definition \ref{dir set 2} is defined and zero.
\end{enumerate}
Then if $Z:=Z_h\cup Z_{1-h}\cup X\cup \{h\}$ and $N_2:=N_2(\mu)$ is as in Proposition \ref{bound exact}, we have that the forget control map 
$$
KK^0_{\kappa,\epsilon}(Z,B)\to KK^0_{N_2,N_2\gamma}(X,B)
$$
of Definition \ref{dir set 2} is zero.
\end{proposition}

\begin{proof}
We need to show that an arbitrary class $\alpha\in KK^0_{\kappa,\epsilon}(X,B)$ vanishes under the forget control map
$$
KK^0_{\kappa,\epsilon}(Z,B)\to KK^0_{N_2,N_2\gamma}(X,B).
$$
Using Proposition \ref{proj same image} part \eqref{same np}, with notation as in Definition \ref{nize proj}, we may assume that there is a cycle $(p,q)\in \mathcal{P}^1_{n,4\kappa^3,\epsilon}(Z,B)$ such that $[p,q]\in KK^0_{4\kappa^3,\epsilon}(Z,B)$ agrees with the image of $\alpha$ under the forget control map 
$$
KK^0_{\kappa,\epsilon}(Z,B)\to KK^0_{4\kappa^3,\epsilon}(Z,B).
$$
It thus suffices to show that $[p,q]\in KK^0_{4\kappa^3,\epsilon}(Z,B)$ vanishes under the forget control map
$$
KK^0_{4\kappa^2,\epsilon}(Z,B)\to KK^0_{N_2,N_2\gamma}(X,B)
$$
(we leave the check that this map is defined under our assumptions to the reader).  Now, with notation as in Proposition \ref{bound exact}, the composition 
$$
\xymatrix{ KK^0_{4\kappa^2,\epsilon}(X\cup Z_h\cup Z_{1-h} \cup\{h\},B) \ar[d] & \\ KK^0_{4\kappa^2,\epsilon}(X\cup Z_h\cup\{h\},B) \ar[r]^-{\eta_h} & KK^0_{4\kappa^2,2\epsilon}(hX\cup Z_h\cup\{h\},B) \ar[d] \\
& KK^0_{\lambda,\delta}(hX\cup Z_{h(1-h)}\cup\{h\},B)}
$$
(compare line \eqref{eta h}) is the zero map: indeed, the right-hand vertical map is zero by assumption \eqref{zh ass}.  Similarly, using assumption \eqref{z1-h ass}, we see that the composition 
$$
\xymatrix{ KK^0_{4\kappa^2,\epsilon}(X\cup Z_h\cup Z_{1-h} \cup\{h\},B) \ar[d] & \\ KK^0_{4\kappa^2,\epsilon}(X\cup Z_{1-h}\cup\{h\},B) \ar[r]^-{\eta_{1-h}} & KK^0_{4\kappa^2,2\epsilon}((1-h)X\cup Z_{1-h}\cup\{h\},B) \ar[d] \\
& KK^0_{\lambda,\delta}((1-h)X\cup Z_{h(1-h)}\cup\{h\},B)}
$$
(compare line \eqref{eta 1-h}) is zero.  Hence Proposition \ref{bound exact} gives us an element 
$$
u\in \mathcal{U}^1_{\infty,N_1,N_1\delta}(h(1-h)X\cup\{h\}\cup Z_{h(1-h)},B)
$$ 
such that in the diagram below (with $N_0=N_0(\mu)$ as in Proposition \ref{boundary map 2})
\begin{equation}\label{mv bound diag}
\xymatrix{ KK^1_{N_1,N_1\delta}(h(1-h)X\cup\{h\}\cup Z_{h(1-h)},B) \ar[d] & KK^0_{4\kappa^2,2\epsilon}(Z,B)\ar[d] \\
 KK^1_{\mu,\gamma}(h(1-h)X\cup \{h\},B) \ar[r]^-\partial & KK^0_{N_0,N_0\gamma}(X\cup \{h\},B) \ar[d] \\
 & KK^0_{N_2,N_2\gamma}(X\cup \{h\},B)}
\end{equation}
the images of the classes $[u]\in KK^1_{N_1,N_1\delta}(h(1-h)X\cup\{h\}\cup Z_{h(1-h)})$ and $[p,q]\in KK^0_{\kappa,\epsilon}(Z,B)$ in the bottom right group $KK^0_{N_2,N_2\gamma}(X\cup \{h\},B)$ are the same; a fortiori their images are also the same if we further compose with the forget control map 
$$
KK^0_{N_2,N_2\gamma}(X\cup \{h\},B)\to KK^0_{N_2,N_2\gamma}(X,B).
$$
Assumption \eqref{zh1-h ass} implies, however, that the left-hand vertical map in line \eqref{mv bound diag} is zero, however, so we are done.
\end{proof}

\subsection{Proof of the main theorems}

We are now ready for our main results.  For the statement of the first of these, we recall what it means for a $C^*$-algebra to decompose over a class of $C^*$-algebras from Definition \ref{ais} above.  After giving a proof of this, we will use it to establish the theorems from the introduction.

\begin{theorem}\label{uct the}
Let $\kappa\geq 1$ and $\gamma>0$.  Let $M_1:=M(4)$ be as in Proposition \ref{uct incl}.  Let $N_1:=N_1(M_1)$ be as in Proposition \ref{bound exact}.  Let $M_2:=M(N_1)$ be as in Proposition \ref{uct incl}.  Let $N_2:=N_2(M_2)$ be as in Proposition \ref{bound exact}.   Then any $\nu\geq N_2$ and $\epsilon\in (0,\gamma(2N_2M_2N_1M_1)^{-1})$ have the following property.

Let $A$ be a separable, unital $C^*$-algebra that decomposes over the class of nuclear $C^*$-algebras that satisfy the UCT.  Let $B$ be any separable $C^*$-algebra such that $K_*(B)=0$.  Then for any finite subset $X$ of $A_1$, and $\epsilon>0$, there is a finite subset $Z$ of $A_1$, such that the forget control map
$$
KK^0_{\kappa,\epsilon}(Z,SB)\to KK^0_{\nu,\gamma}(X,SB)
$$
of Definition \ref{dir set 2} is defined and zero.  

In particular, $A$ satisfies the UCT.
\end{theorem}

\begin{proof}
The claim that $A$ satisfies the UCT follows as the vanishing property in the statement of Theorem \ref{uct the} implies condition \eqref{weak a vanish} from Corollary \ref{vanish cor 2}.  It thus suffices to prove the vanishing property.  Let $\nu$ and $\epsilon$ satisfy the given assumptions.  

As $A$ is decomposable with respect to the family of nuclear $C^*$-subalgebras that satisfy the UCT, there are nuclear, UCT $C^*$-subalgebras $C$, $D$ and $E$ of $A$ and a positive contraction $h\in E$ such that for all $x\in X$, $\|[h,x]\|<\epsilon$, $hx\in_\epsilon C$, $(1-h)x\in_\epsilon D$, and $h(1-h)x\in_\epsilon E$, and such that all $e\in E$ we have that $e\in_\epsilon C$, and $e\in_\epsilon D$.  Replacing $C$, $D$, and $E$ by the $C^*$-subalgebra of $A$ spanned by the algebra and the unit of $A$, we may assume that $C$, $D$, and $E$ are unital subalgebras of $A$ (note that the unitization of a nuclear $C^*$-algebra that satisfies the UCT is nuclear and satisfies the UCT: see \cite[Exercise 2.3.5]{Brown:2008qy} for nuclearity and \cite[Proposition 2.3 (a)]{Rosenberg:1987bh} for the UCT).

Represent $A$ on $\LSB$ using a representation with the properties in Corollary \ref{sub sua} (with $B$ replaced by $SB$), and identify $A$ (therefore also $C$, $D$, and $E$) with unital $C^*$-subalgebras of $\LSB$ using this representation.  Note that the restrictions of this representation to each of $E$, $C$, $D$, (and the representation of $A$ itself) are strongly unitally absorbing.

Throughout the rest of the proof, all unlabeled arrows are forget control maps as in Definitions \ref{dir set 2} or \ref{dir set 3} as appropriate.

Using Proposition \ref{uct incl} there exists a finite self-adjoint subset $Z_{h(1-h)}$ of $E_1$ such that the forget control map 
\begin{align}\label{ze forget 0}
KK^1_{N_1,2N_1M_1\epsilon} & (h(1-h)X\cup Z_{h(1-h)}\cup\{h\},SB) \\ \nonumber & \to KK^1_{M_2,2M_2N_1M_1\epsilon}(h(1-h)X\cup \{h\},SB) 
\end{align}
is zero.  Similarly, Proposition \ref{uct incl} and the facts that for all $z\in Z_{h(1-h)}\subseteq E_1$, $z\in_{\epsilon}C$ and $z\in_{\epsilon}D$ gives finite self-adjoint subsets $Z_h$ and $Z_{1-h}$ of $C_1$ and $D_1$ respectively such that the forget control maps
\begin{align}\label{zc forget}
KK^0_{4,2\epsilon} & (hX\cup Z_h\cup\{h\},SB) \nonumber  \\  & \to KK^0_{M_1,2M_1\epsilon}(hX\cup Z_{h(1-h)}\cup\{h\},SB)
\end{align}
and 
\begin{align}\label{zd forget}
KK^0_{4,2\epsilon} & ((1-h)X\cup Z_{1-h}\cup\{h\},SB) \nonumber \\ &\to KK^0_{M_1,2M_1\epsilon}((1-h)X\cup Z_{h(1-h)}\cup\{h\},SB)
\end{align}
are defined and zero.  Expanding $Z_h$ and $Z_{1-h}$ if necessary (using that for all $e\in E$, $e\in_\epsilon C$, and $e\in_\epsilon D$), we may assume that , 
\begin{equation}\label{zs in} 
\text{for all } z\in Z,~ z\in_\epsilon Z_h \text{ and } z\in_\epsilon Z_{1-h}.
\end{equation}  

We are now in a position to apply Proposition \ref{mv vanish technical} with the given $\epsilon$ and $\kappa$, $\lambda=M_1$, $\delta=2M_1\epsilon$, $\mu=M_2$ and $\gamma$ as given: assumption \eqref{mv final com} follows by choice of $h$; assumption \eqref{mv final in} follows from line \eqref{zs in}; assumption \eqref{zh1-h ass} follows as the map in line \eqref{ze forget 0} is zero; assumption \eqref{zh ass} follows as the map in line \eqref{zc forget} is zero; and assumption \eqref{z1-h ass} follows as the map in line \eqref{zd forget} is zero.  Therefore Proposition \ref{mv vanish technical} implies that the forget control map 
$$
KK^0_{\kappa,\epsilon}(Z,SB)\to KK^0_{\nu,\gamma}(X,SB)
$$
is zero and we are done.
\end{proof}

To establish the main results as stated in the introduction, we need a basic lemma.

\begin{lemma}\label{nuc lem}
The class of unital, nuclear $C^*$-algebras is closed under decomposability.
\end{lemma}

\begin{proof}
Let $A$ be a unital $C^*$-algebra that decomposes over the class of unital nuclear $C^*$-algebras.  Let a finite subset $X$ of $A$ and $\epsilon\in (0,1)$ be given.  To show that $A$ is nuclear, it will suffice to construct a finite rank ccp map 
$$
\phi:A\to A
$$
such that $\phi(x)\approx_\epsilon x$ for all $x\in X$ (compare for example \cite[IV.3.1.6, (iii)]{Blackadar:2006eq}).  We may assume that $X$ contains the unit of $A$.

Let then $C$, $D$, $E$\footnote{One does not actually need $E$ at all in the proof.}, and $h$ be as in the definition of decomposability (Definition \ref{ais}) with respect to the finite set $X$ and the parameter $\delta=\frac{1}{18}(\epsilon/(1+\epsilon))^2$, and with $C$ and $D$ nuclear.  Note that for any $x\in X$, $\|[h^{1/2},x]\|\leq \frac{5}{4}\|[h,x]\|^{1/2}$ by the main result of \cite{Pedersen:1993vh}, whence 
\begin{equation}\label{ped com}
\|hx-h^{1/2}xh^{1/2}\|\leq \frac{5}{4}\|[h,x]\|^{1/2}<\frac{5}{4}\delta^{1/2}<2\delta^{1/2};
\end{equation}
as $hx\in_\delta C$, and as $\delta<1$, this implies that $h^{1/2}xh^{1/2}\in_{3\delta^{1/2}} C$.  Choose a finite subset $Y$ of $C$ such that for all $x\in X$ there is $y_x\in Y$ with 
\begin{equation}\label{yx est}
\|y_x-h^{1/2}xh^{1/2}\|<3\delta^{1/2}.
\end{equation}
Similarly, there is a finite subset $Z$ of $D$ such that for all $x\in X$ there is $z_x\in Z$ with $\|z_x-(1-h)^{1/2}x(1-h)^{1/2}\|<3\delta^{1/2}$.  

Now, as $C$ and $D$ are nuclear there are diagrams 
$$
\xymatrix{ C\ar[dr]^{\psi_C} & & C \\ & F_C \ar[ur]^-{\phi_C} & } \quad \text{and} \quad \xymatrix{ D\ar[dr]^{\psi_D} & & C \\ & F_D \ar[ur]^-{\phi_D} & } \quad
$$
where: all the arrows are ccp maps; $F_C$ and $F_D$ are finite dimensional $C^*$-algebras; and for all $y\in Y$, and all $z\in Z$,
\begin{equation}\label{y and z est}
\phi_C(\psi_C(y))\approx_{\delta^{1/2}}y \quad \text{and}\quad \psi_D(\phi_D(z))\approx_{\delta^{1/2}}z.
\end{equation}
Using Arveson's extension theorem (see for example \cite[Theorem 1.6.1]{Brown:2008qy}), extend $\psi_C$ and $\psi_D$ to ccp maps defined on all of $A$, which we keep the same notation for.  Define 
$$
\phi_0:A\to A,\quad a\mapsto \phi_C(\psi_C(h^{1/2}xh^{1/2}))+\phi_D(\psi_D((1-h)^{1/2}x(1-h)^{1/2})),
$$
and note that $\phi_0$ is completely positive.  For any $x\in X$, let $y_x$ have the property in line \eqref{yx est}.  As $\psi_C$ is contractive, this and lines \eqref{y and z est} and \eqref{ped com} imply that 
$$
\phi_C(\psi_C(h^{1/2}xh^{1/2}))\approx_{3\delta^{1/2}}\phi(\psi_C(y_x))\approx_{\delta^{1/2}} y_x \approx_{3\delta^{1/2}}h^{1/2}xh^{1/2}\approx_{2\delta^{1/2}} hx.
$$
Precisely analogously, for any $x\in X$,
$$
\phi_D(\psi_D((1-h)^{1/2}x(1-h)^{1/2}))\approx_{9\delta^{1/2}} (1-h)x
$$
and so for any $x\in X$, $\phi_0(x)\approx_{18\delta^{1/2}} x$.  Applying this to $x=1$ implies in particular that $\|\phi_0\|=\|\phi_0(1)\|\geq  1-18\delta^{1/2}$.  Hence if we define 
$$
\phi:A\to A,\quad a\mapsto \frac{\phi_0(a)}{\|\phi_0(1)\|}
$$
then $\phi$ is a ccp map such that 
$$
\|\phi(x)-x\|\leq \frac{18\delta^{1/2}}{1-18\delta^{1/2}}
$$
for all $x\in X$.  Using the choice of $\delta$, this completes the proof.
\end{proof}

The next corollary is Theorem \ref{main} from the introduction: it is an immediate consequence of Lemma \ref{nuc lem} and Theorem \ref{uct the}.

\begin{corollary}\label{nu cor}
If a separable, unital $C^*$-algebra decomposes over the class of nuclear, unital $C^*$-algebras that satisfies the UCT, then it is nuclear and satisfies the UCT.\qed
\end{corollary}

The next result is Theorem \ref{complex cor} from the introduction.  For the definition of finite complexity and the classes $\mathcal{D}_\alpha$ used below, see Definition \ref{f c}.

\begin{corollary}\label{fc cor}
Let $\mathcal{C}$ be a class of separable, unital, nuclear $C^*$-algebras that satisfy the UCT.  Then the class of separable unital $C^*$-algebras that have finite complexity relative to $\mathcal{C}$ consists of nuclear $C^*$-algebras that satisfy the UCT.  

In particular, every separable $C^*$-algebra of finite complexity is nuclear and satisfies the UCT.
\end{corollary}

\begin{proof}
With notation as in Definition \ref{f c}, let $\mathcal{D}_0=\mathcal{C}$, and for each ordinal $\alpha$, let $\mathcal{D}_{\alpha,sep}$ consist of the separable $C^*$-algebras in the class $\mathcal{D}_\alpha$ from Definition \ref{f c}.  We proceed by transfinite induction to show that each $\mathcal{D}_{\alpha,sep}$ consists of nuclear, UCT $C^*$-algebras.  If $\alpha=0$, this is just the well-known fact that AF $C^*$-algebras satisfy the UCT.  If $\alpha>0$ (and either a successor or limit ordinal) then any $C^*$-algebra in $\mathcal{D}_{\alpha,sep}$ decomposes over $C^*$-algebras in $\bigcup_{\beta<\alpha}\mathcal{D}_{\beta,sep}$, and so is nuclear and UCT by Corollary \ref{nu cor} and the inductive hypothesis.  
\end{proof}

\appendix

\section{Examples}\label{examples app}

In this appendix we give some examples of $C^*$-algebras with finite complexity.   




\subsection{Cuntz algebras}\label{cuntz app}

The material in this section is based closely on work of Winter and Zacharias \cite[Section 7]{Winter:2010eb}\footnote{More specifically, it is based on the slightly different approach to the material in \cite[Section 7]{Winter:2010eb} suggested in \cite[Remark 7.3]{Winter:2010eb}.}.  Our aim is to establish the following result.

\begin{proposition}\label{ca d}
For any $n$ with $2\leq n<\infty$, the Cuntz algebra $\mathcal{O}_n$ has complexity rank one.
\end{proposition}

We should remark that the proof of Proposition \ref{ca d} uses classification results for Cuntz algebras, and so depends on prior knowledge of the UCT; it therefore cannot be said that Proposition \ref{ca d} gives a new proof of the UCT for Cuntz algebras (and even if it did, it would be quite a complicated one!).  Indeed, the main point of establishing Proposition \ref{ca d} for us is to use it as an ingredient in Theorem \ref{kirch the} from the introduction, not to establish the UCT.

We should also remark that Proposition \ref{ca d} was subsequently generalized in \cite[Theorem 1.5]{Jaime:2021vh}; nonetheless, we hope that the different argument given here still has some interest.

We now embark on the proof of Proposition \ref{ca d}.  We will follow the notation from \cite[Section 7]{Winter:2010eb}.  Fix $n\in \N$ with $n\geq 2$.  Let $H$ be an $n$-dimensional Hilbert space, with fixed orthonormal basis $\{e_1,...,e_n\}$.  Define 
\begin{equation}\label{gamman}
\Gamma(n):=\bigoplus_{l=0}^\infty H^{\otimes l},
\end{equation}
where $H^{\otimes l}$ is the $l^\text{th}$ tensor power of $H$ (and $H^{\otimes 0}$ is by definition a copy of $\C$).  Let $W_n$ be the set of all finite words based on the alphabet $\{1,...,n\}$.  In symbols 
$$
W_n:=\bigsqcup_{k=0}^\infty \{1,...,n\}^k
$$
(with $\{1,...,n\}^0$ by definition consisting only of the empty word).   For each $\mu=(i_1,....,i_k)\in W_n$, define $e_\mu:=e_{i_1}\otimes \cdots  \otimes e_{i_k}$, and define $e_\varnothing$ to be any unit-length element of $H^{\otimes 0}=\C$.  Then the set $\{e_{\mu}\mid \mu \in W_n\}$ is an orthonormal basis of $\Gamma(n)$.  For $\mu\in W_n$, write $|\mu|$ for the length of $\mu$, i.e.\ $|\mu|=k$ means that $\mu=(i_1,...,i_k)$ for some $i_1,..,i_k\in \{1,...,n\}$.  Then the canonical copy of $H^{\otimes k}$ inside $\Gamma(n)$ from line \eqref{gamman} has orthonormal basis $\{e_\mu\mid |\mu|=k\}$.

For each $i\in \{1,...,n\}$ let $T_i$ be the bounded operator on $\Gamma(n)$ that acts on basis elements via the formula 
$$
T_i:e_\mu\mapsto e_i\otimes e_\mu.
$$ 
The \emph{Cuntz-Toeplitz algebra} $\mathcal{T}_n$ is defined to be the $C^*$-subalgebra of $\mathcal{B}(\Gamma(n))$ generated by $T_1,...,T_n$.  We note that each $T_i$ is an isometry, and that $1-\sum_{i=1}^n T_iT_i^*$ is the projection onto the span of $e_\varnothing$.  It follows directly from this that $\mathcal{T}_n$ contains all matrix units with respect to the basis $\{e_\mu\}$ of $\Gamma(n)$, and therefore contains the compact operators $\mathcal{K}$ on $\Gamma(n)$.  Moreover, in the quotient $\mathcal{T}_n/\mathcal{K}$, the images $s_i$ of the generators $T_i$ satisfy the Cuntz relations $s_i^*s_i=1$ and $\sum_{i=1}^n s_is_i^*=1$, and therefore the quotient is a copy of the Cuntz algebra $\mathcal{O}_n$.

Now, for $x\in \R_+$, define $\lceil x\rceil:=\min\{n\in \N\mid n\geq x\}$, and define\footnote{In \cite[Section 7]{Winter:2010eb}, $\Gamma_{0,k}$ is written $\Gamma_{k,2k}$ and $\Gamma_{1,k}$ is written $\Gamma_{k+\lceil k/2\rceil,2k+\lceil k/2\rceil}$.}
\begin{equation}\label{gam01}
\Gamma_{0,k}:=\bigoplus_{l=k}^{2k-1} H^{\otimes l} \quad \text{and}\quad \Gamma_{1,k}:=\bigoplus_{l=k+\lceil k/2\rceil}^{2k+\lceil k/2\rceil}H^{\otimes l}.
\end{equation}
For $i\in \{0,1\}$, define $B_{i,k}^{(0)}:=\mathcal{B}(\Gamma_{i,k})$.  For each $l,m\in \N$, we identify $H^{\otimes l}\otimes H^{\otimes m}$ with $H^{\otimes (l+m)}$ via the bijection of orthonormal bases
$$
\big(e_{i_1}\otimes \cdots \otimes e_{i_l}\big)\otimes \big(e_{j_1}\otimes \cdots \otimes e_{j_m}\big)\leftrightarrow e_{i_1}\otimes \cdots \otimes e_{i_l}\otimes e_{j_1}\otimes \cdots \otimes e_{j_m}.
$$
Fix for the moment $k\in \N$ (it will stay fixed until Lemma \ref{alm com} below).  Then for each $j\in \N$ we get a canonical identification
$$
\Gamma_{0,k}\otimes H^{\otimes jk} = \bigoplus_{l=k}^{2k-1}H^{\otimes l}\otimes H^{\otimes jk} = \bigoplus_{l=jk}^{(j+1)k-1}H^{\otimes l}.
$$
Combining this with line \eqref{gamman} we get a canonical identification
$$
\Gamma(n)=\underbrace{\Big(\bigoplus_{l=0}^{k-1} H^{\otimes l}\Big)}_{=:H_0} \oplus \Big(\bigoplus_{j=0}^\infty \Gamma_{0,k}\otimes H^{\otimes jk}\Big).
$$
Let $\text{id}$ be the identity representation of $B_{0,k}^{(0)}$ on $\Gamma_{0,k}$ and write $B_{0,k}$ for the image of $B_{0,k}^{(0)}$ in the representation on $\Gamma(n)$ that is given by 
$$
0_{H_0}\oplus \Big(\bigoplus_{k=0}^\infty \text{id}\otimes 1_{H^{\otimes jk}}\Big)
$$
with respect to the above decomposition above.  Similarly, we get a decomposition 
$$
\Gamma(n)=\underbrace{\Big(\bigoplus_{l=0}^{k+\lceil k/2 \rceil-1} H^{\otimes l}\Big)}_{=:H_1} \oplus \Big(\bigoplus_{j=0}^\infty \Gamma_{1,k}\otimes H^{\otimes jk}\Big)
$$
and define $B_{1,k}$ to be the image of $B_{1,k}^{(0)}$ under the representation 
$$
0_{H_1}\oplus \Big(\bigoplus_{k=0}^\infty \text{id}\otimes 1_{H^{\otimes jk}}\Big).
$$

Now, let $f:[0,1]\to [0,1]$ be the function with graph pictured, where the non-differentiable points occur at the $x$ values $1/6$, $2/6$, $4/6$, and $5/6$.\\ \begin{center}\vspace{0cm} 
\includegraphics[width=6cm]{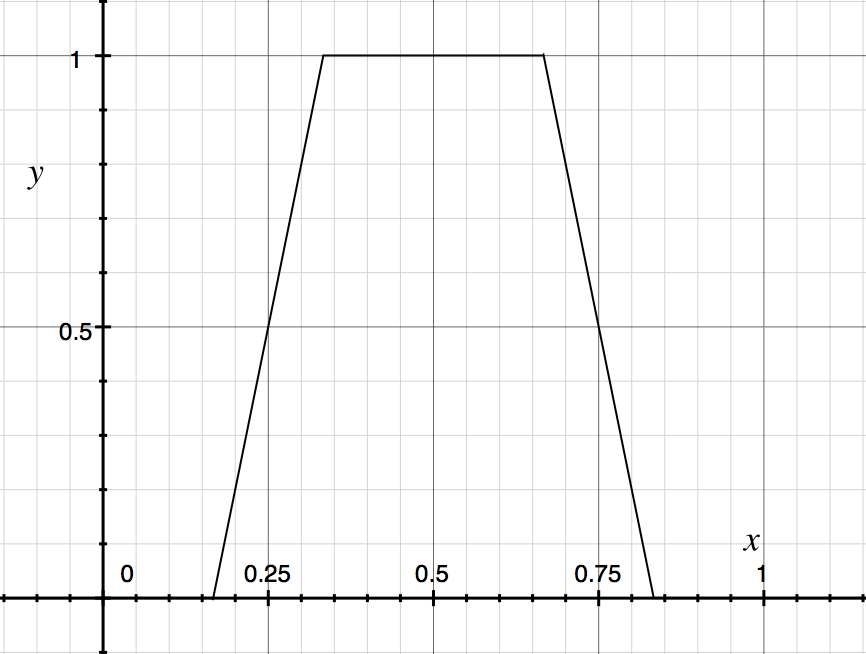} \vspace{0cm} \\\end{center}
Let $h_{0,k}^{(0)}\in B_{0,k}^{(0)}$ be the operator on $\Gamma_{0,k}$ that acts on the summand $H^{\otimes l}$ from line \eqref{gam01} by multiplication by the scalar $f((l-k)/(k-1))$.  Similarly, let $h_{1,k}^{(0)}\in B_{1,k}^{(1)}$ be the operator on $\Gamma_{1,k}$ that acts on the summand $H^{\otimes l}$ from line \eqref{gam01} by multiplication by the scalar $1-f((l-k-\lceil k/2\rceil)/(k-1))$.  Let $h_{0,k}$ and $h_{1,k}$ be the images of $h_{0,k}^{(0)}$ and $h_{1,k}^{(0)}$ in $B_{0,k}$ and $B_{1,k}$ respectively.  Note that the operator on $h_{0,k}+h_{1,k}$ on $\Gamma(n)$ acts on the summand on $H^{\otimes l}$ from line \eqref{gamman} by multiplication by $1$ as long as $l\geq k+\lceil k/2\rceil$.  In particular, 
\begin{equation}\label{alm id}
h_{0,k}+h_{1,k} \text{ equals the identity on $\Gamma(n)$ up to a finite rank perturbation.}
\end{equation}
We will need two technical lemmas about these operators.

\begin{lemma}\label{alm com}
For any $T$ in the Cuntz-Toeplitz algebra $\mathcal{T}_n$ and $i\in \{0,1\}$, we have that $\| [h_{i,k},T]\|\to0$ as $k\to\infty$.
\end{lemma}

\begin{proof}
We will focus on $h_{0,k}$: the case of $h_{1,k}$ is essentially the same.  It suffices to consider the case where $T$ is one of the canonical generators $T_i$ of the Cuntz-Toeplitz algebra.  Let $e_\mu$ be a basis element with $|\mu|=jk+l$ for some $j,l\in \N$ with $l\in \{0,...,k-1\}$.  Then we compute that $[h_{0,k},T_i]e_\mu=0$ if $j=0$, and that otherwise 
$$
[h_{0,k},T_i]e_\mu=  \big(f((l+1)/(k-1))-f(l/(k-1))\big)e_i\otimes e_\mu.
$$
As the elements $\{e_i\otimes e_\mu\mid \mu\in W_n\}$ are an orthonormal set, this implies that 
$$
\|[h_{0,k},T_i]\|\leq \max_{l\in \{0,...,k-1\}}|f((l+1)/(k-1))-f(l/(k-1))|.
$$
The choice of function $f$ implies that the right hand side above is approximately $6/k$, so we are done.
\end{proof}

\begin{lemma}\label{alm in}
For any $T$ in the Cuntz-Toeplitz algebra $\mathcal{T}_n$ we have that:
\begin{enumerate}[(i)]
\item for $i\in \{0,1\}$,  $d(h_{i,k}T,B_{i,k})\to 0$ as $k\to\infty$;
\item $d(h_{0,k}h_{1,k}T,B_{0,k}\cap B_{1,k})\to 0$ as $k\to\infty$.
\end{enumerate}
\end{lemma}

\begin{proof}
We will focus on the case of $h_{0,k}$: the other cases are similar.  It suffices to consider $T$ a finite product $S_1...S_m$, where each $S_j$ is either one of the generators $T_i$ or its adjoint.  Using Lemma \ref{alm com}, we see that $[h_{0,k}^{1/l},S_j]\to 0$ as $k\to\infty$ for any $j$, and any $l\in \N$ with $l\geq 1$.  Hence the difference 
$$
h_{0,k}S_1...S_m-\big(h_{0,k}^{1/(2m)}S_1h_{0,k}^{1/(2m)}\big)\big(h_{0,k}^{1/(2m)}S_2h_{0,k}^{1/(2m)}\big)\cdots \big(h_{0,k}^{1/(2m)}S_mh_{0,k}^{1/(2m)}\big)
$$
tends to zero as $k\to\infty$.  It thus suffices to prove that the distance between each of the terms $h_{0,k}^{1/(2m)}S_jh_{0,k}^{1/(2m)}$ and $B_{0,k}$ tends to zero as $k\to\infty$.  Define $p_k$ to be the strong operator topology limit of $h_{0,k}^{1/l}$ as $l\to\infty$; in other words, $p_k$ is the support projection of $h_{0,k}$.  Then we have that $h_{0,k}^{1/(2m)}S_jh_{0,k}^{1/(2m)}=h_{0,k}^{1/(2m)}p_kS_jp_kh_{0,k}^{1/(2m)}$.  As $h_{0,k}^{1/(2m)}$ is in $B_{0,k}$, it suffices to prove that the distance between $p_kT_ip_k$ and $B_{0,k}$ tends to zero as $k\to\infty$.  However, $p_kT_ip_k$ is actually in $B_{0,k}$, so we are done.
\end{proof}

Now, as in the discussion on \cite[page 488]{Winter:2010eb}, define 
$$
\Gamma_k(n):=\bigoplus_{l=0}^{k-1} H^{\otimes l}.
$$
For a word $\mu\in W_n$ in $\{1,...,n\}$, we may uniquely write $\mu=\mu_0\mu_1$, where the lengths $|\mu_0|$ and $|\mu_1|$ satisfy $|\mu_0|\in \{0,...,k-1\}$, and $|\mu_1|\in k\N$.  Then the bijective correspondence of orthonormal bases
$$
e_\mu\leftrightarrow e_{\mu_0}\otimes e_{\mu_1}
$$
gives rise to a decomposition 
$$
\Gamma(n)=\Gamma_k(n)\otimes \Gamma(n^k).
$$
Identify the $C^*$-algebra $\mathcal{B}(\Gamma_k(n))\otimes \mathcal{T}_{n^k}$ with its image in the representation on $\Gamma(n)$ arising from the above decomposition.  The following is essentially part of \cite[Lemma 7.1]{Winter:2010eb}.

\begin{lemma}\label{fd incl}
With notation as above, $\mathcal{B}(\Gamma_k(n))\otimes \mathcal{T}_{n^k}$ contains the finite-dimensional $C^*$-algebras we have called $B_{0,k}$ and $B_{1,k}$, and in particular also contains $h_{0,k}$ and $h_{1,k}$.
\end{lemma}

\begin{proof}
In the notation of \cite[Lemma 7.1]{Winter:2010eb}, $B_{0,k}=\Lambda_k(\mathcal{B}(\Gamma_{k,2k}))$, and $B_{1,k}=\Lambda_k(\mathcal{B}(\Gamma_{k+\lceil k/2\rceil,2k+\lceil k/2\rceil}))$.  Part (i) of \cite[Lemma 7.1]{Winter:2010eb} says exactly that the image of $\Lambda_k$ is contained in $\mathcal{B}(\Gamma_k(n))\otimes \mathcal{T}_{n^k}$, however, so we are done.
\end{proof}

It is explained on \cite[page 488]{Winter:2010eb} that $\mathcal{B}(\Gamma_k(n))\otimes \mathcal{T}_{n^k}$ contains $\mathcal{T}_n$, so we get a canonical inclusion. 
\begin{equation}\label{tn incl}
\mathcal{T}_n \to \mathcal{B}(\Gamma_k(n))\otimes \mathcal{T}_{n^k}.
\end{equation}
The dimension of $\Gamma_k(n)$ is $d_k:=1+n+n^2+\cdots + n^{k-1}$, so we may make the identification $\mathcal{B}(\Gamma_k(n))\otimes \mathcal{T}_{n^k}= M_{d_k}(T_{n^k})$.  With respect to this identification, the inclusion in line \eqref{tn incl} takes the compact operators on $\Gamma(n)$ to $M_{d_k}(\K(\Gamma(n^k)))$.  Taking the quotient by the compacts on both sides of line \eqref{tn incl} thus gives rise to an inclusion 
\begin{equation}\label{iota map}
\iota:\mathcal{O}_n\to M_{d_k}(\mathcal{O}_{n^k}).
\end{equation}
In this language, we get the following immediate corollary of Lemmas \ref{alm com} and \ref{alm in}.  To state it, let $q:\mathcal{B}(\Gamma(n))\to \mathcal{Q}(\Gamma(n))$ be the quotient map from the bounded operators on $\Gamma(n)$ to the Calkin algebra.  

\begin{corollary}\label{comb lem}
For any $a\in \mathcal{O}_n$, we have that the following all tend to zero as $k\to\infty$: $\|[q(h_{0,k}),\iota(a)]\|$, $\|[q(h_{1,k}),\iota(a)]\|$, $d(q(h_{0,k})\iota(a),q(B_{0,k}))$, $d(q(h_{1,k})\iota(a),q(B_{1,k}))$, and $d(q(h_{0,k}h_{1,k})\iota(a),q(B_{0,k}\cap B_{1,k}))$. \qed
\end{corollary}

We are finally ready for the proof of Proposition \ref{ca d}.

\begin{proof}[Proof of Proposition \ref{ca d}]
Let $\epsilon>0$, and let $X$ be a finite subset of the unit ball of $\mathcal{O}_n$.  Corollary \ref{comb lem} implies that for any large $k$ we have that for all $a\in X$ and $i\in \{0,1\}$, the quantities $\| [q(h_{i,k}),\iota(a)]\|$, $d(q(h_{i,k})\iota(a),q(B_{i,k}))$, and $d(q(h_{0,k}h_{1,k})\iota(a),q(B_{0,k}\cap B_{1,k}))$ are smaller than $\epsilon/2$.  We may assume moreover that $k\equiv 1$ modulo $n-1$.  Fix this $k$ for the remainder of the proof.

As discussed on \cite[page 488]{Winter:2010eb}, we have a canonical unital inclusion $\mathcal{O}_{n^k}\to \mathcal{O}_n$ by treating suitable products of the generators of $\mathcal{O}_n$ as generators of $\mathcal{O}_{n^k}$.  Moreover, $d_k$ is equal to $k$ modulo $n-1$.  It follows that the $K$-theory of $M_{d_k}(\mathcal{O}_n)$ is given by $\Z/(n-1)\Z$ in dimension zero and zero in dimension one, with the class $[1]$ of the unit in $K_0$ represented by the residue of $k$ in $\Z/(n-1)\Z$.  Hence the $K$-theory invariants of $M_{d_k}(\mathcal{O}_n)$ and $\mathcal{O}_n$ agree, as we are assuming that $k \equiv 1$ modulo $n-1$.  In particular, the Kirchberg-Phillips classification theorem (see for example \cite[Corollary 8.4.8]{Rordam:2002cs}) gives a unital isomorphism $M_{d_k}(\mathcal{O}_n)\cong \mathcal{O}_n$.  Combining this with the inclusion $\mathcal{O}_{n^k}\to \mathcal{O}_n$ mentioned above gives a unital inclusion 
\begin{equation}\label{beta map}
\beta:M_{d_k}(\mathcal{O}_{n^k})\to \mathcal{O}_n.
\end{equation}
Now, the composition $\beta\circ \iota:\mathcal{O}_n\to \mathcal{O}_n$ of $\beta$ as in line \eqref{beta map} and $\iota$ as in line \eqref{iota map} is a unital inclusion, whence necessarily induces an isomorphism on $K$-theory.  As $\mathcal{O}_n$ satisfies the UCT, $\beta\circ \iota$ is therefore a $KK$-equivalence (see for example \cite[Proposition 7.3]{Rosenberg:1987bh}).  Hence the uniqueness part of the Kirchberg-Phillips classification theorem (see for example \cite[Theorem 8.3.3, (iii)]{Rordam:2002cs}) implies that $\beta\circ \iota:\mathcal{O}_n\to \mathcal{O}_n$ is approximately unitarily equivalent to the identity.  Thus there is a sequence $(u_m)$ of unitaries in $\mathcal{O}_n$ such that $\|a-u_m\beta\iota(a)u_m^*\|\to 0$ for all $a\in \mathcal{O}_n$.  Choose $m$ large enough so that $\|a-u_m\beta\iota(a)u_m^*\|<\epsilon/2$ for all $a\in X$.

Set $h:=u_m\beta(q(h_{0,k}))u_m^*$, $C_0:=u_m\beta(q(B_{0,k}))u_m^*$, $D_0:=u_m\beta(q(B_{1,k}))u_m^*$, and $E_0:=u_m\beta(q(B_{1,k}\cap B_{0,k}))u_m^*$.  Set $C$ to be the $C^*$-subalgebra of $\mathcal{O}_n$ spanned by $C_0$ and the unit, and similarly for $D$ and $E$.  Our choices, plus the fact that $q(h_{0,k}+h_{1,k})=1$ (see line \eqref{alm id}), imply that this data satisfies the definition of decomposability (Definition \ref{ais}), so we are done.
\end{proof}

\subsection{Groupoids with finite dynamical complexity}\label{fdc app}

In this section, we give another interesting class of $C^*$-algebras with finite complexity: $C^*$-algebras of groupoids with finite dynamical complexity.  To avoid repeating the same assumptions, let us stipulate that throughout this appendix the word ``groupoid'' means ``locally compact, Hausdorff, \'{e}tale groupoid''; we will often also assume that $G$ has compact base space, but not always.  For background on this class of groupoids and their $C^*$-algebras, we recommend \cite[Section 5.6]{Brown:2008qy}, \cite[Section 2.3]{Renault:2009zr}, or \cite{Sims:2017aa}.

Note that if $G$ is a groupoid in this sense, then any open subgroupoid $H$ of $G$ (i.e.\ $H$ is an open subset of $G$ that is algebraically a groupoid with the inherited operations) is also a groupoid in this sense.  Again, to avoid too much repetition, let us say that the word ``subgroupoid'' means ``open subgroupoid''.

The following definitions are essentially contained in the authors' joint work with Guentner \cite[Definition A.4]{Guentner:2014bh}.

\begin{definition}\label{gpd decomp def}
Let $G$ be a groupoid, let $H$ be a subgroupoid of $G$, and let $\mathcal{C}$ be a set of subgroupoids of $G$.  We say that $H$ is \emph{decomposable} over $\mathcal{C}$ if for any compact subset $K$ of $H$ there exists an open cover $\{U_0, U_1\}$ of $r(K)\cup s(K)$ such that for each $i\in \{0,1\}$ the subgroupoid of $H$ generated by 
$$
\{h\in K\mid s(h)\in U_i\} 
$$
is contained in an element of $\mathcal{C}$. 
\end{definition}

\begin{definition}\label{fdc def}
For an ordinal number $\alpha$:
\begin{enumerate}[(i)]
\item if $\alpha=0$, let $\mathcal{C}_0$ be the class of groupoids $G$ such that for any compact subset $K$ of $G$ there is a subgroupoid $H$ of $G$ such that $K\subseteq H$, and such that the closure of $H$ is compact;
\item if $\alpha>0$, let $\mathcal{C}_\alpha$ be the class of groupoids that decompose over the collection of their subgroupoids in the class $\bigcup_{\beta<\alpha}\mathcal{C}_\beta$.
\end{enumerate}
We say that a groupoid $G$ has \emph{finite dynamical complexity} if $G$ is contained in $\mathcal{C}_\alpha$ for some ordinal $\alpha$.  If $G$ has finite dynamical complexity, the \emph{complexity rank} of $G$ is the smallest $\alpha$ such that $G$ is in $\mathcal{C}_\alpha$.
\end{definition}


The main result of this section is as follows.  For the statement, recall that a groupoid is \emph{ample} if it has totally disconnected base space, and \emph{principal} if the units are the elements $g\in G$ that satisfy $s(g)=r(g)$.  Recall also that a $C^*$-algebra is subhomogeneous if it is isomorphic to a $C^*$-subalgebra of $M_N(C(X))$ for some $N\in \N$ and compact Hausdorff space $X$.  Recall finally the notion of complexity rank relative to a class of $C^*$-algebras from Definition \ref{f c}.

\begin{proposition}\label{fdc prop}
Let $G$ be a groupoid with compact base space.
\begin{enumerate}[(i)]
\item \label{fdc part} The complexity rank of $C^*_r(G)$ relative to the class of subhomogeneous $C^*$-algebras is bounded above by the complexity rank of $G$.
\item \label{sfdc part} If $G$ is ample and principal, then the complexity rank of $C^*_r(G)$ (relative to the class of finite-dimensional $C^*$-algebras) is bounded above by the complexity rank of $G$.
\end{enumerate}
In particular, if $G$ is second countable and has finite dynamical complexity, then $C^*_r(G)$ satisfies the UCT.
\end{proposition}

Before getting into the proof of this, let us discuss some remarks and examples.  

\begin{example}\label{fdc exes}
Let $G(X)$ be the coarse groupoid associated to a bounded geometry metric space $X$: see \cite[Section 3]{Skandalis:2002ng} or \cite[Chapter 10]{Roe:2003rw} for background.  For such spaces $X$, Guentner, Tessera and Yu \cite{Guentner:2009tg} introduced a notion called \emph{finite decomposition complexity}; it comes with a natural complexity rank, defined to be the smallest ordinal $\alpha$ such that $X$ is in the class $\mathfrak{D}_\alpha$ of \cite[Definition 2.2.1]{Guentner:2013aa}.  Then \cite[Theorem A.7]{Guentner:2014bh} shows that $G(X)$ has finite dynamical complexity if and only if $X$ has finite decomposition complexity\footnote{This result was one of the key motivations for the definition of finite dynamical complexity, and also motivates the terminology.}; moreover, inspection of the proof shows that the two complexity ranks agree.  It follow from this and \cite[Theorem 4.1]{Guentner:2013aa} that for any $n\in \N$ there are spaces $X$ such that $G(X)$ is not in $\mathcal{C}_n$, but is in $\mathcal{C}_N$ for some finite $N>n$.  Moreover it follows from \cite[Discussion below 2.2.1]{Guentner:2013aa} or the main result of \cite{Chen:2015uc} that there are spaces $X$ such that $G(X)$ is in $\mathcal{C}_\alpha$ for some infinite $\alpha$, but not for any finite $\alpha$.  
\end{example}

Example \ref{fdc exes} shows that the range of possible values of the complexity rank for groupoids is quite rich.  As we do not know the corresponding fact for $C^*$-algebras, the following question is natural.

\begin{question}
Are there any circumstances when the complexity rank of $C^*_r(G)$ is bounded above by that of $G$?
\end{question}

It seems very unlikely that there is a positive answer in general, but it is conceivable that there could be a positive answer for coarse groupoids.

\begin{example}\label{fdc exes 2}
Transformation groupoids provide natural examples with finite complexity rank.  Using the main result of \cite{Amini:2020ue}, the complexity rank of the transformation groupoid associated to any free action of a virtually cyclic group on a finite-dimensional space is one.  We guess that the techniques used in the proof of \cite[Theorem 1.3]{Conley:2020ta} should show that for many discrete groups $\Gamma$, any free action on the Cantor set $X$ gives rise to a groupoid $X\rtimes \Gamma$ with finite dynamical complexity; however, we did try to look into the details, and would be interested in any progress here.  These ideas lead to the following conjecture.  
\end{example}

\begin{conjecture}
If $\Gamma$ has finite decomposition complexity then $X\rtimes \Gamma$ has finite dynamical complexity for \emph{any} free action of $\Gamma$ on the Cantor set.
\end{conjecture}

\begin{remark}\label{no new uct}
Proposition \ref{fdc prop} does not give new information on the UCT: this is because all groupoids with finite dynamical complexity are amenable by \cite[Theorem A.9]{Guentner:2014bh}, whence their groupoid $C^*$-algebras satisfy the UCT by Tu's theorem \cite[Proposition 10.7]{Tu:1999bq}.  However, it seems interesting to have an approach to the UCT for a large class of groupoids that does not factor through the Dirac-dual-Dirac machinery employed by Tu.
\end{remark}

We now turn to the proof of Proposition \ref{fdc prop}.  For a subgroupoid $H$ of a groupoid $G$, write $H':=H\cup G^{(0)}$, which is also a subgroupoid of $G$.

\begin{lemma}\label{add base lem}
Let $G$ be a groupoid with compact base space, and let $H$ be a subgroupoid in $\mathcal{C}_\alpha$.  Then $H\cup G^{(0)}$ is a subgroupoid of $G$ that is also in $\mathcal{C}_\alpha$.
\end{lemma}

\begin{proof}
We proceed by transfinite induction on $\alpha$.  For the base case $\alpha=0$, let $H$ be a subgroupoid of $G$ in $\mathcal{C}_0$, and let $K'$ be a compact subset of $H'$.  As the base space in an \'{e}tale groupoid is open, $K:=K'\setminus G^{(0)}$ is also a compact set, and is contained in $H$.  As $H$ is in $\mathcal{C}_0$, there exists a subgroupoid $L$ of $H$ that contains $K$, and that has compact closure.  Hence $L'$ is a subgroupoid of $H'$ that contains $K'$ and has compact closure.  Thus $H'$ is in $\mathcal{C}_0$ too.  The inductive step follows the same idea.
\end{proof}

The lemma below is very similar to \cite[Lemma B.3]{Willett:2019aa}.

\begin{lemma}\label{dec to dec}
Let $G$ be a groupoid with compact base space.  Let $H$ be a subgroupoid of $G$ that decomposes over some class $\mathcal{C}$ of subgroupoids of $G$.  Then $H'$ decomposes over the collection of subgroupoids $L'$, where $L$ is a subgroupoid of $H$ that is in $\mathcal{C}$.  
\end{lemma}

\begin{proof}
Let $X$ be a finite subset of the unit ball of $C^*_r(H')$, and $\epsilon>0$.  As $C_c(H)+C(G^{(0)})$ is dense in $C^*_r(H')$, perturbing $X$ slightly, we may assume that $X$ is contained in a subset of $C^*_r(H')$ of the form $C_c(K)+C(G^{(0)})$, where $K$ is an open and relatively compact subset of $H$.  The proof of \cite[Lemma B.3]{Willett:2019aa} gives us open subgroupoids $H_1$ and $H_2$ of $H$ and a positive contraction $h$ in $C_c(H_1^{(0)})\subseteq C^*_r(H_1)$ such that $H_1$, $H_2$ and $H_1\cap H_2$ are in the class $\mathcal{C}$, and such that for all $x\in X$, $hx\in C^*_r(H_1)$, $(1-h)x\in C^*_r(H_2)$, and $(1-h)hx\in C^*_r(H_1\cap H_2)$.  Then the data $h$, $C:=C^*_r(H_1')$, $D=C^*_r(H_2')$, and $E=C^*_r(H_1'\cap H_2')$ give the desired decomposability statement. 
\end{proof}

%



\begin{proof}[Proof of Proposition \ref{fdc prop}]
For part \eqref{fdc part}, fix a groupoid $G$.  We show by transfinite induction on $\alpha$ that if $H$ is an open subgroupoid of $G$ in the class $\mathcal{C}_\alpha$, and if $H'=H\cup G^{(0)}$, then $C^*_r(H')$ is in the class $\mathcal{D}_\alpha$ of Definition \ref{f c}, where we define $\mathcal{D}_\alpha$ relative to the class of subhomogeneous $C^*$-algebras.  Applying this to $H=G$ then gives the desired conclusion for $C^*_r(G)$.  

For the base case, we need to show that if $H$ is an open subgroupoid of $G$ in the class $\mathcal{C}_0$ and if $H'=H\cup G^{(0)}$, then $C^*_r(H')$ is locally subhomogeneous.   Let a finite subset $X$ of $C^*_r(H')$ and $\epsilon>0$ be given.  As $C_c(H')$ is dense in $C^*_r(H')$, up to a perturbation, we may assume $X$ is contained in $C_c(K)$ for some open and relatively compact subset $K$ of $H'$.  Lemma \ref{add base lem} implies that $H'$ is in $\mathcal{C}_0$, whence there is an open subgroupoid $L$ of $H'$ with compact closure that contains $K$, and therefore so that $X$ is contained in $C^*_r(L)$.   On the other hand, $C^*_r(L)$ is subhomogeneous by the proof \cite[Lemma 8.14]{Guentner:2014aa}, so we are done with the base case.

Assume now that $\alpha>0$ (and is either a successor ordinal or limit ordinal), and let $H$ be a subgroupoid of $G$ in the class $\mathcal{C}_\alpha$.  According to Lemma \ref{dec to dec}, we have that $H'$ decomposes over 
$$
\Bigg\{C^*_r(L')\mid L\text{ an open subgroupoid of $H'$ in } \bigcup_{\beta<\alpha}\mathcal{C}_\beta\Bigg\}.
$$
which completes the proof of part \eqref{fdc part} by inductive hypothesis.

We now look at part \eqref{sfdc part}, so let $G$ be principal and ample.  We will show that if $G$ is in $\mathcal{C}_0$, then $C^*_r(G)$ is locally finite dimensional; thanks to our work in part \eqref{fdc part}, this will suffice for the proof.  

Let then $G$ be an element of $\mathcal{C}_0$.  We claim that for any compact subset $K$ of $G$ there is a compact open subgroupoid of $H$ of $G$ that contains $K$.  The claim shows that $C^*_r(G)$ is locally finite-dimensional.  Indeed, up to a perturbation we can assume any finite subset of $C^*_r(G)$ is contained in $C_c(K)$ for some open and relatively compact subset $K$ of $G$, and so in $C^*_r(H)$ for some compact, open subgroupoid of $G$.  It is well-known that a compact, Hausdorff, \'{e}tale, principal groupoid with totally disconnected base space has a locally finite-dimensional $C^*$-algebra: for example, this follows directly from the structure theorem for ``CEERs'' in \cite[Lemma 3.4]{Giordano:2003aa}. 

To establish the claim, let a compact subset $K$ of $G$ be given.   According to the definition of $\mathcal{C}_0$ there exists an open subgroupoid $L$ of $G$ with compact closure such that $K$ is contained in $L$.  Note first that as $L$ has compact closure, there is some $m\in \N$ such that $L$ is covered by $m$ open bisections from $G$.  Hence in particular, for any $x\in L^{(0)}$, we have that the range fiber $L^x$ has at most $m$ elements.  Working entirely inside $L$, it suffices to prove that if $K$ is a compact subset of a principal, ample groupoid $L$ such that $\sup_{x\in L^{(0)}}|L^x|=m<\infty$, then there is a compact, open subgroupoid $H$ of $L$ that contains $K$.  

Now, as $L$ is ample (and \'{e}tale), each point $l\in K$ is contained in a compact, open subset of $L$.  As finitely many of these compact, open subsets cover $K$, there is a compact, open subset $K'$ of $L$ such that $K\subseteq K'$.  Let $H$ be the subgroupoid of $L$ generated by $K'$.  A subgroupoid generated by an open subset is always open (see for example \cite[Lemma 5.2]{Guentner:2014aa}), so it suffices to prove that $H$ is compact.  Let $(h_i)_{i\in I}$ be an arbitrary net consisting of elements from $H$.  Each $h_i$ can be written as a finite product $h_i=k_i^{(1)}\cdots k_i^{(n_i)}$, with $k_i^{(j)}$ in $K'':=K'\cup (K')^{-1}\cup s(K')\cup r(K')$.   As each range fibre from $L$ has at most $m$ elements, we may assume that $n_i\leq m$ for all $m$; in fact we may assume it is exactly $m$, as otherwise we can just ``pad'' it with identity elements.  Write then $h_i=k_i^{(1)}\cdots k_i^{(m)}$.  As $K''$ is compact, we may pass to a subnet of $I$, and thus assume that each net $(k_i^{(j)})_{i\in I}$ has a convergent subnet, converging to some $k^{(j)}$ in $K''$.  It follows on passing to this subnet that $(h_i)$ converges to $k^{(1)}\cdots k^{(m)}$.  As we have shown that every net in $H$ has a convergent subnet, $H$ is compact, completing the proof.
\end{proof}

\bibliography{Generalbib}

\end{document}